\documentclass[]{amsart}
\usepackage{mathrsfs}
\usepackage{amsmath}
\usepackage{amssymb}
\usepackage{exscale}
\usepackage{relsize}
\usepackage{url}

\newtheorem{Theorem}{Theorem}[section]
\newtheorem{Corollary}[Theorem]{Corollary}
\newtheorem{Lemma}[Theorem]{Lemma}
\newtheorem{Proposition}[Theorem]{Proposition}

\newtheorem{Definition}[Theorem]{Definition}

\newtheorem{Conjecture}[Theorem]{Conjecture}

\newtheorem{Remark}[Theorem]{Remark}

\numberwithin{equation}{section}

\begin{document}
	\title
	{weighted versions of Saitoh's conjecture in fibration cases}
	
	\author{Qi'an Guan}
	\address{Qi'an Guan: School of Mathematical Sciences,
		Peking University, Beijing, 100871, China.}
	\email{guanqian@math.pku.edu.cn}
	
	\author{Gan Li}
	\address{Gan Li: School of Mathematical Sciences,
		Peking University, Beijing, 100871, China.}
	\email{2301110004@stu.pku.edu.cn}
	
	\author{Zheng Yuan}
	\address{Zheng Yuan: Institute of Mathematics, Academy of Mathematics and Systems
		Science, Chinese Academy of Sciences, Beijing 100190, China.}
	\email{yuanzheng@amss.ac.cn}
	
	\thanks{}
	
	
	
	\date{\today}
	
	\dedicatory{}
	
	\commby{}
	
	\maketitle
	\begin{abstract}
		In this article, we introduce some generalized Hardy spaces on fibrations of planar domains and fibrations of products of planar domains. We  consider the kernel functions on these spaces, and we prove some weighted versions of Saitoh's conjecture in fibration cases.  
	\end{abstract}
	
	\tableofcontents
	
	
	\section{Introduction}
	Let $D$ be a planar regular region with finite boundary components which are analytic Jordan curves (see \cite{saitoh,yamada}). Let us recall the definition of Hardy space $H^2(D)$ (see \cite{saitoh,rudin2}):  
	A holomorphic function $f\in\mathcal{O}(D)$ belongs to the $H^2(D)$  if there exists a harmonic majorant $U(z)$ such that
	$$|f(z)|^2\le U(z)\,\, \text{on}\,\,D.$$
	Each function $f(z)\in H^2(D)$ has Fatou's nontangential boundary value a.e. on $\partial D$  belonging to $L^2(\partial D)$ (see \cite{rudin2,duren}). Some properties about $H^2(D)$ can be seen in \cite{rudin2}. 
	
	Let $G_D(z,t)$ be the Green function on $D$. For fixed $t\in D$, $\frac{\partial G_D(z,t)}{\partial v_z}$ is  positive and continuous on $\partial D$ because of the analyticity of the boundary (see \cite{saitoh,guan-19saitoh}), where
	$\frac{\partial}{\partial v_z}$ denotes the derivative along the outer normal unit vector $v_z$. The conjugate Hardy $H^2$ kernel $\hat K_t(z,\overline w)$ (see \cite{nehari}) on $D$ was defined as 
	$$f(w)=\frac{1}{2\pi}\int_{\partial D}f(z)\overline{\hat K_t(z,\overline w)}\left(\frac{\partial G_D(z,t)}{\partial v_z}\right)^{-1}|dz|$$
	for any $f\in H^2(D)$.  
	When $t=w=z$, $\hat K(z)$ denotes $\hat K_t(z,\overline w)$ for simplicity. 
	
	In \cite{yamada}, Yamada listed the Suita's conjecture (the left part of inequality \eqref{eq:240909a}, see \cite{suita72}) and the Saitoh's conjecture (the left part of inequality \eqref{eq:240909a}, see \cite{saitoh}) as follows:
	\begin{Conjecture}
		If $D$ is not simple connected, then 
		\begin{equation}
			\label{eq:240909a}
			(c_\beta(z))^2<\pi B(z)< \hat K(z),
		\end{equation}
		where $c_{\beta}(z)$ is the logarithmic capacity (see \cite{S-O69}) and $B(z)$ is the Bergman kernel on $D$.
	\end{Conjecture}
	
In \cite{OT87}, Ohsawa and Takegoshi established the famous  $L^2$ extension theorem. After that, Ohsawa \cite{OhsawaObservation} observed that the $L^2$ extension theorem with an optimal estimate can prove the ``$\le$" part of  Suita's conjecture. Thus, to finding the optimal constant in the $L^2$ extension theorem is an interesting and important question.

Guan-Zhou-Zhu \cite{ZGZ} (see also \cite{GZZCRMATH}) firstly used a method of undetermined  function to study the optimal constant in the $L^2$ extension theorem. In \cite{Blocki-12}, Blocki used the same method and the
same undetermined  function  to obtain the same constant as G-Z-Z’s. Developing the equation of undetermined  function, Blocki \cite{Blocki-inv} proved an optimal $L^2$ extension theorem on bounded pseudoconvex domains in $\mathbb{C}^n$, which solved the ``$\le$" part of  Suita's conjecture on planer regions.
 Continuing the work \cite{GZZCRMATH,ZGZ}, Guan-Zhou \cite{guan-zhou CRMATH2012,GZsci} obtained the optimal estimate version
 of Ohsawa’s $L^2$ extension theorem with negligible weight, which proved ``$\le$" part of  Suita's conjecture on open Riemann surfaces (see \cite{suita72}).
 In \cite{guan-zhou13ap}, Guan-Zhou established an optimal $L^2$ extension theorem in a general setting, and used it to prove a necessary and sufficient condition of the holding of $(c_\beta(z))^2=\pi B(z)$ on open Riemann surfaces, which finished the proof of  Suita's conjecture.
	 
Using the concavity property for minimal $L^2$ integrals (see \cite{G16}) and the solution of Suita's conjecture, Guan \cite{guan-19saitoh} proved the Saitoh's conjecture:
	
	\begin{Theorem}
		[\cite{guan-19saitoh}]\label{thm:saitoh}
		If $D$ is not simple connected, then $\hat K(z)>\pi B(z)$.
	\end{Theorem}

In \cite{yamada}, Yamada posed a harmonic weight version of Suita's conjecture (called extended Suita conjecture), which was proved by Guan-Zhou in \cite{guan-zhou13ap}. In \cite{GM}, Guan-Mi proved a subharmonic weight version of  Suita's conjecture, and the the case of weights may not be subharmonic was proved by Guan-Yuan \cite{GY-concavity}. 
In \cite{GMY-concavity2}, Guan-Mi-Yuan  proved a weighted version of  Suita's conjecture for higher derivatives (some related results can be referred to \cite{LXZ,x-z}). For products of open Riemann surfaces, Guan-Yuan \cite{GY-conc4} proved a product version of weighted Suita's conjecture. Bao-Guan-Yuan proved a weighted Suita's conjecture on fibrations over open Riemann surfaces  in \cite{BGY-concavity5}, and proved a weighted Suita's conjecture on fibrations over  products of open Riemann surfaces in \cite{BGY-concavity6} (the negligible weights case can be seen in  \cite{BGMY-concavity7}).

	Recently, Guan-Yuan proved a weighted version of Saitoh's conjecture \cite{GY-weight} and some product versions of Saitoh's conjecture \cite{GY-Hardy and product}.  In this article, we study the weighted saitoh's conjecture in fibration cases.

	Let us recall the weighted version of Saitoh's conjecture in the following, and the product versions can be seen in Section \ref{sec:1.2} and \ref{sec:1.3}.

	Let  $z_0\in D$. Let $\psi$ be a Lebesgue measurable function on  $\overline D$, which satisfies that $\psi$ is  subharmonic on $D$, $\psi|_{\partial D}\equiv 0$
	and the Lelong number $v(dd^c\psi,z_0)>0$, where $d^c=\frac{\partial-\bar\partial}{2\pi\sqrt{-1}}$.
	Assume that $\psi\in C^1(U\cap\overline{D})$ for an open neighborhood $U$ of $\partial D$ and $\frac{\partial\psi}{\partial v_z}$ is positive on $\partial D$, where $\partial/\partial v_z$ denotes the derivative along the outer normal unit vector $v_z$. Let $\varphi$ be a Lebesgue measurable function on $\overline D$ which is continuous on $\partial D$ and satisfies that $\varphi+2\psi$ is subharmonic on $D$, and the Lelong number 
	$$v(dd^c(\varphi+2\psi),z_0)\ge2.$$ 
	Assume that  one of the following two statements holds:
	
	$(a)$ $(\psi-p_0G_{D}(\cdot,z_0))(z_0)>-\infty$, where $p_0=v(dd^c(\psi),z_0)>0$;
	
	$(b)$ $\varphi+2a\psi$ is subharmonic near $z_0$ for some $a\in[0,1)$.
	
	Let $c$ be a positive Lebesgue measurable function on $[0,+\infty)$ satisfying that $c(t)e^{-t}$ is decreasing on $[0,+\infty)$, $\lim_{t\rightarrow0+0}c(t)=c(0)=1$ and $\int_0^{+\infty}c(t)e^{-t}dt<+\infty$.
	
	Denote  
	$$\rho:=e^{-\varphi}c(-2\psi)\,\,\text{and}\,\,K_{\rho,\psi}(z):=K_{\rho\left(\frac{\partial \psi}{\partial v_z}\right)^{-1}}(z,\overline z).$$ 
	Assume that $\rho$ has a positive lower bound on any compact subset of $D\backslash Z$, where $Z\subset\{\psi=-\infty\}$ is a discrete subset of $D$. Let $B_{\rho}(z)$ be the weighted Bergman kernel on $D$ with the weight $\rho$ (see \cite{pasternak} or Section \ref{sec:admissible}).
	
	Let  $P:\Delta\rightarrow D$ be the universal covering from unit disc $\Delta$ to $D$.  We call a holomorphic function $f$ on $\Delta$ a multiplicative function (see \cite{chara}, also see \cite{GY-concavity}),
	if there is a representation of the fundamental group of $D$, whose character $\chi$ satisfies that $g^{*}f=\chi(g)f$ and $|\chi|=1$, where $g$ is an element of the fundamental group of $D$. Denote the set of such kinds of $f$ by $\mathcal{O}^{\chi}(D)$. 
	
	It is well-known that for any harmonic function $u$ on $ D$,
	there exists a character $\chi_{u}$ and a multiplicative function $f_u\in\mathcal{O}^{\chi_{u}}( D)$,
	such that $|f_u|=P^{*}\left(e^{u}\right)$. For Green function $G_D(\cdot,z_0)$, there exists a character $\chi_{z_0}$ and a multiplicative function $f_{z_0}\in\mathcal{O}^{\chi_{z_0}}( D)$,
	such that $|f_{z_0}|=P^{*}\left(e^{G_D(\cdot ,z_0)}\right)$.
	
	The weighted version of Saitoh's conjecture gives a relation between $K_{\rho,\psi}(z_0)$ and $B_{\rho}(z_0).$
	\begin{Theorem}
		[\cite{GY-weight}]
		\label{main theorem}
		\begin{equation}\nonumber
			K_{\rho,\psi}(z_0)\ge \left(\int_0^{+\infty}c(t)e^{-t}dt\right)\pi B_{\rho}(z_0)
		\end{equation}
		holds, and the equality holds if and only if the following statements holds:
		
		$(1)$ $\varphi+2\psi=2G_{D}(\cdot,z_0)+2u$, where $u$ is a harmonic function on $D$;
		
		$(2)$ $\psi=p_0G_{D}(\cdot,z_0)$, where $p_0=v(dd^c(\psi),z_0)>0$;
		
		$(3)$ $\chi_{z_0}=\chi_{-u}$, where $\chi_{-u}$ and $\chi_{z_0}$ are the  characters associated to the functions $-u$ and $G_{D}(\cdot,z_0)$ respectively.
	\end{Theorem}

	In this article, we consider the fibration cases over planar domains and fibration cases over products of planar domains. We generalize the Hardy space, and prove some weighted versions of Saitoh's conjecture in fibration cases.

	\subsection{Space $H^2_\theta(D\times U, \partial D\times U)$}\label{sec:main}
	Let  $D$ be a planar regular region with  finite  boundary components,  which are analytic Jordan curves. 
	Let $\lambda({z})$ be a positive and continuous function on $\partial D$, and $\gamma(u)$  be  an admissible weight (the definition can be seen in Section \ref{sec:admissible}) on $U$, where $U$ is a domain in $\mathbb{C}^m$.  Denote $$\theta(z,u):=\lambda({z})\gamma(u)$$ on $\partial D\times U$.
	
	We give a generalization of the Hardy space on $D\times U$.
	\begin{Definition}
		\label{D1}
		Let $\mathscr{F}$ denote the space defined as the set of all  holomorphic functions $f(z,u)$ on $D\times U$, which satisfy $f(\cdot,u)\in H^2(D)$ for any fixed $u\in U$.
		Let $H^2_\theta(D\times U, \partial D\times U)$ denote the space defined as the set of all  holomorphic functions $f(z,u)\in\mathscr{F}$  with finite norms
		\begin{equation}\nonumber
			\parallel f\parallel_{\partial D\times U,\theta}^2=\frac{1}{2\pi}\int_{\partial D\times U}|f(z,u)|^2\theta(z,u)|dz|dV_{U},
		\end{equation} 
		where  $dV_U$ is the Lebesgue measure on U.     
	\end{Definition}

	We define a weighted  kernel function $K_{\partial D\times U,\theta}((\zeta,u),(\overline{z},\overline{w}))$ (see Section 2.3) as follows: for every $(z,w)\in D\times U$, there exists $K_{\partial D\times U,\theta}((\zeta,u),(\overline{z},\overline{w}))\in H^2_\theta(D\times U, \partial D\times U)$  such that  
	\begin{equation}\nonumber
		f(z,w)=\frac{1}{2\pi}\int_{\partial D\times U}f(\zeta,u)\overline{K_{\partial D\times U,\theta}((\zeta,u),(\overline{z},\overline{w}))}\theta(\zeta,u)|d\zeta|dV_U 
	\end{equation}
	holds for $f\in H^2_\theta(D\times U, \partial D\times U)$.
	
	Let $B_{D\times U,\eta}((\zeta,u),(\overline{z},\overline{w}))$ denote the weighted Bergman kernel on $D\times U$ with an admissible weight (see Section \ref{sec:admissible}) $\eta(\zeta,u)$, where $(\zeta,u),(z,w)\in D\times U$.
	
	We present a weighted version of Saitoh's conjecture in the fibration case as follows:
	\begin{Theorem}\label{main1-theorem}
		Assume that $B_{D\times U,\eta}((z_0,u),(\overline{z}_0,\overline{u}))>0.$ Then
		\begin{equation}\label{2:E39}    
			\left(\int_0^{+\infty}c(t)e^{-t}dt\right)\pi B_{D\times U,\eta}((z_0,u),(\overline{z}_0,\overline{u}))\leq K_{\partial D\times U,\theta}((z_0,u),(\overline{z}_0,\overline{u}))
		\end{equation} 
		holds for $(z_0,u)\in D\times U$, where $\eta=\rho\gamma$, $\theta=\lambda\gamma$, $\lambda=\rho\left(\frac{\partial \psi}{\partial v_p}\right)^{-1}$ and $\psi,\rho$ are the same as those appearing in  Theorem \ref{main theorem}. 
		
		Furthermore, inequality \eqref{2:E39} becomes an equality if and only if the following statements hold:
		
		$(1)$ $\varphi+2\psi=2G_{D}(\cdot,z_0)+2u$, where $u$ is a harmonic function on $D$;
		
		$(2)$ $\psi=p_0G_{D}(\cdot,z_0)$, where $p_0=v(dd^c(\psi),z_0)>0$;
		
		$(3)$ $\chi_{z_0}=\chi_{-u}$, where $\chi_{-u}$ and $\chi_{z_0}$ are the  characters associated to the functions $-u$ and $G_{D}(\cdot,z_0)$ respectively. 
	\end{Theorem}

	\subsection{Space $H^2_\kappa(M\times U, \partial M\times U)$}\label{sec:1.2} In this section, we discuss the product case of planar domains and fibration cases over products of planar domains.
	
	Let $D_j$ ($1\le j\le n$) be a planar regular region with finite boundary components which are analytic Jordan curves (see \cite{saitoh}, \cite{yamada}). Denote $M :=\prod_{j=1}^{n}D_j$. 
	Firstly, we recall the Hardy space $H^2_\rho(M,\partial M)$ (see \cite{GY-Hardy and product}).
	
	Let $M_j=\prod_{1\leq l\leq n,l\neq j}D_l.$ Then $M=D_j\times M_j$ and  $\partial M=\bigcup_{j=1}^n\partial D_j\times \overline{M_j}.$ Let $\rho$ be a positive Lebesgue measurable function on $\partial M$ satisfying that $\inf_{\partial M}\rho>0$. Assume that $z_j\in D_j$ for any $1\leq j\leq n$. Recall that $H^2(D_j)$ denotes the Hardy space on $D_j$ and there exists a  linear map 
	$$\gamma_j: H^2(D_j)\rightarrow L^2(\partial D_j)$$ and  $\gamma_j(f)$ denotes the nontangential boundary values of $f$ a.e. on $\partial D_j$
	for any $f\in H^2(D_j)$.    
	
	Let $d\mu_j$ denote the Lebesgue measure on $M_j$ for $j=1,2,\cdots, n$, and let $d\mu$ denote the Lebesgue measure on $\partial M$. Then we have 
	$$\int_{\partial M}hd\mu=\sum\limits_{1\leq j\leq n}\frac{1}{2\pi}\int_{M_j}\int_{\partial D_j}h(w_j,\hat{w}_j)|dw_j|d\mu_j(\hat{w}_j)$$
	for every $h\in L^1(\partial M).$  For simplicity, denote $d\mu\big|_{\partial D_j\times M_j}$ also by $d\mu$. Denote
	\[
	\begin{split}
		\left\{
		f\in L^2(\partial D_j\times M_j,\rho d\mu):\exists f^*\in \mathcal{O}(M),\; \text{s.t.}\; f^*(\cdot,\hat{w_j})\in H^2(D_j)\; \mbox{for any} \; \hat{w}_j\in M_j \right.\\\&\left. f=\gamma_j(f^*)\;\mbox{a.e.}  \;\mbox{on}\; \partial D_j\times M_j
		\right\}
	\end{split}    
	\]
	by $H_\rho^2(M,\partial D_j\times M_j)$, which is a Hilbert space with the inner product $\ll\cdot,\cdot\gg_{\partial D_j\times M_j,\rho}$
	(see \cite{GY-Hardy and product}) defined as follows:
	$$\ll f,g\gg_{\partial D_j\times M_j,\rho}: =\frac{1}{2\pi}\int_{M_j\times \partial D_j}f\overline{g}\rho(w_j,\hat{w}_j)d\mu_j(\hat{w}_j)|dw_j|$$
	for any $f,g\in H_\rho^2(M,\partial D_j\times M_j)$. For each $f\in H_\rho^2(M,\partial D_j\times M_j)$,  the norm is defined by
	$$||f||^2_{\partial D_j\times M_j,\rho}:=\frac{1}{2\pi}\int_{\partial D_j\times M_j }|f(w_j,\hat{w}_j)|^2\rho(w_j,\hat{w}_j)|dw_j|d\mu_j(\hat{w}_j).$$ 
	
	In \cite{GY-Hardy and product}, the following space  $H^2_\rho(M,\partial M)$  was introduced.
	\begin{Definition}[\cite{GY-Hardy and product}]
		Let $f\in L^2(\partial M,\rho d\mu)$. We call $f\in H^2_\rho(M,\partial M)$ if there exists $f^*\in\mathcal{O}(M)$ such that for $1\leq j\leq n$, $f^*(\cdot,\hat{w}_j)\in H^2(D_j)$ for any $\hat{w}_j\in M_j$ and $f=\gamma_j(f^*)\; a.e.$ on $\partial D_j\times M_j$.
	\end{Definition}
	There exists   a linear injective map $P_{\partial M}$ from $H^2_\rho(M,\partial M)$ to $\mathcal{O}(M)$ (see \cite{GY-Hardy and product}) satisfying that $P_{\partial M}(f)=f$ for any $f\in \mathcal{O}(M)\cap C(\overline{D})\cap L^2(\partial M,\rho)$. Denote  $f^*:=P_{\partial M}(f)$ for any $f\in H^2_\rho(M,\partial M)$.
	$H^2_\rho(M,\partial M)$ is a Hilbert space equipped with the inner product $\ll\cdot,\cdot\gg_{\partial M,\rho}$ (see \cite{GY-Hardy and product}), defined by
	$$\ll f,g\gg_{\partial M,\rho} :=\int_{\partial M}f\overline{g}\rho d\mu.$$
	And for each $f\in H^2_\rho(M,\partial M)$, denote its norm by
	$$||f||^2_{\partial M,\rho}=\int_{\partial M}|f|^2\rho d\mu.$$
	When $\rho\equiv 1$, $H^2(M,\partial M)$ denotes $H^2_\rho(M,\partial M).$

	We define a kernel function $K_{\partial M,\rho}(z,\overline{w})$ (see \cite{GY-Hardy and product}) as follows:
	$$K_{\partial M,\rho}(z,\overline{w}) :=\sum\limits_{m=1}^{+\infty}P_{\partial M}(e_m)(z)\overline{P_{\partial M}(e_m)(w)}$$
	for $(z,w)\in M\times M\subset \mathbb{C}^{2n}$, where $\left\{e_m\right\}_{m\in\mathbb{Z}_{\geq 1}}$ is a complete orthonormal basis of $H^2_\rho(M,\partial M)$. The definition of $K_{\partial M,\rho}(z,\overline{w})$ is independent of the choice of $\left\{e_m\right\}_{m\in\mathbb{Z}_{\geq 1}}$ (see \cite{GY-Hardy and product}). Denote  $K_{\partial M,\rho}(z) :=K_{\partial M,\rho}(z,\overline{z})$.
	
	Let $z_0=(z_1,\cdots,z_n)\in M$, and 
	$$\psi(w_1,\cdots,w_n) = \max\limits_{1\leq j\leq n}\left\{2p_jG_{D_j}(w_j,z_j)\right\},$$
	which is a plurisubharmonic function on $M$, where $G_{D_j}(\cdot,z_j)$ is the Green function on $D_j$ and $p_j>0$ is a constant for any $1\leq j\leq n$.
	
	Let $\varphi_j$ be a function on $\overline{D}_j$,  which is subharmonic on $D_j$, and continuous at $w_j$ for all $w_j\in \partial D_j$. Now $\rho$ is given as follows: 
	$$\rho(w_1,\cdots,w_n)|_{\partial D_j\times \overline{M_j}} \;:=\frac{1}{p_j}\bigg(\frac{\partial G_{D_j}(w_j,z_j)}{\partial\nu_{w_j}}\bigg)^{-1}\times\prod\limits_{1\leq l\leq n}e^{-\varphi_l(w_l)},$$
	where $\partial/\partial\nu_{w_j}$ denotes the derivative along the outer normal unit $\nu_{w_j}$. Let $c$ be a positive function on $[0,+\infty)$, which satisfies that $c(t)e^{-t}$ is decreasing on $[0,+\infty)$, $\lim\limits_{t\rightarrow 0+0}c(t)=c(0)=1$ and $\int_{0}^{+\infty}c(t)e^{-t}dt<+\infty.$ 
	
	Take 
	$$\tilde{\rho} :=c(-\psi)\prod\limits_{1\leq j\leq n}e^{-\varphi_j}$$
	on $M$, which is an admissible weight (see Section \ref{sec:admissible}). 
	Denote that $B_{\tilde{\rho}}(z)$ is the Bergman kernel on $M$ with the weight $\tilde{\rho}$.

	Let us recall a product version of Saitoh's conjecture \cite{GY-Hardy and product}.

	\begin{Theorem}[\cite{GY-Hardy and product}]
		\label{key-Theorem}
		Assume that $B_{\tilde{\rho}}(z_0)>0$ and $n>1$. Then
		\begin{equation}\nonumber
			K_{\partial M,\rho}(z_0)\geq\bigg(\int_{0}^{+\infty}c(t)e^{-t}dt\bigg)\pi B_{\tilde{\rho}}(z_0)
		\end{equation}
		
		holds, and the equality holds if and only if the following statements hold:
		
		$(1)$ $\sum_{j=1}^{n}\frac{1}{p_j}=1;$
		
		$(2)$ $\varphi_j=2\log|g_j|+2u_j$ on $D_j$ for $1\leq j\leq n$, where $u_j$ a is harmonic function on $D_j$ and $g_j$ is a holomorphic function on $\mathbb{C}$ such that $g_j(z_j)\neq 0;$
		
		$(3)$ $\chi_{j,z_j}=\chi_{j,-u_j}$, where $\chi_{j,-u_j}$ and $\chi_{j,z_j}$ are the characters associated to the functions $-u_j$ and $G_{D_j}(\cdot,z_j)$ respectively.
	\end{Theorem}

	Now, we study the fibration case.
	
	Let $U$ be a domain in $\mathbb{C}^m$.
	Let $\gamma(u)$  be  an admissible weight on $U$ (the definition can be seen in Section \ref{sec:admissible}). 
	
	\begin{Definition}
		Let $H^2_\kappa(M\times U, \partial M\times U)$ denote the space defined as the set of all holomorphic functions $f(z,u)$ on $M\times U$, which satisfy that for every $u\in U$, we have  $P_{\partial M}^{-1}(f(\cdot,u))\in H_\rho^2(M,\partial M)$,  with finite norms
		\begin{equation*}
			\begin{split}
				\parallel f\parallel^2_{\partial M\times U,\kappa}:&=\int_{\partial M\times U}|P^{-1}_{\partial M}(f(w,u))|^2\kappa(w,u)d\mu(w)dV_U\\
				&=\int_U\gamma(u)dV_U\int_{\partial M}|P^{-1}_{\partial M}(f(w,u))|^2\rho(w)d\mu(w)\\
				&=\int_U\gamma(u)dV_U\sum\limits_{1\leq j\leq n}\frac{1}{2\pi}\int_{M_j}\int_{\partial D_j}|\gamma_j(f(w_j,\hat{w}_j,u))|^2\rho(w_j,\hat{w}_j)|dw_j|d\mu_j(\hat{w}_j),
			\end{split}
		\end{equation*}
		where $\kappa(z,u)=\rho(z)\gamma(u)$, and  $\rho(z)$ is a positive Lebesgue measurable function on $\partial M$ satisfying $\inf_{\partial M}\rho>0$.         
	\end{Definition}

	In  Section  \ref{sec:3}, we will prove that $H^2_\kappa(M\times U, \partial M\times U)$ is a Hilbert space and each evaluation functional is bounded, so we have the following reproducing kernel:
	for every $(z,w)\in M\times U$, there exists $K_{\partial M\times U,\kappa}((\zeta,u),(\overline{z},\overline{w}))\in H^2_\kappa(M\times U, \partial M\times U) $ such that
	$$f(z,w)=\int_{\partial M\times U}P_{\partial M}^{-1}(f(\zeta,u))\overline{P_{\partial M}^{-1}(K_{\partial M\times U,\kappa}((\zeta,u),(\overline{z},\overline{w})))}\kappa(\zeta,u)d\mu dV_U$$
	for any $f \in H^2_\kappa(M\times U, \partial M\times U)$.

	Let $\eta=\tilde{\rho}\gamma$ and $\kappa=\rho\gamma$, where $\rho$ and $\tilde{\rho}$ are the same as those appearing in  Theorem \ref{key-Theorem}. Let $B_{M\times U,\eta}$ denote the weighted Bergman kernel on $M\times U$ with the weight $\eta(z)$.
	We present a weighted version of Saitoh's conjecture on products of planar domains in the fibration case as follows:
	\begin{Theorem}\label{main2}
		Assume that $B_{M\times U,\eta}((z_0,u_0),(\overline{z}_0,\overline{u}_0))>0$ and $n>1$. Then
		\begin{equation}\nonumber
			\bigg(\int_{0}^{+\infty}c(t)e^{-t}dt\bigg)\pi B_{M\times U,\eta}((z_0,u_0),(\overline{z}_0,\overline{u}_0))\leq K_{\partial M\times U,\kappa}((z_0,u_0),(\overline{z}_0,\overline{u}_0))
		\end{equation}
		holds for $(z_0,u_0)\in M\times U$, and
		the  equality holds if and only if the following statements hold: 
		
		$(1)$ $\sum_{j=1}^{n}\frac{1}{p_j}=1;$
		
		$(2)$ $\varphi_j=2\log|g_j|+2u_j$ on $D_j$ for $1\leq j\leq n$, where $u_j$ a is harnonic function on $D_j$ and $g_j$ is a holomorphic function on $\mathbb{C}$ such that $g_j(z_j)\neq 0;$
		
		$(3)$ $\chi_{j,z_j}=\chi_{j,-u_j}$, where $\chi_{j,-u_j}$ and $\chi_{j,z_j}$ are the characters associated to the functions $-u_j$ and $G_{D_j}(\cdot,z_j)$ respectively.     
	\end{Theorem}

	Let $I$ be an ideal of $\mathcal{O}_{z_0}$ such that $\mathcal{I}(\psi)_{z_0}\subset I\neq \mathcal{O}_{z_0}$, where $\mathcal{I}(\psi)$ is the multiplier ideal sheaf  (see \cite{demailly2010}), which is the sheaf of germs of holomorphic functions $h$ such that $|h|^2e^{-\psi}$ is locally integrable. Let $h_0$ be a holomorphic function on a neighborhood $V_0$ of $z_0$ and assume that $(h_0,z_0)\notin I$. Denote 
	$$K_{\partial M,\rho}^{I,h_0}(z_0) :=\frac{1}{\inf\left\{||f||^2_{\partial M,\rho} :f\in H^2_\rho(M,\partial M)\;\& \;(P_{\partial M}(f)-h_0,z_0)\in I\right\}}$$
	and
	$$B_{\tilde{\rho}}^{I,h_0}(z_0) :=\frac{1}{\inf\left\{\int_M|f|^2\tilde{\rho}: f\in\mathcal{O}(M)\;\&\; (f-h_0,z_0)\in I\right\}},$$ 
	which are generalizations of $K_{\partial M,\rho}(z_0)$ and the Bergman kernel $B_{\tilde{\rho}}(z_0)$   respectively.
	When $I$ takes the maximal ideal of $\mathcal{O}_{z_0}$ and $h_0(z_0)=1$, we have $K_{\partial M,\rho}^{I,h_0}(z_0)=K_{\partial M,\rho}(z_0)$ (see section 2.1 of \cite{GY-Hardy and product}). 
	
	There exists a unique $f_0\in H^2_\rho(M,\partial M)$ (see \cite{GY-Hardy and product}) such that $(P_{\partial M}(f_0)-h_0,z_0)\in I$ and
	$$K_{\partial M,\rho}^{I,h_0}(z_0)=\frac{1}{||f_0||^2_{\partial M,\rho}}.$$
	Let $E=\left\{(\alpha_1,\cdots,\alpha_n):\sum_{j=1}^{n}\frac{\alpha_j+1}{p_j}=1\;\&\;\alpha_j\in \mathbb{Z}_{\geq 0}\right\}$.
	
	In \cite{GY-Hardy and product}, Guan-Yuan gave a generalization of Theorem  \ref{key-Theorem} as follows:
	\begin{Theorem}[\cite{GY-Hardy and product}]\label{key-Theorem2}
		Assume that $B_{\tilde{\rho}}^{I,h_0}(z_0)>0$. Then
		\begin{equation}\nonumber
			K_{\partial M,\rho}^{I,h_0}(z_0)\geq\left(\int_{0}^{+\infty}c(t)e^{-t}dt\right)\pi B_{\tilde{\rho}}^{I,h_0}(z_0)
		\end{equation}
		holds, where $\rho,\tilde{\rho}$ are the same as those appearing in Theorem \ref{key-Theorem}.
		Furthermore, the equality holds if and only if the following statements hold:
		
		$(1)$ $P_{\partial M}(f_0)=\sum_{\alpha\in E}d_\alpha w^\alpha +g_0$ near $z_0$ for any $t\geq 0$, where $d_\alpha\in\mathbb{C}$ such that $\sum_{\alpha\in E}|d_\alpha|\neq 0$ and $g_0$ is a holomorphic function near $z_0$ such that $(g_0,z_0)\in \mathcal{I}(\psi)_{z_0}$;
		
		$(2)$ $\varphi_j=2\log|g_j|+2u_j$ on $D_j$ for any $1\leq j\leq n$, and $g_1\equiv 1$ when $n=1$, where $u_j$ is a harmonic function on $D_j$ and $g_j$ is a holomorphic function on $\mathbb{C}$ such that $g_j(z_j)\neq 0$;
		
		$(3)$ $\chi_{j,z_j}^{\alpha_j+1}=\chi_{j,-u_j}$ for any $1\leq j\leq n$ and $\alpha\in E$ satisfying $d_\alpha\neq 0$, where $\chi_{j,z_j}$ and $\chi_{j,-u_j}$ are the characters associated to functions $G_{\Omega_j}(\cdot,z_j)$ and $-u_j$ respectively.      
	\end{Theorem}

	Let $w_0=(z_0,u_0)\in M\times U$.
	Let $b_0$ be a holomorphic function  on a neighborhood of $u_0$, $l_0=\prod_{1\leq j\leq n}l_j$, where
	$l_j$ is a holomorphic function on a neighborhood of $z_j\in D_j$, and let $\hat{h}_0=l_0b_0$.   
	
	Denote  $|\alpha|:=\sum_{i=1}^{n}\alpha_{i}$ for $\alpha=(\alpha_1,\cdots,\alpha_n)\in\mathbb{N}^n$.   
	For any $\alpha=(\alpha_1,\cdots,\alpha_n)\in\mathbb{N}^n$ and any $\beta=(\beta_1,\cdots,\beta_n)\in\mathbb{N}^n,$ 
	we say $\alpha <\beta$   if $|\alpha|< |\beta|$ or $|\alpha|=|\beta|$ and there exists some $k$, $1\leq k\leq n$, such that $\alpha_1=\beta_1,\cdots,\alpha_{k-1}= \beta_{k-1},\;\alpha_k<\beta_k $. For any holomorphic $f\in\mathcal{O}(M)$ ($f\not\equiv0$), 
	denote 
	$$ord_{z_0}f:=\min\{\alpha\in\mathbb{N}^n:f^{(\alpha)}(z_0)\neq 0\},$$
	where $f^{(\alpha)}(z_0):=\frac{\partial^{|\alpha|}f}{(\partial w_1)^{\alpha_{1}}\cdots(\partial w_n)^{\alpha_n}}(z_0)$.
	
	Let  $\beta=(\beta^\prime,\beta^{\prime\prime})$, where $\beta^\prime=(\beta_1,\cdots,\beta_{n})$, $\beta_j= ord_{z_j}(l_j)$ for $1\leq j\leq n$ and $\beta^{\prime\prime}=(\beta_{n+1},\cdots,\beta_{n+m})=ord_{u_0}b_0$. 
	Let $\tilde{\beta}^\prime=(\tilde{\beta_1},\cdots,\tilde{\beta_n})\in\mathbb{N}^n$ and $\tilde{\beta}^{\prime\prime}=(\tilde{\beta}_{n+1},\cdots,\tilde{\beta}_{n+m})\in\mathbb{N}^{m}$, which satisfy that $\tilde{\beta_j}\geq\beta_j$ for any $1\leq j\leq n$ and $\beta^{\prime\prime}\leq \tilde{\beta}^{\prime\prime}$.      Denote 
	$$I_1:=\left\{(g,z_0)\in \mathcal{O}_{M,z_0}: g=\sum_{\alpha\in \mathbb{N}^{n}}b_\alpha(w-z_0)^\alpha \mbox{ near }z_0 \mbox{ s.t. } b_\alpha=0 \mbox{ for }\alpha\in L_{\tilde{\beta}^\prime}\right\},$$ where $L_{\tilde{\beta}^\prime}=\left\{\alpha=(\alpha_{1},\cdots,\alpha_{n})\in \mathbb{N}^{n}:\alpha_j\leq \tilde{\beta}_j\;\mbox{for}\;1\leq j\leq n \right\}$. To make sure that $I_1\supset\mathcal{I}(\psi)_{z_0}$, assume that  $\sum_{j=1}^{n}\frac{\tilde\beta_j+1}{p_j}\le1$ (see \cite{guan}, see also \cite{GY-conc4,GY-Hardy and product}).    
	Denote 
	$$I^\prime=\left\{(g,w_0)\in\mathcal{O}_{M,\times U,w_0} : g=\sum_{\alpha\in \mathbb{N}^{n+m}}b_\alpha(w-w_0)^\alpha\;\mbox{near}\;w_0\;\mbox{s.t.}\;b_\alpha=0\;\mbox{for}\;\alpha\in L_{\tilde{\beta}}\right\},$$
	where  $L_{\tilde{\beta}}=\left\{(\alpha_1,\cdots,\alpha_{m+n})\in \mathbb{N}^{n+m} :\alpha_j\leq \tilde{\beta}_j\;\mbox{for}\,1\leq j\leq  n\;\&\;(\alpha_{n+1},\cdots,\alpha_{n+m})\leq \beta^{\prime\prime}\right\}$ and $\tilde{\beta}=(\tilde{\beta}^\prime,\tilde{\beta}^{\prime\prime})\in\mathbb{N}^{n+m}$. It is clear that $(\hat{h}_0,w_0)\notin I^\prime$.    Denote 
	$$K_{\partial M\times U,\rho\gamma}^{I^\prime,\hat{h}_0}(w_0) :=\frac{1}{\inf\left\{||f||^2_{\partial M\times U,\rho\gamma} :f\in H^2_\kappa(M\times U,\partial M\times U)\;\& \;(f-\hat{h}_0,w_0)\in I^\prime\right\}}$$
	and
	$$B_{M\times U,\tilde{\rho}\gamma}^{I^\prime,\hat{h}_0}(w_0) :=\frac{1}{\inf\left\{\int_{M\times U}|f|^2\tilde{\rho}\gamma: f\in\mathcal{O}(M\times U)\;\&\; (f-\hat{h}_0,w_0)\in I^\prime\right\}}.$$ 
	When $I^\prime$ takes the maximal ideal of $\mathcal{O}_{w_0}$ and $\hat{h}_0(w_0)=1$, we have $K_{\partial M\times U,\kappa}^{I^\prime,\hat{h}_0}(w_0)=K_{\partial M\times U,\rho}(w_0)$.    
	
	Now we give a fibration version of Theorem \ref{key-Theorem2} as follows.
	\begin{Theorem}\label{key-Theorem14}
		Assume that $B_{M\times U,\tilde{\rho}\gamma}^{I^\prime,\hat{h}_0}(w_0)>0$. Then
		\begin{equation}\nonumber
			K_{\partial M\times U,\rho\gamma}^{I^\prime,\hat{h}_0}(w_0)\geq\left(\int_{0}^{+\infty}c(t)e^{-t}dt\right)\pi B_{M\times U,\tilde{\rho}\gamma}^{I^\prime,\hat{h}_0}(w_0)
		\end{equation}
		holds, and the equality holds if and only if the following statements hold:
		
		$(1)$ $\sum_{1\le j\le n}\frac{\beta_j'+1}{p_j}=\sum_{1\le j\le n}\frac{\beta_j+1}{p_j}=1$;
		
		$(2)$ $\varphi_j=2\log|g_j|+2u_j$ on $D_j$ for any $1\leq j\leq n$, and $g_1\equiv 1$ when $n=1$, where $u_j$ is a harmonic function on $D_j$ and $g_j$ is a holomorphic function on $\mathbb{C}$ such that $g_j(z_j)\neq 0$;
		
		$(3)$ $\chi_{j,z_j}^{\alpha_j+1}=\chi_{j,-u_j}$ for any $1\leq j\leq n$ and $\alpha\in E$ satisfying $d_\alpha\neq 0$, where $\chi_{j,z_j}$ and $\chi_{j,-u_j}$ are the characters associated to functions $G_{\Omega_j}(\cdot,z_j)$ and $-u_j$ respectively.
	\end{Theorem}

	\subsection{Space $H^2_{\lambda\gamma}(M\times U, S\times U)$} \label{sec:1.3}
	Let $D_j$ be a planar regular region with finite boundary components which are analytic Jordan curves for any $1\leq j\leq n$. Let $M=\prod_{1\leq l\leq n}D_j$. 
	
	Denote  $S :=\prod_{1\leq l\leq n}\partial D_j.$ Let $\lambda$ be a Lebesgue measurable function on $S$ such that $\inf_{S}\lambda>0.$ Let us recall the definition of the Hardy space $H_\lambda^2(M,S)$ over $S$.
	
	\begin{Definition}[\cite{GY-Hardy and product}]
		Let $f\in L^2(S,\lambda d\sigma),$ where $d\sigma :=\frac{1}{(2\pi)^n}|dw_1|\cdots|dw_n|.$ We call $f\in H_\lambda^2(M,S)$ if there exists $\left\{f_m\right\}_{m\in\mathbb{Z}_{\geq 0}}\subset \mathcal{O}(M)\cap C(\overline{M})\cap L^2(S,\lambda d\sigma)$ such that $$\lim\limits_{m\rightarrow+\infty}||f_m-f||^2_{S,\lambda}=0,$$
		where $||g||_{S,\lambda} :=\big(\int_S|g|^2\lambda d\sigma\big)^{\frac{1}{2}}$ for any $g\in L^2(S,\lambda d\sigma).$
	\end{Definition}
	$H_\lambda^2(M,S)$ is a Hilbert space equipped with the inner product $\ll\cdot,\cdot\gg_{S,\lambda}$ (see \cite{GY-Hardy and product}), defined by 
	$$\ll f,g\gg_{S,\lambda}=\frac{1}{(2\pi)^n}\int_Sf\overline{g}\lambda |dw_1|\cdots|dw_n|.$$
	There exists a linear injective map $P_S : H_\lambda^2(M,S)\rightarrow \mathcal{O}(M)$ satisfying that $P_S(f)=f$ for any $f\in \mathcal{O}(M)\cap C(\overline{M})\cap L^2(S,\lambda d\sigma)$ (see \cite{GY-Hardy and product}).
	
	We recall a kernel function $K_{S,\lambda}(z,\overline{w})$ (see \cite{GY-Hardy and product}) defined as follows:
	$$K_{S,\lambda}(z,\overline{w}) :=\sum\limits_{m=1}^{+\infty}P_S(e_m)(z)\overline{P_S(e_m)(w)}$$
	for $(z,w)\in M\times M\subset \mathbb{C}^{2n},$ where ${\left\{e_m\right\}}_{m\in\mathbb{Z}_{\geq 1}}$ is a complete orthonormal basis of $H_\lambda^2(M,S)$. The definition of $K_{S,\lambda}(z,\overline{w})$ is independent of the choice of ${\left\{e_m\right\}}_{m\in\mathbb{Z}_{\geq 1}}$ (see \cite{GY-Hardy and product}). Denote that $K_{S,\lambda}(z) :=K_{S,\lambda}(z,\overline{z}).$
	
	Let $z_0=(z_1,\cdots,z_n)\in M$.   
	Let $\varphi_j$ be a function on $\overline{D}_j$,  which is subharmonic on $D_j$, and continuous at $w_j$ for all $w_j\in \partial D_j$.  $\rho$ is a Lebesgue measurable function on $\partial M$ defined by: 
	$$\rho(w_1,\cdots,w_n)=\frac{1}{p_j}\bigg(\frac{\partial G_{D_j}(w_j,z_j)}{\partial\nu_{w_j}}\bigg)^{-1}\times\prod\limits_{1\leq l\leq n}e^{-\varphi_l(w_l)}$$
	on $\partial D_j\times \overline{M_j}$, where $\partial/\partial\nu_{w_j}$ denotes the derivative along the outer normal unit $\nu_{w_j}$ and $p_j>0$ is a constant for any $1\leq j\leq n$. Let 
	$$\lambda(w_1,\cdots,w_n)=\prod\limits_{1\leq j\leq n}\bigg(\frac{\partial G_{D_j}(w_j,z_j)}{\partial\nu_{w_j}}\bigg)^{-1}e^{-\varphi_j(w_j)}$$
	on $S=\prod\limits_{1\leq j\leq n}\partial D_j$, thus $\lambda>0$ is continuous on $S$.

	When $n=1$,  $K_{S,\lambda}(z_0)=\frac{1}{p_1}K_{\partial M,\rho}(z_0)$ by definitions.
	In \cite{GY-Hardy and product}, Guan-Yuan  
	gave a relation between $K_{S,\lambda}(z_0)$ and $K_{\partial M,\rho}(z_0) $ when $n>1.$
	
	\begin{Theorem}[\cite{GY-Hardy and product}]\label{key2-reference}
		Assume that $n>1$ and $K_{\partial M,\rho}(z_0)>0.$ Then
		\begin{equation}\nonumber
			K_{S,\lambda}(z_0)\geq \bigg(\sum\limits_{1\leq j\leq n}\frac{1}{p_j}\bigg)\pi^{n-1}K_{\partial M,\rho}(z_0)
		\end{equation}
		holds, and the equality holds if and only if the following statements hold:
		
		$(1)$ $\varphi_j$ is harmonic on $D_j$ for $1\leq j\leq n;$
		
		$(2)$ $\chi_{j,z_j}=\chi_{j,-\frac{\varphi_j}{2}},$ where $\chi_{j,z_j}\;\mbox{and}\;\chi_{j,-\frac{\varphi_j}{2}}$ are the characters associated to the functions   $G_{D_j}(\cdot,z_j)$ and $-\frac{\varphi_j}{2}$ respectively.
	\end{Theorem}
	
	Let $U$ be a domain in $\mathbb{C}^m$. We consider a space $H^2_{\lambda\gamma}(M\times U, S\times U)$ on $M\times U$.
	
	\begin{Definition}Let $\lambda$ be a positive and continuous function on $S$, and let $\gamma$ be an admissible weight on $U$.        $H^2_{\lambda\gamma}(M\times U, S\times U)$ denotes the set of holomorphic functions $f$ on $M\times U$ such that $P_S^{-1}(f(\cdot,u))\in H^2_\lambda(M,S)$ for every fixed $u\in U$ and $P_S^{-1}(f(z,u))$ is  Lebesgue mesurable on $S\times U$, with finite norms
		$$||f||_{S\times U,\lambda\gamma}: =\bigg(\int_U\gamma(u)dV_U\int_S|P_S^{-1}(f(\cdot,u))|^2\lambda d\sigma\bigg)^{\frac{1}{2}}.$$
	\end{Definition}
	
	$H^2_{\lambda\gamma}(M\times U, S\times U)$ is a Hilbert space with the following inner product:
	$$\ll f,g\gg_{S\times U,\lambda\gamma}: =\int_{S\times U}P_S^{-1}(f(z,u))\overline{P_S^{-1}(g(z,u))}\lambda(z)\gamma(u) d\sigma dV_U,$$
	and for any $(z,u)\in M\times U,$  the evaluation map $e_{(z,u)} : f\mapsto  f(z,u)$  from $f\in H^2_{\lambda\gamma}(M\times U, S\times U)$ to $\mathbb{C}$ is bounded (see Section \ref{sec:5.1}). Then  by  Riesz representation theorem,  for each $(z,u)\in M\times U$, there exists a unique $K_{S\times U,\lambda\gamma}((\cdot,\cdot),(\overline{z},\overline{u}))\in H^2_{\lambda\gamma}(M\times U, S\times U)$ such that
	$$f(z,u)=\int_{S\times U}P_S^{-1}(f(\zeta,w))\overline{P_S^{-1}(K_{S\times U,\lambda\gamma}((\zeta,w),(\overline{z},\overline{u})))}\lambda(\zeta)\gamma(w) d\sigma dV_U$$
	for every $f\in H^2_{\lambda\gamma}(M\times U, S\times U)$.
	Denote  $K_{S\times U,\lambda\gamma}(z,u) :=K_{S\times U,\lambda\gamma}((z,u),(\overline{z},\overline{u}))$.
	
	In the following, assume that $\lambda$ and $\rho$ are the same as those appearing in Theorem \ref{key2-reference}.  We present a fibration version of Theorem \ref{key2-reference} as follows:
	\begin{Theorem}\label{k-main2}
		Assume that $K_{\partial M\times U,\rho\gamma}(z_0,u_0)>0$ and $n>1$. Then
		\begin{equation}\nonumber
			K_{S\times U,\lambda\gamma}(z_0,u_0)\geq \bigg(\sum\limits_{1\leq j\leq n}\frac{1}{p_j}\bigg)\pi^{n-1}K_{\partial M\times U,\rho\gamma}(z_0,u_0)
		\end{equation}
		holds, and the equality  holds if and only if the following statements hold:
		
		$(1)$ $\varphi_j$ is harmonic on $D_j$ for $1\leq j\leq n;$
		
		$(2)$ $\chi_{j,z_j}=\chi_{j,-\frac{\varphi_j}{2}},$ where $\chi_{j,z_j}\;\mbox{and}\;\chi_{j,-\frac{\varphi_j}{2}}$ are the characters associated to the functions  $G_{D_j}(\cdot,z_j)$ and $-\frac{\varphi_j}{2}$ respectively.
	\end{Theorem}
	
	Let $l_0=\prod_{1\leq j\leq n}l_j$, where
	$l_j$ is a holomorphic function on a neighborhood of $z_j\in D_j$. Let  $\beta=(\beta_1,\cdots,\beta_{n})$, where $\beta_j:= ord_{z_j}(l_j)$ for any $1\leq j\leq n$. Let $\tilde{\beta}=(\tilde{\beta}_1,\cdots,\tilde{\beta}_n)\in\mathbb{N}^n$ satisfying that $\tilde{\beta_j}\geq\beta_j$ for any $1\leq j\leq n$.
	
	Denote $z_0:=(z_1,\cdots,z_n)\in M$, and    
	let 
	$$I_1=\left\{(g,z_0)\in \mathcal{O}_{M,z_0}: g=\sum_{\alpha\in \mathbb{N}^{n}}b_\alpha(w-z_0)^\alpha \mbox{ near }z_0 \mbox{ s.t. } b_\alpha=0 \;\mbox{for}\;\alpha\in L_{\tilde{\beta}^\prime}\right\},$$ where $$L_{\tilde{\beta}^\prime}=\left\{(\alpha_{1},\cdots,\alpha_{n})\in \mathbb{N}^{n}:\alpha_j\leq \tilde{\beta}_j\;\mbox{for}\;1\leq j\leq n \right\}.$$  It is clear that $I_1$ is an ideal of $\mathcal{O}_{M,z_0}$ and $(l_0,z_0)\notin I_1$. Denote 
	$$K_{S,\lambda}^{I_1,l_0}(z_0) :=\frac{1}{\inf\left\{||f||^2_{S,\lambda} :f\in H^2_\lambda(M,S) \&\;(P_S(f)-l_0,z_0)\in I_1\right\}}$$
	and
	$$K_{\partial M,\rho}^{I_1,l_0}(z_0) :=\frac{1}{\inf\left\{||f||^2_{\partial M,\rho} :f\in H^2_\rho(M,\partial M)\;\& \;(P_{\partial M}(f)-l_0,z_0)\in I_1\right\}}.$$

	Let $\lambda$ and $\rho$ be the same as those appearing in Theorem \ref{key2-reference}.    Guan-Yuan \cite{GY-Hardy and product} gave the following result.
	\begin{Theorem}[\cite{GY-Hardy and product}]\label{S-M-th1}
		Assume that $n>1$ and $K_{\partial M,\rho}^{I_1,l_0}(z_0)>0$. Then
		\begin{equation}\label{S-M}
			\left(\prod\limits_{1\leq j\leq n}(\tilde{\beta_j}+1)\right)K_{S,\lambda}^{I_1,l_0}(z_0)\geq\left(\sum\limits_{1\leq j\leq n}\frac{\tilde{\beta_j}+1}{p_j}\right)\pi^{n-1}K_{\partial M,\rho}^{I_1,l_0}(z_0)
		\end{equation}
		holds, and the equality holds if and only  if the following statements hold:
		
		$(1)$ $\beta=\tilde{\beta}$;
		
		$(2)$ $\varphi_j$ is harmonic on $D_j$ for any $1\leq j\leq n$;
		
		$(3)$ $\chi_{j,z_j}^{\tilde{\beta_j}+1}=\chi_{j,-\frac{\varphi_j}{2}}$, where $\chi_{j,-\frac{\varphi_j}{2}}$ and $\chi_{j,z_j}$ are the characters associated to the functions $-\frac{\varphi_j}{2}$ and $G_{D_j}(\cdot,z_j)$ respectively.
	\end{Theorem}
	When $n=1$, inequality (\ref{S-M}) becomes an equality by definitions.
	
	Let $U\subset\mathbb{C}^m$ be a domain. 
	Fix  $u_0\in U$, and  denote  $w_0:=(z_0,u_0)\in M\times U$.
	Let $b_0$ be a holomorphic funtion  on a neighborhood of $u_0$. Denote that $\hat{h}_0:=l_0b_0$.

	Denote  $\beta :=(\beta^\prime,\beta^{\prime\prime})$, where $\beta^\prime=(\beta_1,\cdots,\beta_{n})$, $\beta_j := ord_{z_j}(l_j)$ for $1\leq j\leq n$ and $\beta^{\prime\prime}=(\beta_{n+1},\cdots,\beta_{n+m}):=ord_{u_0}b_0$. 
	Let $\tilde{\beta}^\prime=(\tilde{\beta_1},\cdots,\tilde{\beta_n})\in\mathbb{N}^n$ and $\tilde{\beta}^{\prime\prime}=(\tilde{\beta}_{n+1},\cdots,\tilde{\beta}_{n+m})\in\mathbb{N}^m$, where $\tilde{\beta_j}\geq\beta_j$ for any $1\leq j\leq n$ and $\beta^{\prime\prime}\leq \tilde{\beta}^{\prime\prime}$.       
	Denote  $\tilde{\beta}:=(\tilde{\beta}^\prime,\tilde{\beta}^{\prime\prime})\in\mathbb{N}^{n+m}$.          
	Let 
	$$I^\prime=\left\{(g,w_0)\in\mathcal{O}_{M\times U,w_0} : g=\sum_{\alpha\in \mathbb{N}^{n+m}}b_\alpha(w-w_0)^\alpha\;\mbox{near}\;w_0\;\mbox{s.t.}\;b_\alpha=0\;\mbox{for}\;\alpha\in L_{\tilde{\beta}}\right\},$$
	where $$L_{\tilde{\beta}}=\left\{(\alpha_1,\cdots,\alpha_{m+n})\in \mathbb{N}^{n+m} :\alpha_j\leq \tilde{\beta}_j\mbox{ for }1\leq j\leq  n\;\&\;(\alpha_{n+1},\cdots,\alpha_{n+m})\leq \beta^{\prime\prime}\right\}.$$ It is clear that $(\hat{h}_0,w_0)\notin I^\prime$.     
	Denote 
	$$K_{\partial M\times U,\rho\gamma}^{I^\prime,\hat{h}_0}(w_0) :=\frac{1}{\inf\left\{||f||^2_{\partial M\times U,\rho\gamma} :f\in H^2_\kappa(M\times U,\partial M\times U)\;\& \;(f-\hat{h}_0,w_0)\in I^\prime\right\}}$$ 
	and
	$$K_{S\times U,\lambda\gamma}^{I^\prime,\hat{h}_0}(w_0):=\frac{1}{\inf\left\{||f||_{S\times U,\lambda\gamma}^2 :f\in H^2_{\lambda\gamma}(M\times U, S\times U)\;\&\;(f-\hat{h}_0,w_0)\in I^\prime\right\}}.$$
	
	Now, we give a fibration version of Theorem \ref{S-M-th1} as follows:
	\begin{Theorem}\label{Th4.2}
		Assume that $K_{\partial M\times U,\rho\gamma}^{I^\prime,\hat{h}_0}(w_0)>0$ and $n>1$. Then
		\begin{equation}\nonumber
			\left(\prod\limits_{1\leq j\leq n}(\tilde{\beta_j}+1)\right)K_{S\times U,\lambda\gamma}^{I^\prime,\hat{h}_0}(z_0)\geq\left(\sum\limits_{1\leq j\leq n}\frac{\tilde{\beta_j}+1}{p_j}\right)\pi^{n-1}K_{\partial M\times U,\rho\gamma}^{I^\prime,\hat{h}_0}(z_0)
		\end{equation}
		holds, and the equality holds if and only  if the following statements hold:
		
		$(1)$ $\beta=\tilde{\beta}$;
		
		$(2)$ $\varphi_j$ is harmonic on $D_j$ for any $1\leq j\leq n$;
		
		$(3)$ $\chi_{j,z_j}^{\tilde{\beta_j}+1}=\chi_{j,-\frac{\varphi_j}{2}}$, where $\chi_{j,-\frac{\varphi_j}{2}}$ and $\chi_{j,z_j}$ are the characters associated to the functions $-\frac{\varphi_j}{2}$ and $G_{D_j}(\cdot,z_j)$ respectively.
	\end{Theorem}

	\section{Hardy spaces on $D$ and on $D\times U$}

	In this section, we  firstly review some results about the Hardy space $H^2(D)$. Then we recall the definition of the admissible weight and its equivalent statement as well as a sufficient condition for a weight to be an admissible weight. As a consequence, we give a proof of the Bergman kernel decomposition formula. In the third subsection, we introduce the space $H^2_\theta(D\times U, \partial D\times U)$ and give some properties about the space. Finally, we finish the proof of Theorem \ref{main1-theorem}.

	\subsection{Some results about the space $H^2(D)$}
	For any domain $D\subset\mathbb{C}$, the Hardy space $H^2(D)$ (see \cite{rudin2}) denotes the set of all functions $f$ which are holomorphic on $D$, and for which there exists a harmonic function $u$ on $D$ such that $|f(z)|^2\leq u(z)$ for any $z\in D$. 
	
	The following lemma gives an sufficient and necessary condition for $f\in H^2(D)$.
	\begin{Lemma}[see \cite{rudin2}]
		\label{P1}
		Fix a point  $t\in D$, and let $f$ be a holomorphic function on $D$. Let $\Omega$ be any domain with smooth boundary $\Gamma$, such that $\Omega\cup\Gamma\subset D,$ and $t\in\Omega$. Then $f\in H^2(D)$ if and only if there exists a constant $M$, independent of $\Omega$, such that
		\begin{equation}\label{2:E1}
			\frac{1}{2\pi}\int_{\Gamma}|f(z)|^2\frac{\partial G_{\Omega}(z,t)}{\partial \nu_z}|dz|\leq M.
		\end{equation}
		Here $G_{\Omega}(z,t)$ is the Green function of  $\Omega$, the derivative is taken along the outer normal and a ``smooth" boundary is the union of a finite number of continuously differentiable curves.
		\begin{proof}
			Let $v$ be harmonic in $\Omega$, with boundary values $|f|^2$; then the left member of inequality \eqref{2:E1} is equal to $v(t)$. If $f\in H^2(D)$, let $u$ be a harmonic majorant of $|f|^2$. Since $|f|^2$  is subharmonic and $v(t)\leq u(t)$,  inequality \eqref{2:E1} holds with $M=u(t)$.
			
			Conversely, let $\left\{\Omega_k\right\}$ be an increasing sequence of domains with smooth boundary $\Gamma_k$, such that $\Omega_k\cup\Gamma_k\subset D,$ $t\in\Omega_1$ and $\bigcup_k\Omega_k=D$. Denote  
			$$v_k(w):=\frac{1}{2\pi}\int_{\Gamma_k}|f(z)|^2\frac{\partial G_{\Omega_k}(z,w)}{\partial \nu_z}|dz|,$$
			which is harmonic on $\Omega_k$.
			As $|f|^2$ is subharmonic and $v_k(t)\le M$ for any $k$, $\left\{v_k\right\}$ is increasing with respect to $k$ and converges to a harmonic function $u\not\equiv+\infty$ by Harnack's principle (see \cite{rudin}), and it is easy to see that this $u$ is a majorant  of $|f|^2$ on $D$.
		\end{proof}
	\end{Lemma}

	\begin{Remark}\label{least}In the proof of Lemma \ref{P1}, $u$ is the least harmonic majorant of $|f|^2$ on $D$.
		\begin{proof}
			For arbitrary $V(z)$ which is a harmonic majorant of $|f|^2$, we have 
			$$v_k(w)\leq \frac{1}{2\pi}\int_{\Gamma_k}V(z)\frac{\partial G_{\Omega_k}(z,w)}{\partial\nu_z}|dz|=V(w)$$ for any $w\in D$. Letting $k\rightarrow+\infty$, we have $u(w)\leq V(w).$
			We finish the proof.
		\end{proof} 
	\end{Remark}

	Now, we recall a norm on $H^2(D)$. Fixed  a point $t\in D$, for $f\in H^2(D)$, take
	\begin{equation}\nonumber
		\parallel f\parallel_{H^2(D)}=(u(t))^{1/2},
	\end{equation}
	where $u$ is the least harmonic majorant of $|f|^2$.  Let $\left\{D_k\right\}$ be an increasing sequence of domains with smooth boundaries $C_k$, such that $t\in D_1$ and $D=\bigcup_k D_k$. Let $G_{D_k}(z,t)$ be the Green's function of $D_k$. It follows from Remark \ref{least} that
	\begin{equation}\label{1:E3}
		\parallel f\parallel_{H^2(D)}=\lim\limits_{k\rightarrow\infty}\left\{\frac{1}{2\pi}\int_{C_k}|f(z)|^2\frac{\partial G_{D_k}(z,t)}{\partial \nu_z}|dz|\right\}^{1/2},
	\end{equation}
	where the limit is independent of the choice of $\left\{D_k\right\}$.\par
	
	The following lemma gives two properties about the norm $\|\cdot\|_{H^2(D)}$.
	\begin{Lemma}[see \cite{rudin2}]
		\label{1:C1}
		$(a)$ Let $K$ be a compact subset of $D$. There exists a constant $M(D,K,t)$, indenpendent of $f$, such that
		\begin{equation}\nonumber
			|f(z)|\leqslant M(D,K,t)\parallel f\parallel_{H^2(D)}
		\end{equation}
		for any $f\in H^2(D)$ and $z\in K$.
		
		$(b)$ If $\parallel f_n-f\parallel_{H^2(D)}\rightarrow 0$ as $n\rightarrow\infty$, then $f_n\rightarrow f$ uniformly on every compact subset of $D$.
	\end{Lemma}
	Minkowski's inequality may be applied to the integrals in equality $(\ref{1:E3})$, then the triangle inequality holds:
	\begin{equation}
		\parallel f+g\parallel_{H^2(D)}\leq\parallel f\parallel_{H^2(D)}+\parallel g\parallel_{H^2(D)}.\nonumber
	\end{equation}
	
	From now on, assume that $D$ is a bounded domain, whose boundary $\partial D$ consists of finite analytic Jordan curves.\par
	Let $L^2(\partial D)$ denote the space of complex-valued measurable functions $f^*$ on $\partial D$, normed by
	\begin{equation}
		\label{norm}
		\parallel f^*\parallel_{L^2(\partial D)}=\left\{\frac{1}{2\pi}\int_{\partial D}|f^*(z)|^2\frac{\partial G_D(z,t)}{\partial \nu_z}|dz|\right\}^{1/2}.
	\end{equation}
	
	We let $A(D)$ denote the set of all functions which are holomorphic in the closure  of $D$.
	We let $H^2(\partial D)$ denote the (evidently closed) subspace of $L^2(\partial D)$ consisting of those $f^*$ for which
	\begin{equation}\nonumber
		\int_{ \partial D}f^*(z)\phi(z)dz = 0
	\end{equation} 
	whenever $\phi\in A(D)$.

	To prove that the space $H^2(D)$ is a Banach space, we need the following lemma.
	
	\begin{Lemma}[see \cite{rudin2}]
		\label{1:L4}
		$(a)$ If $f\in H^2(D),$ there is a function $f^*$, defined on $\partial D$, such that $f$ has nontangential boundary values $f^*$ almost everywhere on $\partial D$.\\
		$(b)$ The mapping $\Sigma:f\rightarrow f^*$ is a norm-preserving isomorphism from $H^2(D)$ into $L^2(\partial D)$.\\
		$(c)$  The range of $\Sigma$ is $H^2(\partial D)$, and the inverse of $\Sigma$ is given by
		\begin{equation}\nonumber
			f(z)=\frac{1}{2\pi i}\int_{\partial D}\frac{f^*(w)}{w-z}dw,\qquad\qquad (z\in D)
		\end{equation}
		and also by
		\begin{equation}\label{1:E12}
			f(t)=\frac{1}{2\pi}\int_{\partial D}f^*(\zeta)\frac{\partial G_D(\zeta,t)}{\partial \nu_\zeta}|d\zeta|,
		\end{equation}
		where $G_D$ is the Green's function on $D$, and the derivative is taken along the outer normal. 
	\end{Lemma}
	From the Lemma \ref{1:L4} $(b)$, we know that for every $f\in H^2(D)$,
	\begin{equation}\nonumber
		\begin{split}
			\parallel f\parallel^2_{H^2(D)}
			&=\lim\limits_{k\rightarrow\infty}\frac{1}{2\pi}\int_{C_k}|f(z)|^2\frac{\partial G_{D_k}(z,t)}{\partial \nu_z}|dz|\\
			&=\frac{1}{2\pi}\int_{\partial D}|f^*(z)|^2\frac{\partial G_D(z,t)}{\partial \nu_z}|dz|.
		\end{split}
	\end{equation}

	Lemma \ref{1:L4} $(b)$ implies the following result.
	\begin{Lemma}[see \cite{rudin2}]\label{L1:4}
		$H^2(D)$ is a Banach space with the norm $\parallel\cdot\parallel_{H^2(D)}$.
	\end{Lemma}

	\subsection{Some results about admissible weights}\label{sec:admissible}
	In this section, we recall some results on admissible weights (see \cite{pasternak}).
	
	Let $X$ be an open set in $\mathbb{C}^n$. We say that $\mu$ is  a weight on $X$ if it is a real-valued and almost everywhere positive Lebesgue measurable function on $X$, and we   denote the set of all weights on $X$ by $W(X)$. If $\mu\in W(X)$, we denote by $L^2(X,\mu)$ the space of all Lebesgue measurable complex-valued functions $f$ satisfying that $\int_{ X}|f|^2\mu<+\infty$.  Denote
	$$A^2(X,\mu):=\left\{f\in \mathcal{O}(X):f\in L^2(X,\mu)\right\}.$$
	We define the evaluation functional $E_x$ on $A^2(X,\mu)$ by the formula $$E_x(f):=f(x), \qquad f\in A^2(X,\mu).$$
	
	A weight $\mu\in W(X)$ is called an admissible weight (see \cite{pasternak}), if the following two conditions hold:
	
	$(1)$ $A^2(X,\mu)$ is a closed subspace of $L^2(X,\mu)$;
	
	$(2)$   for any $x\in X$, the evaluation functional $E_x$ is continuous on $A^2(X,\mu)$. 
	
	\
	
	The set of all admissible weights on $X$ is denoted by $AW(X)$.  
	
	\begin{Lemma}[\cite{pasternak}]
		\label{e-c}
		Let $\mu\in W(X)$. The following are equivalent:
		
		$(1)$ $\mu$ is an admissible weight;
		
		$(2)$ for any compact set $Y\subset X$ there exists a constant $C_Y$ such that for any $z\in Y$ and each $f\in A^2(X,\mu),$ $$|E_z(f)|\leq C_Y\parallel f\parallel_\mu;$$
		
		$(3)$ for any $x\in X$ there exists a neighborhood $V_x$ of $x$ in $X$ and a constant $C_x>0$ such that for any $w\in V_x$ and each $f\in A^2(X,\mu)$ $$|E_w(f)|\leq C_x\parallel f\parallel_\mu.$$   
	\end{Lemma}
	Now we recall two sufficient conditions for a weight to be an admissible weight.
	
	\begin{Lemma}[\cite{pasternak}]
		\label{suffi-condition}
		Let $\mu\in W(X)$, and $V$ be an open subset of $X$. Assume that there exists a number $a>0$ such that the function $\mu^{-a}$ is integrable on $V$ with respect to the Lebesgue measure. Then for any $z\in V$, there exists a neighborhood $V_z$ of $z$ in $X$ and a constant $C_z>0$ such that for each $w\in V_z$ and each $f\in A^2(X,\mu)$
		$$|E_w(f)|\leq C_z\parallel f\parallel_\mu.$$
	\end{Lemma}
	
	\begin{Lemma}[see \cite{GY-Hardy and product}]
		\label{2:C2}
		Let 
		$\mu\in W(X),$ and $S$ be an analytic subset of $X$. Assume that  for any subset $V\Subset X\backslash S$, there exists a number $a>0$ such that the function $\mu^{-a}$ is integrable on $V$ with respect to the Lebesgue measure. Then $\mu$ is an admissible weight on $X$.
	\end{Lemma}

	Let $X_1$, $X_2$ be two domains of $\mathbb{C}^n$ and $\mathbb{C}^m$ respectively, and $\varphi_1,$ $\varphi_2$ be admissible weights on $X_1,$ $X_2$ respectively. 
	\begin{Lemma}
		$\varphi_1\varphi_2$ is an admissible weight on $X_1\times X_2$. 
	\end{Lemma}
	\begin{proof}
		By Fubini's theorem, we know that if $f(z_1,z_2)\in A^2(X_1\times X_2,\varphi_1\varphi_2)$, then for almost all $z_2\in X_2$, $|f(z_1,z_2)|^2$ is integrable on $X_1$. Thus, for almost all $z_2\in X_2$, it follows from  Lemma \ref{e-c} that for any open subset $Y_1\subset\subset X_1$, there exists $C_{Y_1}$ such that
		$$
		\sup\limits_{z_1\in Y_1}|f(z_1,z_2)|^2\leq C_{Y_1}\int_{X_1}|f(z_1,z_2)|^2\varphi_1d\lambda_1(z_1).
		$$
		Similarly, for any open subset $Y_2\subset\subset X_2$ and almost all $z_1\in X_1$, there exists $C_{Y_2}$ such that
		$$
		\sup\limits_{z_2\in Y_2}|f(z_1,z_2)|^2\leq C_{Y_2}\int_{X_2}|f(z_1,z_2)|^2\varphi_2d\lambda_2(z_2).
		$$
		Combining with the above two inequalities, we obtain that 
		\begin{equation}
			\label{eq:0816a}|f(z_1,z_2)|^2\leq C_{Y_1}C_{Y_2}\int_{X_1\times X_2}|f(z_1,z_2)|^2\varphi_1\varphi_2d\lambda_1(z_1)d\lambda_2(z_2)
		\end{equation}
		holds for almost all $(z_1,z_2)\in Y_1\times Y_2$.
		As $|f(z_1,z_2)|$ is continous on $X_1\times X_2$, we know that  inequality \eqref{eq:0816a} holds for all $(z_1,z_2)\in Y_1\times Y_2$.
		Hence, by  Lemma \ref{e-c}, we know that $\varphi_1\varphi_2$ is an admissible weight on $X_1\times X_2$.   
	\end{proof}

	The following lemma says that the product of bases of $A^2(X_1,\varphi_1)$ and $A^2(X_2,\varphi_2)$  makes up of a basis of $A^2(X_1\times X_2,\varphi_1\varphi_2)$. 
	
	\begin{Lemma}[see \cite{boundary-minimal-concavity}]
		\label{A-composite-basis}
		Let $\left\{f_i(z)\right\}_{i\in\mathbb{N}}$ and $\left\{g_j(w)\right\}_{j\in\mathbb{N}}$ be the complete orthonormal basis of $A^2(X_1,\varphi_1)$ and $A^2(X_2,\varphi_2)$ respectively. Then $\left\{f_i(z)g_j(w)\right\}_{i,j\in\mathbb{N}}$ is a complete orthonormal basis of $A^2(X_1\times X_2,\varphi_1\varphi_2)$.     
	\end{Lemma}
	
	Let $X$ be an open set of $\mathbb{C}^n$, $\varphi_0$ be  an admissible weight on $X$. Denote  $B_{X,\varphi_0}(\zeta,\overline{z})$ the weighted Bergman kernel on $X$ with the weight $\varphi_0$ if
	$$f(z)=\int_{ X}f(\zeta)\overline{B_{X,\varphi_0}(\zeta,\overline{z})}\varphi_0(\zeta)d\zeta$$ for any  $f\in A^2(X,\varphi_0)$.
	When $\zeta=z$, $B_{X,\varphi_0}(z)$ denotes $B_{X,\varphi_0}(z,\overline{z})$ for simplicity.
	
	In the following, we recall a decomposition formula for the Bergman kernel on $X_1\times X_2$, for which we give a proof since we could not track down  a reference in the literature.
	\begin{Lemma}\label{Le-719}
		\begin{equation}\label{3:E8}
			B_{X_1\times X_2,\varphi_1\varphi_2}((\zeta,u),(\overline{z},\overline{w}))=B_{X_1,\varphi_1}(\zeta,\overline{z})B_{X_2,\varphi_2}(u,\overline{w})
		\end{equation}
		holds for $(\zeta,u),(z,w)\in X_1\times X_2$.
	\end{Lemma}
	\begin{proof}
		For every $f\in A^2(X_1\times X_2,\varphi_1\varphi_2)$,  let $g(\zeta)=\int_{ X_2}|f(\zeta,z_2)|^2{\varphi_2(z_2)}d\lambda_2(z_2)$, which is finite for almost all $\zeta\in X_1$ by Fubini's theorem.
		We claim that
		\begin{equation}\label{2:E33}
			g(\zeta)<+\infty
		\end{equation}
		holds for all $\zeta\in X_1$, i.e., $f(\zeta,\cdot)\in A^2(X_2,\varphi_2)$ for all $\zeta\in X_1$. 
		
		Firstly, $\int_{X_1}|f(z_1,z_2)|^2\varphi_1(z_1)d\lambda(z_1)<+\infty$ holds for a.e. $z_2\in X_2$ by Fubini's theorem. On the other hand, for $\varphi_1(z_1)$ is an admissible weight, by   Lemma \ref{e-c},  there exists $C_{Y_1}$, independent of $f$ and $z_2$, such that
		\begin{equation}
			\sup\limits_{z_1\in Y_1} |f(z_1,z_2)|^2\leq C_{Y_1}\int_{ X_1}|f(\zeta,z_2)|^2\varphi_1(\zeta)d\lambda_1(\zeta),\nonumber
		\end{equation}
		where $Y_1$ is a compact subset of $X_1$. Thus,
		$$\int_{X_2}|f(z_1,z_2)|^2\varphi_2(z_2)d\lambda(z_2)\leq C_{Y_1}\int_{ X_1\times X_2}|f(\zeta,z_2)|^2\varphi_1(\zeta)\varphi_2(z_2)d\lambda(\zeta)d\lambda(z_2)<+\infty, $$
		where $z_1\in Y_1$.
		Hence, inequality (\ref{2:E33}) holds for all $\zeta\in X_1$. 
		Similarly, we get 
		\[\int_{X_1}|f(z_1,z_2)|^2{\varphi_1(z_1)}d\lambda(z_1)<+\infty\]
		holds for all $z_2\in X_2$, i.e., $f(\cdot,z_2)\in A^2(X_1,\varphi_1)$ for all $z_2\in X_2$.
		Then we obtain equality (\ref{3:E8})
		because of the equality
		\begin{equation}
			\begin{aligned}
				&\int_{ X_1\times X_2}f(\zeta,u)\overline{B_{X_1,\varphi_1}(\zeta,\overline{z})B_{X_2,\varphi_2}(u,\overline{w})}{(\varphi_1(\zeta)\varphi_2(u))}d\lambda(\zeta) d\lambda(u)\\
				&=\int_{ X_1}\overline{B_{X_1,\varphi_1}(\zeta,\overline{z})}{\varphi_1(\zeta)}d\lambda(\zeta)\int_{ X_2}f(\zeta,u)\overline{B_{X_2,\varphi_2}(u,\overline{w})}{\varphi_2(u)}d\lambda(u)\\
				&=\int_{ X_1}f(\zeta,w)\overline{B_{X_1,\varphi_1}(\zeta,\overline{z})}{\varphi_1(\zeta)}d\lambda(\zeta)\\
				&=f(z,w).
			\end{aligned}
			\nonumber
		\end{equation}  
	\end{proof}

	\subsection{Some results about the space $H^2_\theta(D\times U, \partial D\times U)$}
	Firstly, we recall some notations. Let
	$D$ be a planar regular region with  finite  boundary components,  which are analytic Jordan curves. Let $U$ be a domian in $\mathbb{C}^m$.  let
	$\lambda({z})$ be a positive and continuous function on $\partial D$, and $\gamma(u)$   be  an admissible weight on $U$. Let $\theta(z,u)=\lambda({z})\gamma(u)$ on $\partial D\times U$.

	Let $H^2_\theta(D\times U, \partial D\times U)$ denote the space defined as the set of all  holomorphic functions $f(z,u)\in\mathscr{F}$  with finite norms
	\begin{equation}\nonumber
		\parallel f\parallel_{\partial D\times U,\theta}^2=\frac{1}{2\pi}\int_{\partial D\times U}|f(z,u)|^2\theta(z,u)|dz|dV_{U},
	\end{equation} 
	where  $dV_U$ is the Lebesgue measure on $U$ and $\mathscr{F}$ denotes the set of all  holomorphic functions $f(z,u)$ on $D\times U$, which satisfy
	$f(\cdot,u)\in H^2(D)$ for every fixed $u\in U$.

	For $f\in H^2_\theta(D\times U, \partial D\times U)$,  let $W_f(u)=\frac{1}{2\pi}\int_{\partial D}|f(z,u)|^2\lambda({z})|dz|$.
	\begin{Lemma}\label{2L:5}
		For any compact subset $K$ of $ U$,
		there is a positive constant ${C(K)}$, independent of $f$, such that
		\begin{equation}\label{2:E5}
			\sup\limits_{u\in K} W_f(u)\leq {C(K)}\parallel f\parallel^2_{\partial D\times U,\theta}.
		\end{equation}
	\end{Lemma}
	\begin{proof}
		Let $t\in D,$ and  $Q(t,u):=\frac{1}{2\pi}\int_{\partial D}|f(\zeta,u)|^2\frac{\partial G_{D}(\zeta,t)}{\partial\nu_\zeta}|d\zeta|$. As $f\in H^2_\theta(D\times U, \partial D\times U)$, we have $Q(t,u)<+\infty$ for any $ u\in U$.  For fixed $t\in D$, let $D_{k}=\left\{z\in D|G_{D}(z,t)\leq \log(1-\frac{1}{k})\right\}$ with smooth boundary $\partial D_{k}$, which is  an increasing sequence of domains  with respect to $k$  such that  $D=\bigcup_k D_{k}$.  It is well known that $G_{D_{k}}(\cdot,t)=G_{D}(\cdot,t)-\log(1-\frac{1}{k})$ is the Green function on $D_{k}$.
		
		By equalities (\ref{norm}), (\ref{1:E3}) and Lemma \ref{1:L4} $(b)$, we know that
		\begin{equation}\label{eq:0826a}
			Q(t,u)=\lim\limits_{k\rightarrow\infty}\frac{1}{2\pi}\int_{\partial D_k}|f(z,u)|^2\frac{\partial G_{D_k}(z,t)}{\partial\nu_z}|dz|.
		\end{equation}
		Denote 
		$$Q_k(t,u):=\frac{1}{2\pi}\int_{\partial D_k}|f(z,u)|^2\frac{\partial G_{D_k}(z,t)}{\partial\nu_z}|dz|.$$ Then $Q_k(t,u)$ increasingly tends to $Q(t,u)$ as $k\rightarrow\infty$ by equality \eqref{eq:0826a} and the subharmonicity of $|f(z,u)|^2$ with respect to $z$.

		By equality (\ref{1:E12}), we get
		\[f(z,u)=\frac{1}{2\pi}\int_{\partial D}f(\zeta,u)\frac{\partial G_D(\zeta,z)}{\partial \nu_\zeta}|d\zeta|.\]
		Then it follows from  the Cauchy-Schwarz inequality that
		\begin{equation}\nonumber
			|f(z,u)|^2\leq \frac{1}{2\pi}\int_{\partial D}|f(\zeta,u)|^2\frac{\partial G_D(\zeta,z)}{\partial \nu_\zeta}|d\zeta|
		\end{equation}
		holds. Thus, we obtain
		\begin{equation}\nonumber
			\begin{aligned}
				&\int_{ U}|f(z,u)|^2\gamma(u)dV_U\\
				&\leq \int_{ U}\frac{1}{2\pi}\int_{\partial D}|f(\zeta,u)|^2\frac{\partial G_D(\zeta,z)}{\partial \nu_\zeta}|d\zeta|\gamma(u)dV_U\\
				&=\frac{1}{2\pi}\int_{\partial D\times U}|f(z,u)|^2\frac{\partial G_{D}(z,t)}{\partial\nu_z}{\gamma(u)}|dz|dV_{U}\\
				&<+\infty,
			\end{aligned}
		\end{equation}    
		i.e., $f(z,\cdot)\in A^2(U,\gamma)$ for any $z\in D$. On the other hand, $\gamma$ is an admissible weight,
		so by Lemma \ref{e-c} $(2)$, for any compact subset $K\subset U$, we have
		\begin{equation}\label{evalue-ineq}
			|f(z,w)|^2\leq C_K\int_{ U}|f(z,u)|^2\gamma(u)dV_U,\qquad w\in K,
		\end{equation}
		where $C_K$ is a constant independent of $z\in D$ and $f$.
		Inequality (\ref{evalue-ineq}) implies that        \begin{equation}\nonumber
			Q_k(t,w)=\frac{1}{2\pi}\int_{\partial D_k}|f(z,w)|^2\frac{\partial G_{D_k}(z,t)}{\partial\nu_z}|dz|\leq C_K\int_{ U}Q_k(t,u)\gamma(u)dV_U \qquad w\in K. 
		\end{equation}
		Let $k\rightarrow \infty$, and it follows from the monotone convergence theorem that
		\begin{equation}\label{2:E18}
			Q(t,w)\leq C_K \frac{1}{2\pi}\int_{\partial D\times U}|f(z,u)|^2\frac{\partial G_{D}(z,t)}{\partial\nu_z}{\gamma(u)}|dz|dV_{U}\quad w\in K.
		\end{equation}

		Since $(\partial G_D/\partial \nu_z)(z,t)$ and $\lambda(\zeta)$ are positive and continuous on $\partial D$,  it follows from the compactness of $\partial D$ that there exist positive $m$ and $M$ such that
		\begin{equation}\label{2:E05}
			m\leq\frac{(\partial G_D/\partial \nu_z)(z,t)}{\lambda(\zeta)}\leq M. 
		\end{equation}  Thus,
		Combining with  inequalities (\ref{2:E18}) and (\ref{2:E05}), we obtain that
		inequality (\ref{2:E5}) holds with ${C(K)}=\frac{MC_K}{m}$. 
		
		Thus, Lemma \ref{2L:5} holds.
	\end{proof}

	Using Lemma \ref{2L:5}, we have the following convergence result.    
	
	\begin{Corollary}\label{2:C1}
		If $\left\{f_n\right\}$ is a Cauchy sequence in $ H^2_\theta(D\times U, \partial D\times U)$, then for every fixed $u\in U$, there exists a function $f(\cdot,u)\in H^2(D)$ such that $f_n(\cdot,u)\rightarrow f(\cdot,u)$ uniformly on every compact subset of $D$ and
		\begin{equation}\nonumber
			\lim\limits_{k\rightarrow\infty}\parallel f_k(\cdot,u)-f(\cdot,u)\parallel_{H^2(D)}=0.
		\end{equation} 
	\end{Corollary}

	\begin{proof}
		Let $S(u):=\frac{1}{2\pi}\int_{\partial D}|f_n(z,u)-f_m(z,u)|^2\lambda(z)|dz|$. By Lemma \ref{2L:5} and inequality (\ref{2:E05}), then we have 
		\[
		\sup\limits_{u\in K}S(u)\leq{C(K)}\parallel f_n-f_m\parallel^2_{\partial D\times U,\theta}\] 
		and
		\begin{equation*}
			\parallel f_n(\cdot,u)-f_m(\cdot,u)\parallel^2_{H^2(D)}\leq\frac{1}{M}S(u)\leq \frac{C(K)}{M}\parallel f_n-f_m\parallel^2_{\partial D\times U,\theta}.
		\end{equation*}
		So $\left\{f_n(\cdot,u)\right\}$ is a Cauchy sequence in $H^2(D)$ for every fixed $u$. Following from Lemma \ref{L1:4} and Lemma \ref{1:C1} $(b)$, there exists a function $f(\cdot,u)\in H^2(D)$ such that $f_n(\cdot,u)\rightarrow f(\cdot,u)$ uniformly on every compact subset of $D$ and
		$
		\lim\limits_{k\rightarrow\infty}\parallel f_k(\cdot,u)-f(\cdot,u)\parallel_{H^2(D)}=0.
		$
		The proof is finished.
	\end{proof}

	For any $(z,u)\in D\times U$ and $f\in H^2_\theta(D\times U, \partial D\times U)$, let $e_{(z,u)}$ be an evaluation functional from $ H^2_\theta(D\times U, \partial D\times U)$ to $\mathbb{C}$, defined as $$e_{(z,u)}(f):=f(z,u).$$
	
	Let $\pi$ be the nature projection from $V\times U$ to $V$, and $\pi_2$ be the nature projection from $V\times U$ to $U$, where $V$ is a neighborhood of $\overline{D}$. Now we prove that the evaluation functional is bounded.
	
	\begin{Lemma}
		\label{2:L3}
		For any compact subset K of $D\times U$ and $f\in H^2_\theta(D\times U, \partial D\times U)$, there exists a positive constant  $M(K)$, independent of $f$, such that  
		\begin{equation}\label{2:E4}
			|f(z,u)|\leqslant {M}(K)\parallel f\parallel_{\partial D\times U,\theta}
		\end{equation}
		for any $(z,u)\in K.$
		In particular,  $e_{(z,u)}$ is a bounded functional on $H^2_\theta(D\times U, \partial D\times U)$. 
		
		\begin{proof}
			It follows from the Lemma \ref{1:C1} $(a)$ and Lemma \ref{2L:5} that
			\begin{align*}
				&\int_{\partial D\times U }|f(\zeta,u)|^2\lambda(\zeta)\gamma(u)|d\zeta|dV_{U}\\
				&\geqslant C_1\int_{\partial D}|f(\zeta,u)|^2\lambda(\zeta)|d\zeta| \\
				&\geqslant  C_1m_1\int_{\partial D}|f(\zeta,u)|^2\frac{\partial G_D(\zeta,u)}{\partial\nu_{\zeta}}|d\zeta|\\
				&\geqslant 2\pi C_1m_1 C_{2}|f(z,w)|^2
			\end{align*}
			holds
			{for any} $(z,w)\in K$, {where }$m_1=\inf_{\partial D}\frac{\lambda(\zeta)}{\frac{\partial G_D(\zeta,u)}{\partial\nu_{\zeta}}}$, $C_1$ and $C_2$ are constants  independent of $f$.
			Thus, inequality \eqref{2:E4} holds with $M(K)=\sqrt{2\pi C_1m_1 C_{2}}$. 
		\end{proof}
	\end{Lemma}

	Lemma \ref{2:L3} implies the following result.
	\begin{Corollary}
		\label{bound-evaluation}
		If $\left\{f_n(z,u)\right\}$ is a Cauchy sequence in $H^2_\theta(D\times U, \partial D\times U)$, then the sequence converges uniformly on any compact subset of $D\times U$, whose limit  is holomorphic on $D\times U$. 
	\end{Corollary}
	\begin{proof}
		By Lemma \ref{2:L3}, applying inequality (\ref{2:E4}) with $f_n-f_m$, we know that $\left\{f_m\right\}$ converges uniformly  on each compact subset of $D\times U$, whose limit  is holomorphic on $D\times U$. The proof is finished.
	\end{proof}

	The following proposition shows that  $H^2_\theta(D\times U, \partial D\times U)$ is a Hilbert space.

	\begin{Proposition}\label{2:L4}
		$H^2_\theta(D\times U, \partial D\times U)$ is a Hilbert space with the following inner product:
		\begin{equation*}
			\ll f,g\gg_{\partial D\times U,\theta} 
			:\;=\frac{1}{2\pi}\int_{\partial D\times U}f(z,u)\overline{g(z,u)}\theta(z,u)|dz|dV_U.
		\end{equation*}
	\end{Proposition}
	\begin{proof}It suffices to prove the any  Cauchy sequence in $H^2_\theta(D\times U, \partial D\times U)$ has  a limit.      Let $\left\{f_m\right\}$ be a  Cauchy sequence in $H^2_\theta(D\times U, \partial D\times U)$. Then  by Corollary \ref{bound-evaluation}, $\left\{f_m\right\}$ converges uniformly on any compact subset of $D\times U$ to a holomorphic function $f$  on $D\times U$. On the other hand, it follows from Corollary \ref{2:C1} that, for any $u\in U$, there exists an  $f(\cdot,u)\in H^2(D)$ such that $\left\{f_m(\cdot,u)\right\}$ converges uniformly to $f(\cdot,u)$ on any compact subset  of $D$.  Hence, $f$ is holomorphic on $D\times U$ and $f(\cdot,u)\in H^2(D)$, i.e., $f\in\mathscr{F}.$ 
		
		For the present, we show that $f\in H^2_\theta(D\times U, \partial D\times U)$ and $\parallel f_m-f\parallel_{\partial D\times U,\theta}\rightarrow 0$ as $m\rightarrow \infty $, which implies $H^2_\theta(D\times U, \partial D\times U)$ is a Hilbert space. It suffices to prove  
		\begin{equation*}
			\parallel f\parallel_{\partial D\times U,\theta}<+\infty
		\end{equation*}
		and
		$$\parallel f-f_m\parallel_{\partial D\times U,\theta}\rightarrow 0\;\mbox{as}\;m\rightarrow \infty.$$
		
		Let $$V_m(u)=\int_{ \partial D}|f_m(z,u)|^2\lambda(z)|dz|.$$  
		Then  it follows from Lemma \ref{2L:5} that, for any compact subset $K$ of $ U$, there exists $C(K)$ such that
		\begin{equation}\nonumber
			\sup\limits_{u\in K}V_m(u)\leq {C(K)}\parallel f_m\parallel_{\partial D\times U,\theta}^{2}
		\end{equation}
		for any $m$. 
		As $\left\{f_m\right\}$ is a Cauchy sequence in $H^2_\theta(D\times U, \partial D\times U)$, there is a positive constant $C$ such that 
		\begin{equation}\nonumber
			\parallel f_m\parallel_{\partial D\times U,\theta}^2\leq C.
		\end{equation}    
		Let 
		\[V(u)=\frac{1}{2\pi}\int_{\partial D}|f(\zeta,u)|^2\lambda(\zeta)|d\zeta|.\]
		By Corollary \ref{2:C1}, we have
		$
		\lim\limits_{m\rightarrow\infty}V_m(u)=V(u).
		$
		
		By Fatou's lemma and Cauchy-Schwarz inequality, we have
		\begin{equation}\label{key-ine}
			\begin{aligned}
				&\int_{ U}|V(u)-V_n(u)|{\gamma(u)}dV_U\\
				&\leq\liminf\limits_{m\rightarrow\infty}\int_{ U}|V_m(u)-V_n(u)|{\gamma(u)}dV_U\\
				&\leq\liminf\limits_{m\rightarrow\infty}\int_{ U}\int_{ \partial D}\big||f_m(z,u)|^2-|f_n(z,u)|^2\big|\lambda(z)|dz|{\gamma(u)}dV_U\\
				&=\liminf\limits_{m\rightarrow\infty}\int_{ \partial D\times U}|f_m(z,u)-f_n(z,u)|(|f_m(z,u)|+|f_n(z,u)|)\lambda(z){\gamma(u)}|dz|dV_U\\
				&\leq\liminf\limits_{m\rightarrow\infty}\bigg(\int_{ \partial D\times U}|d|^2\lambda(z){\gamma(u)}|dz|dV_U\bigg)^{1/2}
				\bigg(\int_{ \partial D\times U}a^2\lambda(z){\gamma(u)}|dz|dV_U\bigg)^{1/2} \\
				&\leq\liminf\limits_{m\rightarrow\infty}\parallel f_m-f_n\parallel_{\partial D\times U,\theta}\bigg(\int_{ U}2(V_m(u)+V_n(u)){\gamma(u)}dV_U\bigg)^{1/2}\\
				&\leq2\sqrt{C}\liminf\limits_{m\rightarrow\infty}\parallel f_m-f_n\parallel_{\partial D\times U,\theta}\\
				&<+\infty, 
			\end{aligned}
		\end{equation}
		where $a=|f_n(z,u)|+|f_m(z,u)|$ and $d=f_n(z,u)-f_m(z,u)$. 
		Then it follows from inequality (\ref{key-ine}) that
		\begin{equation}\label{2:E34}
			\parallel f\parallel_{\partial D\times U,\theta}^2=\int_{ U}V(u)\gamma(u)dV_U=\lim\limits_{m\rightarrow\infty}\int_{ U}V_m(u)\gamma(u)dV_U<+\infty.
		\end{equation}
		
		Let $g_m(u)= \int_{\partial D}|f(z,u)-f_m(z,u)|^2\lambda(z)|dz|$. Then we know
		$$\lim\limits_{m\rightarrow\infty}g_m(u)=0$$ by Corollary \ref{2:C1}, 
		and 
		\begin{equation}\label{control}
			\begin{aligned}
				g_m(u)&= \int_{\partial D}|f(z,u)-f_m(z,u)|^2\lambda(z)|dz| \\
				&\leq 2\int_{\partial D}(|f(z,u)|^2+|f_m(z,u)|^2)\lambda(z)|dz|\\
				&= 2(V_m(u)+V(u)).
			\end{aligned}
		\end{equation}
		
		Let $h_m(u)=2(V_m(u)+V(u))-g_m(u)$, and we have $h_m(u)\geq 0$ by inequality (\ref{control}). Then by Fautou's Lemma, we obtain
		\begin{equation}
			\int_{ U}\liminf\limits_{m\rightarrow\infty}h_m(u)\gamma(u)dV_U\leq\liminf\limits_{m\rightarrow\infty}\int_{ U}h_m(u)\gamma(u)dV_U,\nonumber
		\end{equation}
		i.e.,
		\begin{equation}
			4\int_{ U}V(u)\gamma(u)dV_U\leq\liminf\limits_{m\rightarrow\infty}\int_{ U}2(V_m(u)+V(u))\gamma(u)-g_m(u)\gamma(u)dV_U.\nonumber
		\end{equation}
		By equality (\ref{2:E34}), we have
		\begin{equation*}
			\limsup\limits_{m\rightarrow\infty}\int_{ U}g_m(u)\gamma(u)dV_U=0,
		\end{equation*}
		i.e., $$\lim_{m\rightarrow\infty}\parallel f-f_m\parallel_{\partial D\times U,\theta}=0.$$
		The proof is finished.  
	\end{proof}

	Now, let us define a kernel function     $K_{\partial D\times U,\theta}$  on $D\times U$.

	By Lemma \ref{2:L3}, we know that the evaluation functional is bounded.
	Proposition \ref{2:L4} tells us  that $H^2_\theta(D\times U, \partial D\times U)$ is a Hilbert space with the inner product $\ll \cdot,\cdot\gg_{\partial D\times U,\theta}$. 
	Hence, by  Riesz Representation Theorem, for every $(z,w)\in D\times U$, there is a unique function $K_{\partial D\times U,\theta}((\cdot,\cdot),(\overline{z},\overline{w}))\in H^2_\theta(D\times U, \partial D\times U)$, with the reproducing property
	\begin{equation}
		f(z,w)=\frac{1}{2\pi}\int_{\partial D\times U,}f(\zeta,u)\overline{K_{\partial D\times U,\theta}((\zeta,u),(\overline{z},\overline{w}))}\theta(\zeta,u)|d\zeta|dV_U
		\nonumber 
	\end{equation}
	for $f\in H^2_\theta(D\times U, \partial D\times U)$.

	\subsection{Proof of Theorem \ref{main1-theorem}} In this section, we prove Theorem \ref{main1-theorem}.    
	
	We call  $K_{ D,\lambda}(z,\overline w)$ (see \cite{nehari}) the weighted Szeg\"o kernel if 
	$$f(w)=\frac{1}{2\pi}\int_{\partial D}f(z)\overline{K_{D,\lambda}(z,\overline w)}\lambda(z)|dz|$$
	holds for any $f\in H^2(D)$.
	
	For the present, we give two decompostition formula.
	
	Choosing any $f\in H^2_\theta(D\times U, \partial D\times U)$, we have  $f(\cdot,u)\in H^2(D)$,  thus 
	$$f(w,u)=\frac{1}{2\pi}\int_{\partial D}f(z,u)\overline{K_{D,\lambda}(z,\overline{w})}\lambda(z)|dz|.$$
	For fixed $z\in D$, $(\partial G_D/\partial \nu_\zeta)(\zeta,z)$ and $\lambda(\zeta)$ are positive and continuous on $\partial D$. Then it follows from the compactness of $\partial D$ that there exists a positive constant  $M$ such that 
	\[\frac{1}{M}\leq\frac{(\partial G_D/\partial \nu_\zeta)(\zeta,z)}{\lambda(\zeta)}\leq M.\] 
	By equality (\ref{1:E12}), we get
	\[f(z,u)=\frac{1}{2\pi}\int_{\partial D}f(\zeta,u)\frac{\partial G_D(\zeta,z)}{\partial \nu_\zeta}|d\zeta|.\]
	Then it follows from  the Cauchy-Schwarz inequality that
	\begin{equation*}
		|f(z,u)|^2\leq \frac{1}{2\pi}\int_{\partial D}|f(\zeta,u)|^2\frac{\partial G_D(\zeta,z)}{\partial \nu_\zeta}|d\zeta|.
	\end{equation*}
	Thus, we obtain
	\begin{equation*}
		\begin{aligned}
			&\int_{ U}|f(z,u)|^2\gamma(u)dV_U\\
			&\leq \int_{ U}\frac{1}{2\pi}\int_{\partial D}|f(\zeta,u)|^2\frac{\partial G_D(\zeta,z)}{\partial \nu_\zeta}|d\zeta|\gamma(u)dV_U\\
			&\leq M\int_{ U}\frac{1}{2\pi}\int_{\partial D}|f(\zeta,u)|^2\lambda(\zeta)|d\zeta|\gamma(u)dV_U\\
			&=M\parallel f\parallel^2_{\partial D\times U,\theta}\\
			&<+\infty,
		\end{aligned}
	\end{equation*}   
	i.e., $f(z,\cdot)\in A^2(U,\gamma)$ for any $z\in D$. 
	Then
	\begin{equation}\label{3:E4}
		K_{\partial D\times U,\theta}((\zeta,u),(\overline{z},\overline{w}))=K_{D,\lambda}(\zeta,\overline{z})B_{U,\gamma}(u,\overline{w}) 
	\end{equation}
	because of the following equality
	\begin{equation}       \begin{aligned}
			&\frac{1}{2\pi}\int_{\partial D\times U}f(\zeta,u)\overline{K_{D,\lambda}(\zeta,\overline{z})B_{U,\gamma}(u,\overline{w})}\lambda(\zeta)\gamma(u) d|\zeta|dV_U\\
			&=\int_{ U}\overline{B_{U,\gamma}(u,\overline{w})}\gamma(u)dV_U\frac{1}{2\pi}\int_{\partial D}f(\zeta,u)\overline{K_{D,\lambda}(\zeta,\overline{z})}\lambda(\zeta)|d\zeta|\\
			&=\int_{ U}f(z,u)\overline{B_{U,\gamma}(u,\overline{w})}\gamma(u)dV_U\\
			&=f(z,w)\\
			&=\frac{1}{2\pi}\int_{\partial D\times U}f(\zeta,u)\overline{K_{D\times U,\theta}((\zeta,u),(\overline{z},\overline{w}))}\lambda(\zeta)\gamma(u)|d\zeta|dV_U.
		\end{aligned}
		\nonumber
	\end{equation} 
	By Lemma \ref{Le-719}, we have 
	\begin{equation}\label{eq:0826b}
		\begin{aligned}
			&B_{D\times U,\eta}((z_0,u),(\overline{z_0},\overline{u}))=B_{D,\rho}(z_0)B_{U,\gamma}(u). 
		\end{aligned}
	\end{equation}
	
	Note that $B_{D\times U,\eta}((z_0,u),(\overline{z_0},\overline{u}))>0$. Following from equalities \eqref{3:E4} and \eqref{eq:0826b},
	Theorem \ref{main theorem} implies Theorem \ref{main1-theorem}.

	\section{Hardy spaces on $M\times U$ over $\partial M\times U$}\label{sec:3}
	In this section, we  give some properties about the space  $H^2_\kappa(M\times U, \partial M\times U)$ and  prove Theorem \ref{main2} and Theorem \ref{key-Theorem14}.
	
	\subsection{Some results about the space $H^2_\kappa(M\times U, \partial M\times U)$}
	Let $M:=\prod_{j=1}^{n}D_j$, where $D_j$ is a planar regular region with finite boundary components which are analytic Jordan curves (see \cite{saitoh,yamada}). Let $U$ be a domain in $\mathbb{C}^m$. Recall that $H^2(D_j)$ denotes the Hardy space on $D_j$ and there exists a  linear map 
	$$\gamma_j: H^2(D_j)\rightarrow L^2(\partial D_j)$$ and $\gamma_j(f)$ denotes the nontangential boundary values of $f$ a.e. on $\partial D_j$
	for any $f\in H^2(D_j)$.    
	
	Set $\kappa(z,u)=\rho(z)\times\gamma(u),$ where $\gamma(u)$ is an admissible weight on $U$, and $\rho(z)$ is a positive Lebesgue measurable function on $\partial M$ such that $\inf_{\partial M}\rho>0$. 
	Let $M_j=\prod_{1\le l\le n,l\not=j}D_l.$ The definitions of $H_\rho^2(M,\partial D_j\times M_j)$, $H_\rho^2(M,\partial M)$, $P_{\partial M}$ and $H^2_\kappa(M\times U, \partial M\times U)$ can be seen in Section \ref{sec:1.2}.

	We recall a lemma about $H^2_\rho(M,\partial M)$.
	
	\begin{Lemma}[see \cite{GY-Hardy and product}]
		\label{lower-evaluation}
		For any compact subset $K$ of $M$, there exists a positive constant $C_K$ such that
		\begin{equation}
			|f^*(z)| \leq C_K ||f||_{\partial M,\rho}
			\nonumber
		\end{equation}
		holds for any $z\in K$ and $f\in H^2_\rho(M,\partial M).$
	\end{Lemma}

	From now on, assume that $\rho$ is a function on $\partial M$ satisfying that
	$\rho\big|_{\partial D_j\times {M}_j}=g_jh_j$, where $g_j$ is  continuous and positive function on $\partial D_j$ and $h_j$ is a Lebesgue measurable function on ${M}_j$ with a postive lower bound on ${M}_j$.
	
	For $f\in H^2_\kappa(M\times U, \partial M\times U)$, 
	let $$\widetilde{W}_f(u):=\int_{\partial M}|f(\cdot,u)|^2\rho d\mu.$$ 
	
	\begin{Lemma}\label{Le-24-1}
		For any compact subset $K$ of $U$,
		there is a positive constant ${L(K)}$, independent of $f$, such that
		\begin{equation}\label{key-ine1}
			\widetilde{W}_f(u)\leq {L(K)}\parallel f\parallel^2_{\partial M\times U,\kappa}
		\end{equation}
		for any $ u\in K$.
	\end{Lemma}
	\begin{proof}
		By definition, we have
		$$\int_{\partial M}|f(z,u)|^2\rho d\mu=\sum\limits_{1\leq j\leq n}\frac{1}{2\pi}\int_{M_j}\int_{\partial D_j}|\gamma_j(f(w_j,\hat{w}_j,u))|^2\rho(w_j,\hat{w}_j)|dw_j|d\mu_j(\hat{w}_j)$$
		and
		\begin{equation*}
			\begin{split}
				\parallel f\parallel^2_{\partial M\times U,\kappa}&=\int_{\partial M\times U}|f(w,u)|^2\kappa(w,u)d\mu(w)dV_U\\
				&=\int_U\gamma(u)dV_U\int_{\partial M}|f(w,u)|^2\rho(w)d\mu(w)\\
				&=\int_U\gamma(u)dV_U\sum\limits_{1\leq j\leq n}\frac{1}{2\pi}\int_{M_j}\int_{\partial D_j}|\gamma_j(f(w_j,\hat{w}_j,u))|^2\rho(w_j,\hat{w}_j)|dw_j|d\mu_j(\hat{w}_j)
			\end{split}
		\end{equation*}
		for every $f\in H^2_\kappa(M\times U, \partial M\times U)$.  Thus, we only need to show that there exists  a positive $C(K)$, independent of $f$, such that
		\begin{equation*}
			\begin{split}
				&\int_{M_j}\int_{\partial D_j}|\gamma_j(f(w_j,\hat{w}_j,u))|^2\rho(w_j,\hat{w}_j)|dw_j|d\mu_j(\hat{w}_j)\\
				&\leq C(K)\int_U\gamma(u)dV_U\int_{M_j}\int_{\partial D_j}|\gamma_j(f(w_j,\hat{w}_j,u))|^2\rho(w_j,\hat{w}_j)|dw_j|d\mu_j(\hat{w}_j).
			\end{split}
		\end{equation*}  
		
		Let   $\left\{D_{j,l}\right\}$ be an increasing sequence of domains with smooth boundaries $\partial D_{j,l}$, such that $D_j=\bigcup_l D_{j,l}$. Take  $z_j\in D_{j,1}$ for any $j$. Let 
		$$\widetilde{W}(z_j,\hat{w}_j,u)=\frac{1}{2\pi}\int_{\partial D_j}|\gamma_j(f(w_j,\hat{w}_j,u))|^2\frac{\partial G_{D_j}(w_j,z_j)}{\partial\nu_{w_j}}
		|dw_j|$$
		and $$\widetilde{W}_l(z_j,\hat{w}_j,u)=\frac{1}{2\pi}\int_{\partial D_{j,l}}| f(w_j,\hat{w}_j,u)|^2\frac{\partial G_{D_{j,l}}(w_j,z_j)}{\partial\nu_{w_j}}
		|dw_j|.$$ 
		Then  $\widetilde{W}_l(z_j,\hat{w}_j,u)$ increasingly tends to $\widetilde{W}(z_j,\hat{w}_j,u)$ as ${l\rightarrow+\infty}$ as $|f(z,u)|^2$ is plurisubharmonic with resprect to $z$, where  $G_{D_{j,l}}$ is the Green function on $D_{j,l}$.
		For  any  compact subset $K$  of $U$,  there exists a constant $C_1(K)$, independent of $(w_j,\hat{w}_j)$ and $f$, such that
		$$\sup\limits_{u\in K}|f(w_j,\hat{w}_j,u)|^2\leq C_1(K)\int_U|f(w_j,\hat{w}_j,u)|^2\gamma(u)dV_U$$
		as $\gamma$ is an admissible weight and $f(z,u)\in \mathcal{O}(U)$ for each fixed $z\in M$.
		It follows from monotone convergence theorem that
		\begin{equation*}
			\begin{split}
				&\frac{1}{2\pi}\int_{M_j}\int_{\partial D_j}|\gamma_j(f(w_j,\hat{w}_j,u))|^2\frac{\partial G_{D_j}(w_j,z_j)}{\partial \nu_{w_j}}|dw_j|h_j d\mu_j(\hat{w}_j) \\
				&=\lim\limits_{l\rightarrow+\infty}\int_{M_j}\frac{1}{2\pi}\int_{\partial D_{j,l}}|(f(w_j,\hat{w}_j,u))|^2\frac{\partial G_{D_{j,l}}(w_j,z_j)}{\partial\nu_{w_j}}
				|dw_j|h_jd\mu_j(\hat{w}_j) \\
				&\leq\lim\limits_{l\rightarrow+\infty} C_1(K)\frac{1}{2\pi}\int_U\gamma(u)dV_U\int_{M_j}\int_{\partial D_{j,l}}|f(w_j,\hat{w}_j,u)|^2\frac{\partial G_{D_{j,l}}(w_j,z_j)}{\partial\nu_{w_j}}|dw_j| h_jd\mu_j(\hat{w}_j)\\
				&=C_1(K)\frac{1}{2\pi}\int_U\gamma(u)dV_U\int_{M_j}\int_{\partial D_{j}}|\gamma_j(f(w_j,\hat{w}_j,u))|^2\frac{\partial G_{D_j}(w_j,z_j)}{\partial \nu_{w_j}}|dw_j| h_jd\mu_j(\hat{w}_j).
			\end{split}
		\end{equation*}  
		There exist two positive constants $l_j$ and $L_j$ such that 
		$$l_j\leq\frac{g_j}{\frac{\partial G_{D_j}(\cdot,z_j)}{\partial \nu_\zeta}}\leq L_j$$ on $\partial D_j$.
		Thus, we obtain 
		\begin{equation*}
			\begin{split}
				&\frac{1}{2\pi}\int_{M_j\times \partial D_j}|\gamma_j(f(w_j,\hat{w}_j,u))|^2g_jh_jd\mu_j(\hat{w}_j)\\
				&\leq \frac{L_j}{2\pi}\int_{M_j\times \partial D_j}|\gamma_j(f(w_j,\hat{w}_j,u))|^2\frac{\partial G_{D_j}(w_j,z_j)}{\partial \nu_{w_j}}h_j|dw_j|d\mu_j(\hat{w}_j)\\
				&\leq C_1(K)\frac{L_j}{2\pi}\int_U\gamma(u)dV_U\int_{M_j}\int_{\partial D_{j}}|\gamma_j(f(w_j,\hat{w}_j,u))|^2\frac{\partial G_{D_j}(w_j,z_j)}{\partial \nu_{w_j}}|dw_j| h_jd\mu_j(\hat{w}_j)\\
				&\leq C_1(K)\frac{L_j}{2\pi l_j}\int_U\gamma(u)dV_U\int_{M_j}\int_{\partial D_j}|\gamma_j(f(w_j,\hat{w}_j,u))|^2g_jh_j|dw_j|d\mu_j(\hat{w}_j).
			\end{split}
		\end{equation*}
		Hence, inequality (\ref{key-ine1}) holds.
	\end{proof}

	Using Lemma \ref{Le-24-1}, we give a convergence property.

	\begin{Lemma}\label{key-Coro2}
		Let $\left\{f_n\right\}$ be a Cauchy sequence in $ H^2_\kappa(M\times U, \partial M\times U)$. Then for every fixed $u\in U$, there exists an $f(\cdot,u)\in H^2_{\rho}(M,\partial M)$ such that $f_n(\cdot,u)\rightarrow f^*(\cdot,u)$ uniformly on every compact subset of $M$ and
		\begin{equation}\nonumber
			\lim\limits_{k\rightarrow\infty}\parallel P_{\partial M}^{-1}f_k(\cdot,u)-f(\cdot,u)\parallel_{\partial M,\rho}=0.
		\end{equation}
	\end{Lemma}
	\begin{proof}
		For each compact subset $K\subset U$, it follows from Lemma \ref{Le-24-1} that there exists a constant $L(K)>0$ such that
		$$\sup_{u\in K}\int_{\partial  M}|P_{\partial M}^{-1}(f_n(\cdot,u))-P_{\partial M}^{-1}(f_m(\cdot,u))|^2\rho d\mu\leq L(K)||f_n-f_m||^2_{\partial M\times U,\kappa},$$
		which implies that
		$\left\{P^{-1}_{\partial M}(f_n(\cdot,u))\right\}$ is a Cauchy sequence in $H_\rho^2(M,\partial M)$.
		Note that $H_\rho^2(M,\partial M)$ is a Hilbert space. Then  there exists an $f(\cdot,u)\in H_\rho^2(M,\partial M)$ such that    
		$$\lim\limits_{n\rightarrow+\infty}||P_{\partial M}^{-1}(f_n(\cdot,u)-f(\cdot,u))||_{\partial M,\rho}=0$$
		and it follows from Lemma \ref{lower-evaluation}  that  $f_n(\cdot,u)\rightarrow f^*(\cdot,u)$ uniformly on every compact subset of $M$.
	\end{proof}

	In the following, we show that  for each fixed $(z,u)\in M\times U$, the evaluation map from $ H^2_\kappa(M\times U, \partial M\times U)$ to $\mathbb{C}$ defined by:
	$f\mapsto f(z,u)$    is a bounded linear map. 
	\begin{Lemma}\label{evaluation-ine}
		For every $f\in H^2_\kappa(M\times U, \partial M\times U)$ and $K=K_1\times K_2$, where $K_1$ and $K_2$ are compact subsets of $M$ and $U$ respectively, there exists a constant $C(K)$, independent of $f$, such that
		\begin{equation}\label{key-ine3}
			|f(z,u)| \leq C(K)||f||_{\partial M\times U,\kappa}
		\end{equation}
		for every $(z,u)\in K.$
	\end{Lemma}
	\begin{proof}
		As $P_{\partial M}^{-1}(f(\cdot,u))\in H^2_\rho(M,\partial M)$ for fixed $u$, by Lemma \ref{lower-evaluation}, there exists a constant $C(K_1)$ such that
		\begin{equation*}
			\begin{split}
				|f(z,u)|^2 &\leq C(K_1)||P_{\partial M}^{-1}(f(\cdot,u))||^2_{\partial M,\rho}\\
				&\leq C(K_1)L(K_2)||f||^2_{\partial M\times U,\kappa}  \qquad(\mbox{by Lemma \ref{Le-24-1}}).
			\end{split} 
		\end{equation*}
		Hence, inequality (\ref{key-ine3}) holds with $C(K)=\sqrt{C(K_1)L(K_2)}.$
	\end{proof}

	Now, we prove that $H^2_\kappa(M\times U, \partial M\times U)$ is a Hilbert space with the following inner product:
	\begin{equation}
		\ll f,g\gg_{\partial M\times U,\kappa} 
		=\int_{\partial M\times U}P_{\partial M}^{-1}(f(z,u) )\overline{P_{\partial M}^{-1}(g(z,u))}\kappa(z,u)d\mu(z) dV_U
		\nonumber
	\end{equation}
	for $f,g\in H^2_\kappa(M\times U, \partial M\times U).$
	
	\begin{Proposition}\label{completion}
		$H^2_\kappa(M\times U, \partial M\times U)$ is a Hilbert space.
	\end{Proposition} 
	\begin{proof}It suffices to prove the any  Cauchy sequence in $H^2_\kappa(M\times U, \partial M\times U)$ has  a limit in $H^2_\kappa(M\times U, \partial M\times U)$. Let $\left\{f_n\right\}$ be a Cauchy sequence in $H^2_\kappa(M\times U, \partial M\times U)$. By Lemma \ref{key-Coro2}, there exists a function $f$ defined on $M\times U$ such that $f(z,u)$ is holomorphic on $M$,  $P_{\partial M}^{-1}(f(\cdot,u))\in H^2_\rho(M,\partial M)$ for each fixed $u\in U$, and
		\begin{equation}
			\label{eq:0827a}
			\lim_{n\rightarrow+\infty}\|P_{\partial M}^{-1}(f_n(\cdot,u))-P_{\partial M}^{-1}(f(\cdot,u))\|_{\partial M,\rho}=0.
		\end{equation}
		It suffices to show that $f(z,\cdot)$ is holomorphic on $U$ for any fixed $z\in M$ and $||f-f_n||_{\partial M\times U,\kappa} \rightarrow 0 $ when $n\rightarrow \infty.$ 
		
		For the former, by Lemma \ref{evaluation-ine}, we know that $\left\{f_n\right\}$ converges to a  $g(z,u)\in \mathcal{O}(M\times U)$ uniformly on each compact subset of $M\times U$. As $f(z,u)$ and $g(z,u)$ both are the limits of $\left\{f_n(z,u)\right\}$ for fixed $u\in U$, we have $f=g$, which implies that $f(z,u)$ is holomorphic on $M \times U$. In the following, we denote $P_{\partial M}^{-1}(f_n(z,u))$ and $P_{\partial M}^{-1}(f(z,u))$ by $f_n(z,u)$ and $f(z,u)$ respectively for simplicity.
		
		For the latter, let
		$$W_n(u):=\int_{\partial  M}|f_n(z,u)|^2\rho d\mu$$ for $n=1,2,\cdots,$ 
		and 
		$$W(u):=\int_{\partial  M}|f(z,u)|^2\rho d\mu.$$
		As $\left\{f_n\right\}$ is a Cauchy sequence, there exists a constant $L$ such that
		$$||f_n||^2_{\partial M\times U,\kappa} \leq L$$ 
		for arbitrary $n\geq 1$.  
		Thus, we have
		\begin{equation}\label{estimate}
			\begin{split}
				&\bigg(\int_U\int_{\partial  M}\bigg||f_n(z,u)|+|f_m(z,u)|\bigg|^2\rho(z)\gamma(u)d\mu(z)dV_U\bigg)^{1/2}\\
				&\leq \bigg(\int_U\int_{\partial  M}2(|f_n(z,u)|^2+|f_m(z,u)|^2)\rho(z)\gamma(u)d\mu(z)dV_U\bigg)^{1/2}\\
				&\leq {2}\sqrt{L}.
			\end{split}
		\end{equation}
		By equality \eqref{eq:0827a}, we get
		\begin{equation}\label{W-con}
			\lim\limits_{n\rightarrow+\infty}W_n(u)=W(u).
		\end{equation}     
		Besides, we have
		\begin{equation}\nonumber
			\begin{split}
				&\int_U|W_n(u)-W(u)|\gamma(u)dV_U\\ &=\int_U\bigg|\int_{\partial  M}\bigg(|f_n(z,u)|^2-|f(z,u)|^2\bigg)\rho(z)d\mu(z)\bigg|\gamma(u)dV_U\\
				&\leq\liminf\limits_{m\rightarrow\infty}\int_U\bigg|\int_{\partial  M}\bigg(|f_n(z,u)|^2-|f_m(z,u)|^2)\bigg)\rho(z)d\mu(z)\bigg|\gamma(u)dV_U\\
				&\leq \liminf\limits_{m\rightarrow\infty}\int_U\int_{\partial  M}\bigg||f_n(z,u)|^2-|f_m(z,u)|^2\bigg|\rho(z) d\mu(z)\gamma(u)dV_U\\
				&=\liminf\limits_{m\rightarrow\infty}\int_U\int_{\partial  M}\bigg||f_n(z,u)|-|f_m(z,u)|\bigg|(|f_n(z,u)|+|f_m(z,u)|)\rho(z)\gamma(u)d\mu(z)dV_U\\
				&\leq \liminf\limits_{m\rightarrow\infty}\bigg(\int_U\int_{\partial  M}|d|^2\rho(z)\gamma(u)d\mu(z)dV_U\bigg)^{1/2}
				\bigg(\int_U\int_{\partial  M}|a|^2\rho(z)\gamma(u)d\mu(z)dV_U\bigg)^{1/2} \\
				&\leq {2}\sqrt{L}\liminf\limits_{m\rightarrow\infty}||f_n-f_m||_{\partial M\times U,\kappa}\quad\mbox{(by inequality (\ref{estimate}))},
			\end{split}
		\end{equation}
		where $a=|f_n(z,u)|+|f_m(z,u)|,d=|f_n(z,u)|-|f_m(z,u)|$, the first $``\leq"$ holds because of $ \mbox{equality (\ref{W-con}) and Fatou's Lemma}$,  and the third $``\leq"$ holds due to Cauchy-Schwarz inequality. It follows that
		$$\lim\limits_{n\rightarrow\infty}\int_U|W_n(u)-W(u)|\gamma(u)dV_U=0,$$ which implies that
		\begin{equation}\label{W-equality}
			\lim\limits_{n\rightarrow\infty}\int_UW_n(u)\gamma(u)dV_U=\int_UW(u)\gamma(u)dV_U<+\infty.
		\end{equation}
		Let $$g_m(u)=\int_{\partial  M}\bigg|f(z,u)-f_m(z,u)\bigg|^2\rho(z)d\mu(z).$$
		Then it follows from equality (\ref{eq:0827a}) that
		$$\lim\limits_{m\rightarrow\infty}g_m(u)=0.$$
		On the other hand, we know that
		
		\begin{equation*}
			\begin{split}
				g_m(u)&=\int_{\partial  M}\bigg|f(z,u)-f_m(z,u)\bigg|^2\rho(z)d\mu(z)\\
				&\leq 2\int_{\partial  M}(|f(z,u)|^2+|f_n(z,u)|^2)\rho(z)d\mu(z)\\
				&=2(W_m(u)+W(u)),
			\end{split}
		\end{equation*}
		which  indicates that $$2(W_m(u)+W(u))-g_m(u)\geq 0.$$
		Then by Fatou's Lemma, we obtain
		\begin{equation*}
			\begin{split}
				&\int_U\liminf\limits_{m\rightarrow\infty}(2(W_m(u)+W(u))-g_m(u))\gamma(u)dV_U\\
				&\leq \liminf\limits_{m\rightarrow\infty}\int_U(2(W_m(u)+W(u))-g_m(u))\gamma(u)dV_U, 
			\end{split}
		\end{equation*} 
		which implies that
		$$\limsup\limits_{m\rightarrow\infty}\int_Ug_m(u)\gamma(u)dV_U=0$$
		with the help of equality (\ref{W-equality}).
		Namely,
		$$\lim\limits_{m\rightarrow\infty}||f-f_m||_{\partial M\times U,\kappa}=0.$$
		The proof is finished.
	\end{proof}

	We define a weighted  kernel $K_{\partial M\times U,\kappa}((\cdot,\cdot),(\overline{z},\overline{u}))$ on $M\times U$. 
	By Lemma \ref{evaluation-ine}, we know that the evaluation functional is bounded.
	Proposition \ref{completion} shows that $H^2_\kappa(M\times U, \partial M\times U)$ is a Hilbert space with the inner product $\ll \cdot,\cdot\gg_{\partial M\times U,\kappa}$. 
	Hence, by  Riesz Representation Theorem, for every $(z,w)\in M\times U$, there is a unique function $K_{\partial M\times U,\kappa}((\cdot,\cdot),(\overline{z},\overline{w}))\in H^2_\kappa(M\times U, \partial M\times U)$, with the reproducing property
	\begin{equation*}\nonumber
		f(z,w)=\int_{\partial M\times U}P_{\partial M}^{-1}(f(z,u))\overline{P_{\partial M}^{-1}(K_{\partial M\times U,\kappa}((\zeta,u),(\overline{z},\overline{w})))}\kappa(\zeta,u)d\mu(\zeta)dV_U 
	\end{equation*}
	for every $f\in H^2_\kappa(M\times U, \partial M\times U)$.
	For simplicity, denote 
	$$K_{\partial M\times U,\kappa}({z},{u}):=K_{\partial M\times U,\kappa}((z,u),(\overline{z},\overline{u})).$$
	
	Finally, we give the following decomposition fomula.

	\begin{Proposition}\label{Pro-28}
		$K_{\partial M\times U,\kappa}((\zeta,u),(\overline{z},\overline{w}))=K_{\partial M,\rho}(\zeta,\overline{z})B_{U,\gamma}(u,\overline{w}).$        
	\end{Proposition}
	\begin{proof}
		Let $f\in H^2_\kappa(M\times U, \partial M\times U).$
		For $P_{\partial M}^{-1}(f(\cdot, u))\in H^2_\rho(M,\partial M)$, it follows from Lemma \ref{lower-evaluation} that for any compact $ K\subset M,$ there exists a constant $C_K$ such that
		$$\sup\limits_{z\in K}|f(z,u)|^2\leq C_K||P_{\partial M}^{-1}(f(\cdot, u))||^2_{\partial M,\rho}.$$
		Therefore, we have
		\begin{equation*}
			\begin{split}
				\int_U|f(z,u)|^2\gamma(u)dV_U \leq C_K\int_{\partial M\times U}|P_{\partial M}^{-1}(f(z,u))|^2\gamma(u)\rho(z) d\mu dV_U
				<+\infty.
			\end{split}
		\end{equation*}
		Hence,  $f(z,\cdot)\in A^2(U,\gamma)$ for any fixed $z\in M$.
		Then, we have        \begin{equation}
			\begin{split}
				&\int_{\partial M\times U}P_{\partial M}^{-1}(f(\zeta,u))\overline{K_{\partial M,\rho}(\zeta,\overline{z})B_{U,\gamma}(u,\overline{w})}\rho(\zeta) \gamma(u)d\mu dV_U \\
				&=\int_U\overline{B_{U,\gamma}(u,\overline{w})}\gamma(u)dV_U\int_{\partial M}P_{\partial M}^{-1}(f(\zeta,u))\overline{K_{\partial M,\rho}(\zeta,\overline{z})}\rho(\zeta)  d\mu\\
				&=\int_Uf(z,u)\overline{B_{U,\gamma}(u,\overline{w})}\gamma(u)dV_U
				=f(z,w),
			\end{split}
			\nonumber
		\end{equation}
		which shows that Proposition \ref{Pro-28} holds.  
	\end{proof}

	The notations in this section follow from Section \ref{sec:1.2}.
	
	The following lemma will be used in the proof of Lemma \ref{basis-comp}.

	\begin{Lemma}[see \cite{GY-Hardy and product}]
		\label{partial caculation}
		Let $f\in L^2(\partial D_j\times M_j,d\mu)$. Assume that there is $f^*\in\mathcal{O}(M)$ such that $f^*(\cdot,\hat{w}_j)\in H^2(D_j)$ for a.e. $\hat{w}_j\in M_j$ and $f=\gamma_j(f^*)$ a.e. on $\partial D_j\times M_j$. Then we have $f^*(\cdot,\hat{w}_j)\in H^2(D_j)$ for any $\hat{w}_j\in M_j$.        
	\end{Lemma}

	Let $\gamma(u)$  be  an admissible weight on $U$, and let $\rho(z)$ be a positive Lebesgue measurable function on $\partial M$ satisfying $\inf_{\partial M}\rho>0$.  Denote $\kappa=\gamma\rho$ on $M\times U$. We give a  product property of the space $H^2_\kappa(M\times U,\partial M\times U)$.

	\begin{Lemma}\label{basis-comp}
		Let $\left\{e_m\right\}_{m\in\mathbb{N}_{>0}}$ be  a complete orthonormal basis of $H_\rho^2(M,\partial M)$   and $\left\{\tilde{e}_l\right\}_{l\in\mathbb{N}_{>0}}$ be a complete orthonormal basis of $A^2(U,\gamma)$. Then $\left\{P_{\partial M}(e_m)\tilde{e}_l\right\}_{m,l\in\mathbb{N}_{>0}}$ is a complete orthonormal basis of $H^2_\kappa(M\times U,\partial M\times U)$.
	\end{Lemma} 
	\begin{proof}
		It is clear that $\left\{P_{\partial M}(e_m)\tilde{e}_l\right\}_{m,l\in\mathbb{N}_{>0}}$ is an orthonormal sequence by Fubini's theorem. In the following, we prove its completeness.
		
		For any $f\in H^2_\kappa(M\times U,\partial M\times U)$ satisfying
		\begin{equation}\label{cau}
			\int_{\partial M\times U}P_{\partial M}^{-1}(f(z,u))\overline{e_m(z)\tilde{e}_l(u)}\rho(z)\gamma(u)=0
		\end{equation}
		for every $m,l\in\mathbb{N}_{>0}$,  denote  
		$$f_l(z) :=\int_UP_{\partial M}^{-1}(f(z,u))\overline{\tilde{e}_l(u)}\gamma(u)dV_U$$
		on  a.e. $\partial M$ for any $l>0$. Using Fubini's theorem, we know that $f_l\in L^2(\partial M,\rho)$. In the following, we  prove that $f_l(z)\in H^2_\rho(M,\partial M)$ . 
		
		By Fubini's theorem, we know that $f_l(\cdot,\hat{z}_j)\in L^2(\partial D_j)$ for almost all $\hat{z}_j\in M_j$.
		By Lemma \ref{1:L4}, we have
		\begin{equation}\label{holo}
			\begin{split}
				\int_{\partial D_j}\frac{f_l(z_j,\hat{z}_j)}{z_j-w_j}|dz_j|&=\int_{\partial D_j}\frac{1}{z_j-w_j}\int_U\gamma_j(f(\cdot,\hat{z}_j,u))\overline{\tilde{e}_l(u)}\gamma(u)dV_U|dz_j|\\
				&=\int_U\overline{\tilde{e}_l(u)}\gamma(u)\int_{\partial D_j}\frac{\gamma_j(f(\cdot,\hat{z}_j,u))}{z_j-w_j}|dz_j|dV_U\\
				&=\int_Uf(w_j,\hat{z}_j,u)\gamma(u)\overline{\tilde{e}_l(u)}dV_U
			\end{split}
		\end{equation}
		for almost all $\hat{z}_j\in M_j$.
		Denote 
		$$f_l^*(z_j,\hat{z}_j) :=\int_Uf(z_j,\hat{z}_j,u)\overline{\tilde{e}_l(u)}\gamma(u)dV_U.$$
		
		\emph{Claim $\#1$}: $f_l^*(z_j,\hat{z}_j)$ can be defined on $M$.
		
		\emph{Proof of Claim $\#1$}: By Lemma \ref{lower-evaluation}, for any compact subset $K\subset M$, there exists a constant $C_K>0$ such that 
		\begin{equation}\label{eq-33-1}
			\sup\limits_{z\in K}|f(z,u)|^2\leq C_K\int_{\partial M}|P_{\partial M}^{-1}(f(z,u))|^2\rho d\mu.
		\end{equation} 
		Then it follows from Cauchy-Schwarz inequality that
		$$\bigg(\int_U|f(z_j,\hat{z}_j,u)\overline{\tilde{e}_l(u)}|\gamma(u)dV_U\bigg)^2\leq \int_{ U}|f(z,u)|^2\gamma(u)dV_U\int_{U}|\overline{\tilde{e}_l(u)}|^2\gamma(u)dV_U.$$
		For $\int_{U}|\overline{\tilde{e}_l(u)}|^2\gamma(u)dV_U=1$, combining with  inequality (\ref{eq-33-1}), we get
		$$\sup\limits_{z\in K}\bigg(\int_U|f(z_j,\hat{z}_j,u)\overline{\tilde{e}_l(u)}|\gamma(u)dV_U\bigg)^2\leq C_K\int_{ \partial M\times U}|P_{\partial M}^{-1}(f(z,u))|^2\rho d\mu\gamma(u)dV_U<\infty.$$
		Hence, we finish the proof.
		
		\emph{Claim $\#2$}: $f_l^*(z)$ is holomorphic on $M$.
		
		\emph{Proof of Claim$\#2$}: Let
		$$g(z_1,\cdots,z_n)=\int_{U}\frac{\partial f}{\partial z_1}(z_1,\cdots,z_n,u)\overline{\tilde{e}_l(u)}\gamma(u)dV_U.$$
		Now we fix $w_1\in D_1$ and choose $r$ such that  $\Delta_r:=\left\{z\in D_1:|z-w_1|< r\right\}\subset D_1$. $\partial\Delta_r$ denotes the boundary of $\Delta_r$. Then it follows from Cauchy's formula that
		$$\frac{\partial f}{\partial z}(z,\hat{z}_1,u)=\frac{1
		}{2\pi\sqrt{-1}}\int_{\partial\Delta_r}\frac{f(w,\hat{z}_1,u)}{(w-z)^2}dw$$
		for any $z\in\Delta_r\subset D_1$. Then for any $z\in\Delta_{\frac{r}{2}}$, we have
		\begin{equation*}
			|\frac{\partial f}{\partial z}(z,\hat{z}_1,u)|\leq \frac{4}{r^2}\sup\limits_{z\in \Delta_{r}}|f(z,\hat{z}_1,u)|,
		\end{equation*}
		which implies that $g(z)$ is well defined for any fixed $\hat{z}_1\in M_1$ by inequality (\ref{eq-33-1}).
		Furthermore, we have
		\begin{equation*}
			\begin{split}
				f(z,\hat{z}_1,u)-f(w_1,\hat{z}_1,u)-\frac{\partial f}{\partial z}(w_1,\hat{z}_1,u)(z-w_1)=\frac{(z-w_1)^2
				}{2\pi\sqrt{-1}}\int_{\partial\Delta_r}\frac{f(w,\hat{z}_1,u)}{(w-z)(w-w_1)^2}dw
			\end{split}
		\end{equation*}
		for $|z-w_1|<\frac{r}{2}$. Then we get
		\begin{equation}\label{ine-41-1}
			\begin{split}
				|f(z,\hat{z}_1,u)-f(w_1,\hat{z}_1,u)-\frac{\partial f}{\partial z}(w_1,\hat{z}_1,u)(z-w_1)|\leq\frac{2}{r^3}|z-w_1|^2\sup\limits_{w\in \Delta_{r}}|f(w,\hat{z}_1,u)|
			\end{split}
		\end{equation}
		for $|w-z|\geq |w_1-w|-|z-w_1|\geq \frac{r}{2}$.   
		Combining with inequality (\ref{ine-41-1}) and Cauchy-Schwarz inequality, we have
		\begin{equation}\label{ine-41-2}
			\begin{split}
				&|f_l^*(z,\hat{z}_1)-f_l^*(w_1,\hat{z}_1)-(z-w_1)g(w_1,\hat{z}_1)|\\
				&=|\int_{U}\big(f(z,\hat{z}_1,u)-f(w_1,\hat{z}_1,u)-\frac{\partial f}{\partial z}(w_1,\hat{z}_1,u)(z-w_1)\big)\overline{\tilde{e}_l(u)}\gamma(u)dV_U|\\
				&\leq \int_{ U}|f(z,\hat{z}_1,u)-f(w_1,\hat{z}_1,u)-\frac{\partial f}{\partial z}(w_1,\hat{z}_1,u)(z-w_1)||\overline{\tilde{e}_l(u)}|\gamma(u)dV_U\\
				&\leq \bigg(\int_{ U}|f(z,\hat{z}_1,u)-f(w_1,\hat{z}_1,u)-\frac{\partial f}{\partial z}(w_1,\hat{z}_1,u)(z-w_1)|^2\gamma(u)dV_U\bigg)^{1/2}\\
				&\leq |z-w_1|^2\frac{2}{r^3}\bigg(\int_{U}\sup\limits_{w\in \Delta_{r}}|f(w,\hat{z}_1,u)|^2\gamma(u)dV_U\bigg)^{1/2}.
			\end{split}
		\end{equation}      
		By inequality (\ref{eq-33-1}), we know that the last term  in inequality (\ref{ine-41-2}) is finite, which implies that
		$$\frac{\partial f_l^*}{\partial z_1}(z_1,\cdots,z_n)=g(z)=\int_{U}\frac{\partial f}{\partial z_1}(z_1,\cdots,z_n,u)\overline{\tilde{e}_l(u)}\gamma(u)dV_U.$$
		Similarly, for $1\leq j\leq n$, we have
		$$\frac{\partial f_l^*}{\partial z_j}(z_1,\cdots,z_n)=g(z_1,\cdots,z_n)=\int_{U}\frac{\partial f}{\partial z_j}(z_1,\cdots,z_n,u)\overline{\tilde{e}_l(u)}\gamma(u)dV_U.$$   
		Hence, $f_l^*(z_1,\cdots,z_n)$ is holomorphic with respect to each variable $z_j$ for $1\leq j\leq n$. Then it follows from Hartogs' theorem that $f_l^*(z_1,\cdots,z_n)$ is holomorphic on $M$.   We finish the proof of  \emph{Claim$\#2$}.
		
		Besides, we have
		\[
		\int_{\partial D_j}f_l(z_j,\hat{z}_j)\phi(z_j)|dz_j|=\int_{ U}\overline{\tilde{e}_l(u)}\gamma(u)dV_U\int_{\partial D_j}P_{\partial M}^{-1}(f(z_j,\hat{z}_j,u))\phi(z_j)|dz_j|=0
		\]
		for any $\phi\in A(D_j)$  and almost all $\hat z_j\in M_j$ by $(c)$ of Lemma \ref{1:L4}.
		Then we have $f_l^*(\cdot,\hat{z}_j)\in H^2(D_j)$ and 
		$$\gamma_j(f_l^*(\cdot,\hat{z}_j))=f_l(\cdot,\hat{z}_j)$$
		for almost all $\hat{z}_j\in M_j$ by $(c)$ of Lemma \ref{1:L4}  and equality \eqref{holo}. Then it follows from Lemma \ref{partial caculation} that $f_l^*(\cdot,\hat{z}_j)\in H^2(D_j)$ for all $\hat{z}_j\in M_j$. By definition, we know that $f_l\in H^2_\rho(M,\partial M)$.
		
		Due to $\ll f_l(z),e_m(z)\gg_{\partial M,\rho}=0$ by equality (\ref{cau}), we know that
		$f_l(z)=0$ on a.e. $\partial M$. Furthermore, it follows from the injectivity of $\gamma_j$ that $f_l^*(z)=0$. Since $\left\{e_l(u)\right\}$ is a complete orthonormal basis of $A^2(U,\gamma)$, we get $f(z,u)=0$.    
		The proof is finished.    
	\end{proof}
	
	For any $\alpha=(\alpha_1,\cdots,\alpha_m)\in\mathbb{N}^m$ and any $\beta=(\beta_1,\cdots,\beta_m)\in\mathbb{N}^m,$ 
	we say $\alpha <\beta$   if $|\alpha|< |\beta|$ or $|\alpha|=|\beta|$ and there exists some $k$, $1\leq k\leq m$, such that $\alpha_1=\beta_1,\cdots,\alpha_{k-1}= \beta_{k-1},\;\alpha_k<\beta_k$. For $f\in\mathcal{O}(U)$, denote by $\alpha=ord_{u_0}(f)$ if $f=a_{\alpha}(u-u_0)^{\alpha}+\sum_{\alpha\geq\beta}b_{\beta}(u-u_0)^{\beta}$ with $a_{\alpha}\neq 0$. 
	
	\begin{Lemma}[see \cite{modules-at-boundary}]
		\label{A_basis}
		Let $u_0\in U$. There exists a countable complete orthonormal basis $\left\{f_\alpha(z)\right\}_{\alpha\in S_2}$ of $A^2(U,\gamma)$ such that $\alpha =ord_{u_0}(f_\alpha)$, where $S_2\subset\mathbb{N}^m$.   
	\end{Lemma} 
	
	\begin{proof}
		For the convenience of reader, we recall the proof.
		For fixed $\alpha\in\mathbb{N}^m$, denote 
		$$E^\alpha\;:=\left\{f\in A^2(U,\gamma):\;f^{(\alpha)}(u_0)=1,\;\mbox{and}\;f^{(\beta)}(u_0)=0,\;\forall\;\beta<\alpha\right\},$$
		and
		$$J_\alpha\;:=\left\{f\in A^2(U,\gamma):\;f^{(\beta)}(u_0)=0,\;\forall\;\beta<\alpha\right\}.$$
		Let $S_1\;:=\left\{\alpha\in\mathbb{N}^m\;:E^\alpha\;=\emptyset\right\}$ and $S_2\;:=\mathbb{N}^m-S_1$.
		Note that $\alpha\in S_1\Longleftrightarrow\;$  $f^{(\alpha)}(u_0)=0 \;\mbox{for any}\; f\in J_\alpha$.
		
		For $\alpha\in S_2$, let
		$$A_\alpha\;:=\inf\left\{||f||^2_{U,\gamma}:f\in E^\alpha\right\}.$$
		Choose $\left\{f_k\right\}\subset E^\alpha$ satisfying that $||f_k||^2_{U,\gamma}\;\rightarrow A_\alpha$ as $k\rightarrow\infty$. As $\gamma$ is an admissible weight on $U$, we know that $\left\{f_k\right\}$ is uniformly bounded on each compact subset of $U$ by $(2)$ of Lemma \ref{e-c}. It follows from Montel's Theorem that there is a subsequence of $\left\{f_k\right\}$ which converges uniformly on each compact subset of $U$. Without loss of generality, assume that $f_k$ uniformly converges to $g_\alpha$ on each compact subset of $U$. Then we get that $g_\alpha\in\mathcal{O}(U)$ and
		\begin{equation*}
			\begin{split}
				&g_\alpha^{(\alpha)}(u_0)=1;\\
				&g_\alpha^{(\beta)}(u_0)=0, \;\forall\beta<\alpha.
			\end{split}
		\end{equation*}   
		On the another hand, we know that
		$$\int_{ U}|g_\alpha|^2\gamma dV_U\leq\liminf\limits_{k\rightarrow\infty}\int_{ U}|f_k|^2\gamma dV_U=A_\alpha $$
		by Fatou's lemma.
		Now we find a function $g_\alpha\in A^2(U,\gamma)$ satisfying  $g_\alpha\in E^\alpha$ and $||g_\alpha||^2_{U,\gamma}=A_\alpha$. And we know that $g_\alpha \not\equiv 0$ and $A_\alpha>0$ for $g_\alpha^{(\alpha)}(u_0)=1$.

		In the following, we prove that $\left\{g_\alpha\right\}_{\alpha\in S_2}$ is an orthogonal subset of $A^2(U,\gamma)$. 
		
		We claim that  for any $h\in A^2(U,\gamma)$ satisfying  $h^{(\beta)}(u_0)=0$ for any $\beta\leq\alpha$, we have $\ll g_\alpha,h\gg_{U,\gamma}=0$, where $\alpha\in S_2.$
		
		\emph{Proof of  the claim}:
		Note that for any $c\in\mathbb{C}$, we have
		\begin{equation*}
			(ch+g_\alpha)^{(\alpha)}(u_0)=1,\;(ch+g_\alpha)^{(\beta)}(u_0)=0 \;\;\;\forall \beta<\alpha.
		\end{equation*}
		Hence, $ch+g_\alpha\in E^\alpha$. Therefore, $||ch+g_\alpha||_{U,\gamma}\geq ||g_\alpha||_{U,\gamma}$ holds for any $c\in\mathbb{C}$, which implies that $\ll g_\alpha,h\gg_{U,\gamma}=0$. (In fact, we have proven that $g_\alpha$ is unique. Otherwise, there exists another $\hat{g}_\alpha\in E^\alpha$ such that $||\hat{g}_\alpha||^2_{U,\gamma}=A_\alpha$. Then let $h=g_\alpha-\hat{g}_\alpha$. we see that
		$h^{(\beta)}(z_0)=0$ for any $\beta\leq\alpha$. So,
		$$\ll g_\alpha,h\gg_{U,\gamma}=\ll \hat{g}_\alpha,h\gg_{U,\gamma}=0,$$
		$$\ll h,h\gg_{U,\gamma}=0,$$ which implies that $\hat{g}_\alpha={g}_\alpha.$) 
		
		For  $\alpha\in S_2$, set $h_\alpha=\frac{g_\alpha}{||g_\alpha||_{U,\gamma}}$. For the present, we prove that $\left\{h_\alpha\right\}_{\alpha\in S_2}$ is a complete orhtonormal basis of $A^2(U,\gamma)$. We have proven that the set is orthonormal, so it remains to show that the set is complete.
		
		For any $f\in A^2(U,\gamma)$ satisfying that
		\begin{equation}\label{eq-34-1}
			\ll f,h_\alpha\gg_{U,\gamma}=0
		\end{equation}
		for all $\alpha\in S_2$,  we prove $f=0$. Toward a contradiction, we   assume that $f \not\equiv0$ and $ord_{z_0}f=\alpha_0$. Then there exists a  constant $c_0\neq0$ such that
		$$(f-c_0h_{\alpha_0})^{(\beta)}(u_0)=0,\quad\forall\beta\leq\alpha_0.$$
		It follows from the above claim that $\ll f-c_0h_{\alpha_0},h_{\alpha_0}\gg_{U,\gamma}=0$,
		which implies that $c_0=0$ by equality (\ref{eq-34-1}). Hence, we finish the proof of the completeness of $\left\{h_\alpha\right\}_{\alpha\in S_2}$.    
	\end{proof}

	Let $w_0=(z_0,u_0)\in M\times U$.  Let $b_0$ be a holomorphic funtion  on a neighborhood of $u_0$, and let $l_0=\prod_{1\leq j\leq n}l_j$, where
	$l_j$ is a holomorphic function on a neighborhood of $z_j\in D_j$, Denote  $\hat{h}_0=l_0b_0$ on a neighborhood of $w_0$.   
	
	Let  $\beta=(\beta^\prime,\beta^{\prime\prime})$, where $\beta^\prime=(\beta_1,\cdots,\beta_{n})$, $\beta_j= ord_{z_j}(l_j)$ for $1\leq j\leq n$ and $\beta^{\prime\prime}=(\beta_{n+1},\cdots,\beta_{n+m})=ord_{u_0}b_0$. 
	Let $\tilde{\beta}^\prime=(\tilde{\beta_1},\cdots,\tilde{\beta_n})\in\mathbb{N}^n$ and $\tilde{\beta}^{\prime\prime}=(\tilde{\beta}_{n+1},\cdots,\tilde{\beta}_{n+m})\in\mathbb{N}^{m}$, which satisfy that $\tilde{\beta_j}\geq\beta_j$ for any $1\leq j\leq n$ and $\beta^{\prime\prime}\leq \tilde{\beta}^{\prime\prime}$.      Denote 
	$$I_1:=\left\{(g,z_0)\in \mathcal{O}_{M,z_0}: g=\sum_{\alpha\in \mathbb{N}^{n}}b_\alpha(w-z_0)^\alpha \mbox{ near }z_0 \mbox{ s.t. } b_\alpha=0 \mbox{ for }\alpha\in L_{\tilde{\beta}^\prime}\right\}$$
	and
	$$I_2:=\left\{(g,u_0)\in \mathcal{O}_{U,u_0}: g=\sum_{\alpha\in \mathbb{N}^{m}}b_\alpha(u-u_0)^\alpha \mbox{ near }u_0 \mbox{ s.t. } b_\alpha=0 \mbox{ for }\alpha\in L_{\tilde{\beta}^{\prime\prime}}\right\},$$
	where $L_{\tilde{\beta}^\prime}=\left\{(\alpha_{1},\cdots,\alpha_{n})\in \mathbb{N}^{n}:\alpha_j\leq \tilde{\beta}_j\;\mbox{for}\;1\leq j\leq n \right\}$ and $L_{\tilde{\beta}^{\prime\prime}}=\left\{\alpha\in \mathbb{N}^{m}:\alpha\le \tilde\beta^{\prime\prime} \right\}$. 
	Denote 
	$$I^\prime=\left\{(g,w_0)\in\mathcal{O}_{M,\times U,w_0} : g=\sum_{\alpha\in \mathbb{N}^{n+m}}b_\alpha(w-w_0)^\alpha\;\mbox{near}\;w_0\;\mbox{s.t.}\;b_\alpha=0\;\mbox{for}\;\alpha\in L_{\tilde{\beta}}\right\},$$
	where  $L_{\tilde{\beta}}=\left\{(\alpha_1,\cdots,\alpha_{m+n})\in \mathbb{N}^{n+m} :\alpha_j\leq \tilde{\beta}_j\;\mbox{for}\,1\leq j\leq  n\;\&\;(\alpha_{n+1},\cdots,\alpha_{n+m})\leq \beta^{\prime\prime}\right\}$ and $\tilde{\beta}=(\tilde{\beta}^\prime,\tilde{\beta}^{\prime\prime})\in\mathbb{N}^{n+m}$. It is clear that $(\hat{h}_0,w_0)\notin I^\prime$.

	Let $\gamma$  be  an admissible weight on $U$, and $\tilde\rho$ be an admissible weight on $M$.  
	Denote 
	$$B_{U,\gamma}^{I_2,b_0}(u_0):=\frac{1}{\inf\left\{\int_U|f|^2\gamma: f\in\mathcal{O}(U)\;\&\; (f-b_0,u_0)\in I_2\right\}}.$$
	Similarly, we can define $B_{M\times U,\tilde{\rho}\gamma}^{I^\prime,\hat{h}_0}(w_0)$ and $B_{M,\tilde{\rho}}^{I_1,l_0}(z_0)$.
	We have the following decomposition formula.

	\begin{Lemma} \label{Le-44-1}
		\begin{equation}\label{Berg-decomp}
			B_{M\times U,\tilde{\rho}\gamma}^{I^\prime,\hat{h}_0}(w_0)=B_{M,\tilde{\rho}}^{I_1,l_0}(z_0)B_{U,\gamma}^{I_2,b_0}(u_0)
		\end{equation}
		holds.
	\end{Lemma}
	\begin{proof}
		Using Lemma \ref{A_basis},  there is a   
		complete orthonormal basis $\left\{g_\tau(u)\right\}_{\tau\in S_2}$  of $A^2(U,\gamma)$ 
		such that 
		$ord_{u_0}(g_\tau)=\tau$, where  $S_2\subset\mathbb{N}^m$. Let  $\left\{f_\alpha(z)\right\}_{\alpha\in S_2^\prime}$ be a   complete orthonormal basis  of $A^2(M,\tilde{\rho})$, where $S_2^\prime\subset\mathbb{N}^n$.
		Then by Lemma \ref{A-composite-basis}, we know that $\left\{f_\alpha (z)g_\tau(u)\right\}_{\alpha\in S_2^\prime,\tau\in S_2}$ is a complete orthonormal basis of $A^2(M\times U,\tilde{\rho}\gamma)$.
		
		Denote   
		\begin{equation*}
			\begin{split}
				&A:=\left\{\tau\in S_2:\beta^{\prime\prime}\leq\tau\leq\tilde{\beta}^{\prime\prime}\right\},\\
				&B:=\left\{f\in A^2(M\times U,\tilde{\rho}\gamma):(f-\hat{h}_0,w_0)\in I^\prime\right\},\\
				&C_1:=\left\{f\in A^2(M,\tilde{\rho}):(f-l_0,z_0)\in I_1\right\},\\
				&C_2:= \left\{f
				\in A^2(U,\gamma):(f-b_0,u_0)\in I_2\right\}.
			\end{split}
		\end{equation*}        
		We say that $\inf\left\{||f||^2_{M\times U,\rho\gamma} :f\in A^2(M\times U,\tilde{\rho}\gamma)\&(f-\hat{h}_0,w_0)\in I^\prime\right\}=+\infty$ if and only if $B=\emptyset$, i.e. $B_{M\times U,\tilde{\rho}\gamma}^{I^\prime,\hat{h}_0}(w_0)=0$. 
		
		If $B_{M\times U,\tilde{\rho}\gamma}^{I^\prime,\hat{h}_0}(w_0)=0$,  it is clear that  at least one of $B_{M,\tilde{\rho}}^{I_1,l_0}(z_0)$ and $B_{U,\gamma}^{I_2,b_0}(u_0)$ is zero.  Hence 
		the right member of equality (\ref{Berg-decomp}) is also zero.

		Now assume that $B_{M\times U,\tilde{\rho}\gamma}^{I^\prime,\hat{h}_0}(w_0)>0$, we know that   
		$ B\neq\emptyset$.
		We claim that for any $f\in B$, there exist $f_0\in C_1$ and $f_1\in C_2$ such that
		$$||f||_{M\times U,\tilde{\rho}\gamma}^2\geq ||f_1||^2_{U,\gamma}||f_0||^2_{M,\tilde{\rho}}.$$
		
		We prove the claim as follows.
		Since $\left\{f_\alpha (z)g_\tau(u)\right\}_{\alpha\in S_2^\prime,\tau\in S_2}$ is a complete orthonormal basis of $A^2(M\times U,\tilde{\rho}\gamma)$, we have 
		$$f=\sum_{\tau\in S_2}f_\tau g_\tau,$$ 
		where $f_\tau\in A^2(M,\tilde{\rho})$.
		As $(f-\hat{h}_0,w_0)\in I^\prime$, there is a constant $c_\tau$ such that
		$$(f_\tau-c_\tau l_0,z_0)\in I_1$$ 
		for $\tau\in A$. Note that $\sum_{\tau\in A}|c_\tau|>0$. There exsits $f_0\in A^2(M,\tilde{\rho})$ such that $(f_0-l_0,z_0)\in I_1$ and
		$$||f_0||_{M,\tilde{\rho}}^2=\inf\left\{||f||_{M,\tilde{\rho}}^2:f\in A^2(M,\tilde{\rho})\&\;(f-l_0,z_0)\in I_1\right\}.$$ 
		Thus, we know that $(f-\sum_{\tau\leq\tau_0}c_\tau g_\tau f_0,w_0)\in I^\prime$ and
		\begin{equation}\label{38-4-ine-1}
			\begin{split}
				||f||_{M\times U,\tilde{\rho}\gamma}^2&\geq ||\sum_{\tau\in A}f_\tau g_\tau ||^2_{M\times U,\tilde{\rho}\gamma}\\
				&=\sum_{\tau\in A}||f_\tau  ||^2_{M,\tilde{\rho}}\\
				&\geq \sum_{\tau\in A}|c_\tau|^2||f_0  ||^2_{M,\tilde{\rho}}. 
			\end{split}
		\end{equation}
		Let 
		$$f_1:=\sum_{\tau\in A}c_\tau g_\tau.$$
		Then we have $(f_1-b_0,u_0)\in I_2$ and $f_1
		\in A^2(U,\gamma) $. Then it follows from inequality (\ref{38-4-ine-1}) that
		$$||f||_{M\times U,\tilde{\rho}\gamma}^2\geq ||f_1||^2_{U,\gamma}||f_0||^2_{M,\tilde{\rho}}.$$
		It follows  that
		\begin{equation*}
			\inf\left\{||f||_{M\times U,\tilde{\rho}\gamma}^2:f\in B\right\}\geq \inf\left\{||gh||_{M\times U,\tilde{\rho}\gamma}^2:g\in C_1  \& h
			\in C_2\right\}.
		\end{equation*}
		
		By definitions, we have 
		$$\inf\left\{||f||_{M\times U,\tilde{\rho}\gamma}^2:f\in B\right\}\leq \inf\left\{||gh||_{M\times U,\tilde{\rho}\gamma}^2:g\in C_1  \& h
		\in C_2\right\}.$$
		Hence, we get the equality
		$$\inf\left\{||f||_{M\times U,\tilde{\rho}\gamma}^2:f\in B\right\}= \inf\left\{||gh||_{M\times U,\tilde{\rho}\gamma}^2:g\in C_1  \& h
		\in C_2\right\}.$$
		Thus, equality (\ref{Berg-decomp}) holds.   
	\end{proof}

	The definitions of $K_{\partial M,\rho}^{I_1,l_0}(z_0)$ and $K_{\partial M\times U,\rho\gamma}^{I^\prime,\hat{h}_0}(w_0)$ can be seen in Section \ref{sec:1.2}.  We have  the following decomposition formula.
	\begin{Proposition}\label{pf of 1}
		\begin{equation}\label{1}
			K_{\partial M\times U,\rho\gamma}^{I^\prime,\hat{h}_0}(w_0)=K_{\partial M,\rho}^{I_1,l_0}(z_0)B^{I_2,b_0}_{U,\gamma}(u_0)
		\end{equation}
		holds.
	\end{Proposition}
	\begin{proof}
		The proof is similar to the proof of Lemma \ref{Le-44-1}.
		Let $\left\{e_\alpha\right\}_{\alpha\in S_1}$ be  a complete orthonormal basis of $H_\rho^2(M,\partial M)$, where  $S_1$ is a subset of $\mathbb{N}^n$. Let $\left\{{g}_\tau\right\}_{\tau\in S_2}$ be a complete orthonormal basis of $A^2(U,\gamma)$ such that $\tau=ord_{u_0}({g}_\tau)$ by Lemma \ref{A_basis}, where $S_2$ is a subset of $\mathbb{N}^m$. Then  we know that $\left\{e_\alpha^*{g_\tau}\right\}_{\alpha\in S_1,\tau\in S_2}$ is a complete orthonormal basis of $H^2_\kappa(M\times U,\partial M\times U)$ by Lemma \ref{basis-comp}.
		
		Denote 
		\begin{equation*}
			\begin{split}
				&A:=\left\{\tau\in S_2:\beta^{\prime\prime}\leq\tau\leq\tilde{\beta}^{\prime\prime}\right\},\\
				&B:=\left\{f\in H^2_\kappa( M\times U,\partial M\times U):(f-\hat{h}_0,w_0)\in I^\prime\right\},\\
				&C_1:=\left\{f\in H^2_\rho(M,\partial M):(P_{\partial M}(f)-l_0,z_0)\in I_1\right\},\\
				&C_2:=\left\{g\in A^2(U,\gamma):(f-b_0,u_0)\in I_2\right\}.
			\end{split}
		\end{equation*}         
		We say that $\inf\left\{||f||_{\partial M\times U,\kappa}: f\in B\right\}=+\infty$ if and only if $B=\emptyset$, i.e. $K_{\partial M\times U,\kappa}^{I^\prime,\hat{h}_0}(w_0)=0$.        
		
		If $K_{\partial M\times U,\kappa}^{I^\prime,\hat{h}_0}(w_0)=0$, it is clear that at least one of $B_{U,\gamma}^{I_2,b_0}(u_0)$ and $K_{\partial M,\rho}^{I_1,l_0}(z_0)$ equals zero.  Hence, the right member of equality (\ref{1}) is also zero.
		
		Now  assume that $K_{\partial M\times U,\kappa}^{I^\prime,\hat{h}_0}(w_0)>0$. 
		We claim that for any $f\in B$, there exist $f_1\in C_1$ and $f_2\in C_2$ such that
		$$||f||^2_{\partial M\times U,\kappa}\geq ||f_1||^2_{\partial M,\rho}||f_2||^2_{U,\gamma}.$$
		
		We prove the claim in the following.       
		As $\left\{e_\alpha^*{g_\tau}\right\}_{\alpha\in S_1,\tau\in S_2}$ is a complete orthonormal basis of $H^2_\kappa(M\times U,\partial M\times U)$, we have 
		$$f=\sum_{\tau}f_\tau g_\tau,$$ 
		where $P_{\partial M}^{-1}(f_\tau)\in H^2_\rho(M,\partial M)$.    Since $(f-\hat{h}_0,w_0)\in I^\prime$, there exists a constant $c_\tau$ such that
		$$(f_\tau-c_\tau l_0,z_0)\in I_1$$
		for any $\tau\in A$. Note that $\sum_{\tau\in A}|c_\tau|>0$. There exsits an $f_1\in H^2_\rho(M,\partial M)$ such that $(P_{\partial M}(f_1)-l_0,z_0)\in I_1$ and
		$$||f_1||_{\partial M,\rho}^2=\inf\left\{||f||_{\partial M,\rho}^2:f\in H^2_\rho(M,\partial M)\&\;(P_{\partial M}(f)-l_0,z_0)\in I_1\right\}.$$ 
		Thus, we know that $(f-\sum_{\tau\in A}c_\tau g_\tau P_{\partial M}(f_1),w_0)\in I^\prime$ and
		\begin{equation}\label{40-4-ine-1}
			\begin{split}
				||f||_{\partial M\times U,\kappa}^2&\geq ||\sum_{\tau\in A}P_{\partial M}^{-1}(f_\tau) g_\tau ||^2_{\partial M\times U,\kappa}\\
				&=\sum_{\tau\in A}||P_{\partial M}^{-1}(f_\tau)  ||^2_{\partial M,\rho}\\
				&\geq \sum_{\tau\in A}|c_\tau|^2||f_1  ||^2_{\partial M,\rho}. 
			\end{split}
		\end{equation}
		Let 
		$$f_2:=\sum_{\tau\in A}c_\tau g_\tau.$$
		Then we have $f_2\in C_2$. Then it follows from inequality (\ref{40-4-ine-1}) that
		$$||f||_{\partial M\times U,\kappa}^2\geq ||f_2||^2_{U,\gamma}||f_1||^2_{\partial M,\rho}.$$
		It follows  that
		\begin{equation*}
			\inf\left\{||f||_{\partial M\times U,\kappa}^2:f\in B\right\}\geq \inf\left\{||gh||_{\partial M\times U,\kappa}^2:g\in C_1  \& h
			\in C_2\right\}.
		\end{equation*}
		
		By definitions, we have
		$$\inf\left\{||f||_{\partial M\times U,\kappa}^2:f\in B\right\}\leq \inf\left\{||gh||_{\partial M\times U,\kappa}^2:g\in C_1  \& h
		\in C_2\right\}.$$
		Hence, we get the equality
		$$\inf\left\{||f||_{\partial M\times U,\kappa}^2:f\in B\right\}= \inf\left\{||gh||_{\partial M\times U,\kappa}^2:g\in C_1  \& h
		\in C_2\right\}.$$
		Thus, equality (\ref{1}) holds.      
	\end{proof}

	\subsection{Proof of Theorem \ref{main2}}
	In this section, we prove Theorem \ref{main2}.
	
	By Lemma \ref{Le-719} ($X_1=M,X_2=U,\varphi_1=\tilde{\rho},\varphi_2=\gamma$),
	we have the following Bergman decomposition formula
	\begin{equation}\nonumber
		B_{M\times U,\eta}((\zeta,u),(\overline{z},\overline{w}))=B_{M,\tilde{\rho}}(\zeta,\overline{z})B_{U,\gamma}(u,\overline{w}).
	\end{equation}
	Combining with the equality in Proposition \ref{Pro-28}
	\begin{equation}
		\begin{split}
			K_{\partial M\times U,\kappa}((\zeta,u),(\overline{z},\overline{w}))=K_{\partial M,\rho}(\zeta,\overline{z})B_{U,\gamma}(u,\overline{w}),
		\end{split}
		\nonumber
	\end{equation}
	we know that the inequality 
	\begin{equation}
		\bigg(\int_{0}^{+\infty}c(t)e^{-t}dt\bigg)\pi B_{M\times U,\eta}((z_0,u_0),(\overline{z_0},\overline{u_0}))\leq K_{M\times U,\kappa}((z_0,u_0),(\overline{z_0},\overline{u_0}))
		\label{eq:0827c}
	\end{equation}
	holds if and only if  the inequality  
	\begin{equation}
		K_{\partial M,\rho}(z_0)\geq\bigg(\int_{0}^{+\infty}c(t)e^{-t}dt\bigg)\pi B_{\tilde{\rho}}(z_0)
		\label{eq:0827d}
	\end{equation}
	holds by  the assumption that $B_{M\times U,\eta}((z_0,u_0),(\overline{z}_0,\overline{u}_0))>0$. 
	Furthermore, the $``="$ in inequality (\ref{eq:0827c}) holds if and only if the $``="$ holds in inequality (\ref{eq:0827d}) holds. We complete the proof by Theorem \ref{key-Theorem}.

	\subsection{Proof of Theorem \ref{key-Theorem14}} We follow the notations in Lemma \ref{Le-44-1} and Proposition \ref{pf of 1}.
	Using  the equality (see Proposition \ref{pf of 1})
	\begin{equation}
		K_{\partial M\times U,\rho\gamma}^{I^\prime,\hat{h}_0}(w_0)=K_{\partial M,\rho}^{I_1,l_0}(z_0)B^{I_2,b_0}_{U,\gamma}(u_0)
		\nonumber
	\end{equation}
	and the equality (see Lemma \ref{Le-44-1})
	\begin{equation}
		B_{M\times U,\tilde{\rho}\gamma}^{I^\prime,\hat{h}_0}(w_0)=B_{M,\tilde{\rho}}^{I_1,l_0}(z_0)B_{U,\gamma}^{I_2,b_0}(u_0),
		\nonumber   
	\end{equation}
	we know that the 
	inequality 
	\begin{equation}
		K_{\partial M\times U,\rho\gamma}^{I^\prime,\hat{h}_0}(w_0)\geq\left(\int_{0}^{+\infty}c(t)e^{-t}dt\right)\pi B_{M\times U,\tilde{\rho}\gamma}^{I^\prime,\hat{h}_0}(w_0)
		\label{eq:2818b}
	\end{equation}
	holds   
	if and only if the inequality  
	\begin{equation}
		K_{\partial M,\rho}^{I,h_0}(z_0)\geq\left(\int_{0}^{+\infty}c(t)e^{-t}dt\right)\pi B_{\tilde{\rho}}^{I,h_0}(z_0)
		\label{eq:2818a}
	\end{equation}
	holds  by  the assumption that $B_{M\times U,M,\tilde{\rho}\gamma}^{I^\prime,\hat{h}_0}(w_0)>0$. Furthermore, the $``="$ in   inequality \eqref{eq:2818b}  holds if and only if the $``="$ in  inequality (\ref{eq:2818a}) holds.  As $\mathcal{I}(\psi)_{z_0}\subset I_1$, we conclude  Theorem \ref{key-Theorem14} by Theorem \ref{key-Theorem2}.

	\section{Hardy spaces on $M\times U$ over $S\times U$}
	In this section, we firstly give some properties about space $H^2_{\lambda\gamma}(M\times U, S\times U)$.  Then, we give the proofs of Theorem \ref{k-main2} and Theorem \ref{Th4.2}. 
	
	\subsection{Some properties about space $H^2_{\lambda\gamma}(M\times U, S\times U)$}\label{sec:5.1}   
	
	\ 
	
	We firstly recall some notations.
	Let $S:=\prod_{1\leq j\leq n}\partial D_j$, where $D_j$ is a planar regular region with finite boundary components which are analytic Jordan curves and $\partial D_j$ denotes the boundary of $D_j$. Let $U$ be a domain in $\mathbb{C}^m$. Let $\lambda$ be a positive continuous function on $S$, and $\gamma$ be an admissible weight on $U$.
	
	Let $f\in L^2(S,\lambda d\sigma),$ where $d\sigma :=\frac{1}{(2\pi)^n}|dw_1|\cdots|dw_n|.$ We call $f\in H_\lambda^2(M,S)$ if there exists $\left\{f_m\right\}_{m\in\mathbb{Z}_{\geq 0}}\subset \mathcal{O}(M)\cap C(\overline{M})\cap L^2(S,\lambda d\sigma)$ such that $$\lim\limits_{m\rightarrow+\infty}||f_m-f||^2_{S,\lambda}=0,$$
	where $||g||_{S,\lambda} :=\big(\int_S|g|^2\lambda d\sigma\big)^{\frac{1}{2}}$ for any $g\in L^2(S,\lambda d\sigma)$ (see \cite{GY-Hardy and product}). There is an  injective linear map (see \cite{GY-Hardy and product})
	$P_S : H_\lambda^2(M,S)\rightarrow \mathcal{O}(M)$   satisfying that $P_S(f)=f$ for any $f\in \mathcal{O}(M)\cap C(\overline{M})\cap L^2(S,\lambda d\sigma)$. So $H_\lambda^2(M,S)$ can be seen as a subspace of $\mathcal{O}(M)$.

	\begin{Lemma} [see \cite{GY-Hardy and product}]
		\label{P_S-eva-bdd}
		For any compact subset $K$ of $M$, there exists a positive constant $C(K)$ such that
		\begin{equation}\nonumber
			|P_S(f(z))| \leq C(K) ||f||_{S,\lambda}
		\end{equation}
		holds for any $z\in K$ and $f\in H_\lambda^2(M,S)$.
		
	\end{Lemma}
	
	$H^2_{\lambda\gamma}(M\times U, S\times U)$ denotes the set of holomorphic functions $f$ on $M\times U$ such that $P_S^{-1}(f(\cdot,u))\in H^2_\lambda(M,S)$ for every fixed $u\in U$, with finite norms
	$$||f||_{S\times U,\lambda\gamma}: =\bigg(\int_U\gamma(u)dV_U\int_S|P_S^{-1}(f(\cdot,u))|^2\lambda d\sigma\bigg)^{\frac{1}{2}}.$$ 
	
	For every $f\in H^2_{\lambda\gamma}(M\times U, S\times U)$, we define
	$$\hat{W}_f(u):=\int_S|P_S^{-1}(f(z,u))|^2\lambda(z)d\sigma(z).$$
	Now, we give an inequality as follows. 
	\begin{Lemma}\label{first-lem}
		For every compact subset $K$ of $U$, there exists a constant $L(K)$, independent of $f$, such that
		\begin{equation}\label{key-S}
			\hat{W}_f(u) \leq L(K)||f||^2_{S\times U,\lambda\gamma}
		\end{equation}
		for every $u\in K.$
	\end{Lemma}
	\begin{proof}
		As both $\prod_{1\leq l\leq n}\frac{\partial G_{D_{l,k}}}{\partial\nu_{z_l}}(z_l,w_l)$ and $\lambda$ have positive lower bounds and upper bounds, we only prove $(\ref{key-S})$ with weight $\prod_{1\leq l\leq n}\frac{\partial G_{D_{l,k}}}{\partial\nu_{z_l}}(z_l,w_l)\gamma$.
		For fixed $w=(w_1,\cdots,w_n)\in M$, let $D_{l,k}=\left\{z_l\in D_l|G_{D_l}(z_l,w_l)\leq \log(1-\frac{1}{k})\right\}$ with smooth boundary $\partial D_{l,k}$, which is  an increasing sequence of domains  with respect to $k$  such that $w_l\in D_{l,1}$ and $D_l=\bigcup_k D_{l,k}$.  It is well known that $G_{D_{l,k}}(\cdot,w_l)=G_{D_l}(\cdot,w_l)-\log(1-\frac{1}{k})$ is the Green function on $D_{l,k}$.
		Let
		$$\hat{W}_k(u):=\int_{\prod_{1\leq l\leq n}\partial D_{l,k}}|f(z,u)|^2\prod_{1\leq l\leq n}\frac{\partial G_{D_{l,k}}}{\partial\nu_{z_l}}(z_l,w_l)d\sigma.$$
		As $|f(z,u)|^2$ is plurisubharmonic in $z$ for fixed $u$, we know that $\hat{W}_k(u)$ is increasing with respect to $k$. Since $f(z,\cdot)\in \mathcal{O}(U)$ and $\gamma$ is an admissible weight on $U$, there exists a constant $C(K)$ (independent of $z$ and $f$) such that
		$$\sup\limits_{u\in K}|f(z,u)|^2 \leq C(K)\int_U|f(z,u)|^2\gamma(u)dV_U$$
		for every compact subset $K$ of $U$.
		Thus, we have
		\begin{equation}\label{transfer}
			\begin{split}
				\hat{W}_k(u)&\leq C(K)\int_{\prod_{1\leq l\leq n}\partial D_{l,k}}\int_U|f(z,u)|^2\gamma(u)dV_U\prod_{1\leq l\leq n}\frac{\partial G_{D_{l,k}}}{\partial\nu_{z_l}}(z_l,w_l)d\sigma\\
				&=C(K)\int_U\gamma(u)\int_{\prod_{1\leq l\leq n}\partial D_{l,k}}|f(z,u)|^2\prod_{1\leq l\leq n}\frac{\partial G_{D_{l,k}}}{\partial\nu_{z_l}}(z_l,w_l)d\sigma dV_U\\
				&=C(K)\int_U\gamma(u)\hat{W}_k(u)dV_U.
			\end{split}
		\end{equation}
		In the following, we  show that
		\begin{equation}\label{target}
			\hat{W}_k(u)\nearrow \int_S|P_S^{-1}f(z,u)|^2\prod_{1\leq l\leq n}\frac{\partial G_{D_{l}}}{\partial\nu_{z_l}}(z_l,w_l)d\sigma\quad (k\rightarrow\infty).
		\end{equation}  
		For each fixed $u\in U$,  $P_S^{-1}(f(z,u))\in H^2_\lambda(M,S)$. Then there exists $\left\{f_m(z,u)\right\}_{m\in\mathbb{Z}_{\geq 0}}\subset \mathcal{O}(M)\cap C(\overline{M})\cap L^2(S,\lambda d\sigma)$ such that
		\begin{equation}\label{app}
			\lim\limits_{m\rightarrow+\infty}||f_m(\cdot,u)-P_S^{-1}(f(\cdot,u))||^2_{S,\lambda}=0,
		\end{equation}
		and by Lemma \ref{P_S-eva-bdd}, we know that
		$\left\{f_m(\cdot,u)\right\}_{m\in\mathbb{Z}_{\geq 0}}$ converges uniformly to $f(\cdot,u)$ on every compact subset of $M$.
		Since $\lambda$ and $\prod_{1\leq l\leq n}\frac{\partial G_{D_{l}}}{\partial\nu_{z_l}}(z_l,w_l)$ are both positive and continuous on  $S$, using equality (\ref{app}), we  get 
		\begin{equation}\nonumber
			\lim\limits_{m\rightarrow+\infty}||f_m(\cdot,u)-P_S^{-1}(f(\cdot,u))||^2_{S,\prod_{1\leq l\leq n}\frac{\partial G_{D_{l}}}{\partial\nu_{z_l}}(z_l,w_l)}=0. 
		\end{equation} 
		Then there exists a constant $L(u)$ such that 
		$$||f_m(\cdot,u)||^2_{{S,\prod_{1\leq l\leq n}\frac{\partial G_{D_{l,k}}}{\partial\nu_{z_l}}(z_l,w_l)}}\leq L(u)$$
		for any $k$.
		And it is clear that
		\begin{equation}\label{usual-ine}
			\begin{split}
				&\int_{\prod_{1\leq l\leq n}\partial D_{l,k}}\bigg(|f_m(z,u)|+|f_n(z,u)|\bigg)^2\prod_{1\leq l\leq n}\frac{\partial G_{D_{l,k}}}{\partial\nu_{z_l}}(z_l,w_l)d\sigma\\
				&\leq 2\int_{\prod_{1\leq l\leq n}\partial D_{l,k}}\bigg(|f_m(z,u)|^2+|f_n(z,u)|^2\bigg)\prod_{1\leq l\leq n}\frac{\partial G_{D_{l,k}}}{\partial\nu_{z_l}}(z_l,w_l)d\sigma\\
				&\leq 4L(u).
			\end{split}
		\end{equation}
		For $f_m(\cdot,u)\in \mathcal{O}(M)\cap C(\overline{M})$, equality (\ref{target}) obviously holds as $P_S^{-1}(f_m(\cdot,u))=f_m(z,u)$ on $S$ and $\frac{\partial G_{D_{l,k}}}{\partial\nu_{z_l}}(z_l,w_l)=\frac{\partial G_{D_{l}}}{\partial\nu_{z_l}}(z_l,w_l)\in C(\overline{M})$.  Let
		$$\hat{W}_{k,m}(u)=\int_{\prod_{1\leq l\leq n}\partial D_{l,k}}|f_m(z,u)|^2\prod_{1\leq l\leq n}\frac{\partial G_{D_{l,k}}}{\partial\nu_{z_l}}(z_l,w_l)d\sigma.$$ 
		Then it follows from the Cauchy inequality and inequality (\ref{usual-ine}) that
		\begin{equation}\label{ine-32}
			\begin{split}
				&|\hat{W}_{k,m}(u)-\hat{W}_{k,n}(u)|\\
				&\leq \int_{\prod_{1\leq l\leq n}\partial D_{l,k}}\bigg||f_m(z,u)|^2-|f_n(z,u)|^2\bigg|\prod_{1\leq l\leq n}\frac{\partial G_{D_{l,k}}}{\partial\nu_{z_l}}(z_l,w_l)d\sigma\\
				&\leq \int_{\prod_{1\leq l\leq n}\partial D_{l,k}}\bigg|f_m(z,u)-f_n(z,u)\bigg|\bigg(|f_m(z,u)|+|f_n(z,u)|\bigg)\prod_{1\leq l\leq n}\frac{\partial G_{D_{l,k}}}{\partial\nu_{z_l}}(z_l,w_l)d\sigma\\
				&\leq \bigg(\int_{\prod_{1\leq l\leq n}\partial D_{l,k}}|a_{m,n}|^2c_kd\sigma\bigg)^{1/2}
				\times \bigg(\int_{\prod_{1\leq l\leq n}\partial D_{l,k}}(b_{m,n})^2c_kd\sigma\bigg)^{1/2}\\ &\leq   2\sqrt{L(u)}||f_m(\cdot,u)-f_n(\cdot,u)||_{{S,\prod_{1\leq l\leq n}\frac{\partial G_{D_{l}}}{\partial\nu_{z_l}}(z_l,w_l)}},         
			\end{split}
		\end{equation}
		where $a_{m,n}=f_m(z,u)-f_n(z,u),b_{m,n}=|f_m(z,u)|+|f_n(z,u)|,c_k=\prod_{1\leq l\leq n}\frac{\partial G_{D_{l,k}}}{\partial\nu_{z_l}}(z_l,w_l)$.
		It follows from equality $(\ref{app})$ and inequality (\ref{ine-32})  that $$\lim\limits_{m\rightarrow\infty}\hat{W}_{k,m}(u)=\hat{W}_{k}(u)$$ uniformly on $k=1,2,\cdots$.
		On the other hand, we have
		$$\lim\limits_{m\rightarrow\infty}\lim\limits_{k\rightarrow\infty}\hat{W}_{k,m}(u)=\int_S|P_S^{-1}f(z,u)|^2\prod_{1\leq l\leq n}\frac{\partial G_{D_{l}}}{\partial\nu_{z_l}}(z_l,w_l)d\sigma.$$
		So, we get
		\begin{equation}\nonumber
			\begin{split}
				\lim\limits_{k\rightarrow\infty}\hat{W}_k(u)&=
				\lim\limits_{k\rightarrow\infty}\lim\limits_{m\rightarrow\infty}\hat{W}_{k,m}(u)
				=\lim\limits_{m\rightarrow\infty}\lim\limits_{k\rightarrow\infty}\hat{W}_{k,m}(u)\\
				&=\int_S|P_S^{-1}f(z,u)|^2\prod_{1\leq l\leq n}\frac{\partial G_{D_{l}}}{\partial\nu_{z_l}}(z_l,w_l)d\sigma,
			\end{split}
		\end{equation}
		where the second $``="$  arises from which $\lim\limits_{m\rightarrow\infty}\hat{W}_{k,m}(u)=\hat{W}_{k}(u)$ uniformly on $k=1,2,\cdots$.

		Let $k\rightarrow\infty$
		in inequality (\ref{transfer}), we get the inequality 
		\begin{equation}\nonumber
			\begin{split}
				&\int_{S}|P_S^{-1}f(z,u)|^2\prod_{1\leq l\leq n}\frac{\partial G_{D_{l}}}{\partial\nu_{z_l}}(z_l,w_l)d\sigma\\
				&\leq C(K)\int_{S}\int_U|P_S^{-1}f(z,u)|^2\gamma(u)dV_U\prod_{1\leq l\leq n}\frac{\partial G_{D_{l}}}{\partial\nu_{z_l}}(z_l,w_l)d\sigma.
			\end{split}
		\end{equation}  
		Since  $\lambda$ and $\prod_{1\leq l\leq n}\frac{\partial G_{D_{l}}}{\partial\nu_{z_l}}(z_l,w_l)$ are both positive and continuous on compact $S$, inequality (\ref{key-S}) holds.
	\end{proof}
	Here we show that the evaluation $e_{(z,u)} :f\longmapsto f(z,u)$ from $H^2_{\lambda\gamma}(M\times U, S\times U)$ to $\mathbb{C}$ is a bounded  linear map for each fixed $(z,u)\in M\times U$.
	\begin{Lemma}\label{Le-33-1}
		For every compact subset $K=K_1\times K_2$ of $M\times U$, there exists a constant $C_2(K)$ such that
		\begin{equation}\label{eva-bdd}
			\sup\limits_{(z,u)\in K}  | f(z,u)|\leq C_2(K)||f||_{S\times U,\lambda\gamma}
		\end{equation}
		for every $f\in H^2_{\lambda\gamma}(M\times U, S\times U)$.
	\end{Lemma}
	\begin{proof} 
		As $\gamma$ is an admissible weight, we get
		\begin{equation*}
			\begin{split}
				\sup\limits_{u\in K_2}|f(z,u)|^2\leq C({K_2}) \int_U|f(z,w)|^2\gamma(w)dV_U.
			\end{split}  
		\end{equation*}
		By  Lemma \ref{P_S-eva-bdd}, we know that
		\begin{equation}
			\sup\limits_{z\in K_1}|f(z,u)|^2\leq C(K_1)\int_S|P_S^{-1}(f(z,u))|^2\lambda(z)d\sigma.
			\nonumber
		\end{equation}
		Combinig with the above two inequalities, we obtain that
		\begin{equation}
			\sup\limits_{(z,u)\in K}|f(z,u)|^2\leq  {C(K_1)C({K_2})}\int_U\int_S|P_S^{-1}(f(z,u))|^2\lambda(z)d\sigma(z)\gamma(u)dV_U.
			\nonumber
		\end{equation}
		Hence, inequality (\ref{eva-bdd}) holds.
	\end{proof}
	\begin{Lemma}\label{key-Coro}
		Let $\left\{f_m\right\}$ be a Cauchy sequence in $H^2_{\lambda\gamma}(M\times U, S\times U)$. Then $\left\{f_m\right\}$ converges to $f\in\mathcal{O}(M\times U)$ uniformly on each compact subset of $M\times U$, and satisfies that for fixed $u$, $P_S^{-1}(f(\cdot,u))\in H^2_\lambda(M,S)$ and
		\begin{equation}\label{key-equal}
			\lim\limits_{m\rightarrow+\infty}||P_S^{-1}(f_m(\cdot,u))-P_S^{-1}(f(\cdot,u))||_{S,\lambda}=0.
		\end{equation}
	\end{Lemma}
	\begin{proof}
		By Lemma \ref{Le-33-1}, $\left\{f_m\right\}$ converges to a holomorphic function $f$ on $M\times U$ uniformly on each compact subset of $M\times U$. By Lemma \ref{first-lem}, we know $\left\{P_S^{-1}(f_m(\cdot,u))\right\}$ is a Cauchy sequence in $H^2_\lambda(M,S)$ for each fixed $u$. Then there exists  $g(\cdot,u)\in H^2_\lambda(M,S)$ such that
		$$\lim\limits_{m\rightarrow+\infty}||P_S^{-1}(f_m(\cdot,u))-g(\cdot,u)||_{S,\lambda}=0$$
		since $H^2_\lambda(M,S)$ is a Hilbert space.
		Besides, it follows from Lemma \ref{P_S-eva-bdd} that $\left\{f_m\right\}$ converges uniformly to $P_S(g)$ on each compact subset of $M$. Thus, for fixed $u\in U$, we have $f=P_S(g)$. Thus, Lemma \ref{key-Coro} holds.
	\end{proof}
	Now we show that $H^2_{\lambda\gamma}(M\times U, S\times U)$ is a Hilbert space.
	\begin{Proposition}
		$H^2_{\lambda\gamma}(M\times U, S\times U)$ is a Hilbert space with the following inner product:
		\begin{equation}
			\ll f,g\gg_{S\times U,\lambda\gamma}=\int_{S\times U}P_S^{-1}(f(z,u))\overline{P_S^{-1}(g(z,u))}\lambda\gamma d\sigma dV_U.
			\nonumber
		\end{equation}
	\end{Proposition}
	\begin{proof}
		Let $\left\{f_m\right\}$ be a Cauchy sequence in $H^2_{\lambda\gamma}(M\times U, S\times U)$. Then it follows from Lemma \ref{key-Coro} that the sequence converges to a holomorphic  function $f$ on $M\times U$, such that equality (\ref{key-equal}) holds. Thus, we only need to prove that
		\begin{equation}
			\lim\limits_{m\rightarrow+\infty}||f_m-f||_{S\times U,\lambda\gamma}=0.
			\nonumber
		\end{equation}
		Since $\left\{f_m\right\}$ is a Cauchy sequence in $H^2_{\lambda\gamma}(M\times U, S\times U)$, there exists a constant $L$ such that
		$$||f_m||^2_{S\times U,\lambda\gamma}\leq L.$$
		Thus, we have
		\begin{equation}\label{k-estimate}
			\begin{split}
				&\bigg(\int_U\int_{S}\bigg||P_S^{-1}(f_n(z,u))|+|P_S^{-1}(f_m(z,u))|\bigg|^2\lambda(z)\gamma(u)d\sigma(z)dV_U\bigg)^{1/2}\\
				&\leq \bigg(\int_U\int_{S}2(|P_S^{-1}(f_n(z,u))|^2+|P_S^{-1}(f_m(z,u))|^2)\lambda(z)\gamma(u)d\sigma(z)dV_U\bigg)^{1/2}\\
				&\leq {2}\sqrt{L}.
			\end{split}
		\end{equation} Denote  
		$$\hat{W}_m(u):=\int_S|P_S^{-1}(f_m(z,u))|^2\lambda(z)d\sigma$$ 
		and
		$$\hat{W}(u):=\int_S|P_S^{-1}(f(z,u))|^2\lambda(z)d\sigma.$$
		By equality (\ref{key-equal}), we know that
		\begin{equation}\label{key-e2}
			\lim\limits_{m\rightarrow+\infty}\hat{W}_m(u)=\hat{W}(u).
		\end{equation}        
		Besides, we have
		\begin{equation}\nonumber
			\begin{split}
				&\int_U|\hat{W}_n(u)-\hat{W}(u)|\gamma(u)dV_U\\ &=\int_U\bigg|\int_{S}\bigg(|P_S^{-1}(f_n(z,u))|^2-|P_S^{-1}(f(z,u))|^2)\bigg)\lambda(z)d\sigma(z)\bigg|\gamma(u)dV_U\\
				&\leq\liminf\limits_{m\rightarrow\infty}\int_U\bigg|\int_{S}\bigg(|P_S^{-1}(f_n(z,u))|^2-|P_S^{-1}(f_m(z,u))|^2)\bigg)\lambda(z)d\sigma(z)\bigg|\gamma(u)dV_U \\
				&\leq \liminf\limits_{m\rightarrow\infty}\int_U\int_{S}\bigg||P_S^{-1}(f_n(z,u))|^2-|P_S^{-1}(f_m(z,u))|^2\bigg|\lambda(z) d\sigma(z)\gamma(u)dV_U\\
				&=\liminf\limits_{m\rightarrow\infty}\int_U\int_{S}|d||a|\lambda(z)\gamma(u)d\sigma(z)dV_U\\
				&\leq \liminf\limits_{m\rightarrow\infty}\bigg(\int_U\int_{S}|d|^2\lambda(z)\gamma(u)d\sigma(z)dV_U\bigg)^{1/2}
				\bigg(\int_U\int_{S}|a|^2\lambda(z)\gamma(u)d\sigma(z)dV_U\bigg)^{1/2}  \\
				&\leq {2}\sqrt{L}\liminf\limits_{m\rightarrow\infty}||f_n-f_m||_{S\times U,\lambda\gamma}\quad\mbox{(by inequality (\ref{k-estimate}))},
			\end{split}
		\end{equation}
		where
		$d=(|P_S^{-1}(f_n(z,u))|-|P_S^{-1}(f_m(z,u))|),a=(|P_S^{-1}(f_n(z,u))|+|P_S^{-1}(f_n(z,u))|$,
		the first $``\leq"$ holds because of equality (\ref{key-e2}) and  Fatou's Lemma, and the third $``\leq"$ holds by the Cauchy-Schwartz inequality. Then it follows that
		$$\lim\limits_{n\rightarrow\infty}\int_U|\hat{W}_n(u)-\hat{W}(u)|\gamma(u)dV_U=0,$$
		which implies that
		\begin{equation}\label{W-int-equa}
			\lim\limits_{n\rightarrow\infty}\int_U\hat{W}_n(u)\gamma(u)dV_U=\int_U\hat{W}(u)\gamma(u)dV_U.
		\end{equation}
		Let
		$$g_m(u)=\int_S|P_S^{-1}(f_m(z,u))-P_S^{-1}(f(z,u))|^2\lambda(z)d\sigma(z).$$
		Then it follows from equality (\ref{key-equal}) that
		\begin{equation}\label{k-equal3}
			\lim\limits_{m\rightarrow+\infty}g_m(u)=0.
		\end{equation}
		On the other hand, we know that
		\begin{equation*}
			\begin{split}
				g_m(u)&=\int_S|P_S^{-1}(f_m(z,u))-P_S^{-1}(f(z,u))|^2\lambda(z)d\sigma(z)\\
				&\leq \int_S\big(|P_S^{-1}(f_m(z,u))|^2+|P_S^{-1}(f(z,u))|^2\big)\lambda(z)d\sigma(z)\\
				&=2(\hat{W}_m(u)+\hat{W}(u)),
			\end{split}
		\end{equation*}
		which implies that
		$$2(\hat{W}_m(u)+\hat{W}(u))-g_m(u)\geq 0.$$
		Then by Fatou's Lemma, we obtain
		\begin{equation*}
			\begin{split}
				&\int_U\liminf\limits_{m\rightarrow\infty}(2(\hat{W}_m(u)+\hat{W}(u))-g_m(u))\gamma(u)dV_U\\
				&\leq \liminf\limits_{m\rightarrow\infty}\int_U(2(\hat{W}_m(u)+\hat{W}(u))-g_m(u))\gamma(u)dV_U,
			\end{split}
		\end{equation*}
		which indicates that
		$$\limsup\limits_{m\rightarrow\infty}\int_Ug_m(u)\gamma(u)dV_U=0$$
		with the help of equality (\ref{W-int-equa}) and equality (\ref{k-equal3}). Namely,
		$$\lim\limits_{m\rightarrow\infty}||f_m-f||_{S\times U,\lambda\gamma}=0.$$
		The proof is finished.
	\end{proof}
	Since $H^2_{\lambda\gamma}(M\times U, S\times U)$ is a Hilbert space and the evaluation map from $H^2_{\lambda\gamma}(M\times U, S\times U)$ to $\mathbb{C}$ is  bounded, it follows from the Riesz representation theorem that  for each $(z,w)\in M\times U$, there is a unique $K_{S\times U,\lambda\gamma}((\cdot,\cdot),(\overline{z},\overline{w}))\in H^2_{\lambda\gamma}(M\times U, S\times U)$ such that
	$$f(z,w)=\int_{S\times U}P_S^{-1}(f(\zeta,u))\overline{P_S^{-1}(K_{S\times U,\lambda\gamma}((\zeta,u),(\overline{z},\overline{w})))}\lambda(\zeta)\gamma(u) d\sigma dV_U$$
	for $f\in H^2_{\lambda\gamma}(M\times U, S\times U)$.
	
	Now we give a decomposition formula for  $K_{S\times U,\lambda\gamma}((\cdot,\cdot),(\overline{z},\overline{w}))$.
	\begin{Proposition}\label{key-decomp1p}
		\begin{equation}\label{key-decomp1}
			K_{S\times U,\lambda\gamma}((\zeta,u),(\overline{z},\overline{w}))=K_{S,\lambda}(\zeta,\overline{z})B_{U,\gamma}(u,\overline{w}).
		\end{equation}
	\end{Proposition}
	\begin{proof}
		For any $f\in H^2_{\lambda\gamma}(M\times U, S\times U)$, we have $P_S^{-1}(f(\cdot,u))\in H^2_\lambda(M,S)$ for each fixed $u\in U$.
		By Lemma \ref{P_S-eva-bdd}, for every compact subset $K$ of $M$, we have
		$$\sup\limits_{z\in K}|f(z,u)|^2\leq C(K)\int_S|P_S^{-1}(f(\zeta,u))|^2\lambda(\zeta)d\sigma(\zeta).$$ 
		Then it follows that
		\begin{equation*}
			\int_U|f(z,u)|^2\gamma(u)dV_U\leq  {C(K)}\int_U\int_S|P_S^{-1}(f(\zeta,u))|^2 \lambda(\zeta)d\sigma(\zeta)\gamma(u)dV_U
			<+\infty, 
		\end{equation*}
		which shows that $f(z,\cdot)\in A^2(U,\gamma)$ for each fixed $z\in M$. 
		
		Hence, we have
		\begin{equation}
			\begin{split}
				&\int_{S\times U}P_S^{-1}(f(\zeta,u))\overline{P_S^{-1}(K_{S,\lambda}(\zeta,\overline{z}))B_{U,\gamma}(u,\overline{w})}\lambda\gamma d\sigma dV_U\\
				&=\int_U\overline{B_{U,\gamma}(u,\overline{w})}\gamma(u)dV_U\int_{S}P_S^{-1}(f(\zeta,u))\overline{P_S^{-1}(K_{S,\lambda}(\zeta,\overline{z}))}\lambda(z)d\sigma(z)\\
				&=\int_Uf(z,u)\overline{B_{U,\gamma}(u,\overline{w})}\gamma(u)dV_U=f(z,w).
			\end{split}
			\nonumber
		\end{equation}
		Now, we obtain equality (\ref{key-decomp1}) by the uniqueness part of the Riesz representation theorem.
	\end{proof}

	Let $\left\{e_i\right\}_{i\in\mathbb{N}}$ be a complete orthonormal basis of $H^2_\lambda(M,S)$, and $\left\{\tilde{e}_j\right\}_{j\in\mathbb{N}}$ be a complete orthonormal basis of $A^2(U,\gamma)$. 
	Then $\left\{P_S(e_i)\tilde{e}_j\right\}_{i,j\in\mathbb{N}}$ is a  orthonormal set in $H^2_{\lambda\gamma}(M\times U, S\times U)$. The following lemma shows that this orthonormal set is complete in $H^2_{\lambda\gamma}(M\times U, S\times U)$.
	
	\begin{Lemma}\label{S-U product}
		$\left\{P_S(e_i)\tilde{e}_j\right\}_{i,j\in\mathbb{N}}$ is an  orthonormal basis of $H^2_{\lambda\gamma}(M\times U, S\times U)$.
	\end{Lemma}
	\begin{proof}
		We divide the proof into three steps.
		
		\
		
		\emph{Step 1: constructing a complete orthonormal basis of $H^2_\lambda(M,S)$.}
		
		\
		
		We choose a countable dense subset $\left\{z_k\right\}_{k\in\mathbb{Z}_{>0}}$ of $M$. Denote 
		$$L_k\;:=\left\{f\in H^2_\lambda(M,S):f^*(z_j)=0\;\mbox{for any}\;1\leq j<k \mbox{ and }f^*(z_k)=1\right\}.$$
		Let $A=\left\{k\in \mathbb{Z}_{>0}:L_k=\emptyset\right\}$ and $B=\mathbb{Z}_{>0}\backslash A$. For $k\in B$, we select $\sigma_k\in H^2_\lambda(M,S)$ such that
		$||\sigma_k||^2_{S,\lambda}=\inf\left\{||f||^2_{S,\lambda}:f\in L_k\right\}$. Now we assert that $\left\{\sigma_k\right\}_{k\in B}$ is an orthonormal set of $H^2_\lambda(M,S)$. To see this, it suffices to note that
		$$||\sigma_k+c\sigma_l||^2_{S,\lambda}\geq||\sigma_k||^2_{S,\lambda}$$ 
		for any $c\in\mathbb{C}$ and $l>k$, which implies that $\ll \sigma_k,\sigma_l\gg_{S,\lambda}=0$. Now we set $\tilde\sigma_k=\frac{\sigma_k}{||\sigma_k||_{S,\lambda}}$ for any $k\in B$.
		
		In the following, we prove the set $\left\{\tilde\sigma_k\right\}_{k\in B}$ is complete in $H^2_\lambda(M,S)$.
		We prove this by contrdiction: if not, there exists a nonzero $f\in H^2_\lambda(M,S)$ such that 
		\begin{equation}\label{E-23-2}
			\ll f,\tilde\sigma_k\gg_{S,\lambda}=0
		\end{equation}
		for any $k\in B$.  Because of the density of $\left\{z_k\right\}_{k\in\mathbb{Z}_{>0}}$, we have $\left\{k\in B:f^*(z_k)\neq 0\right\}\not=\emptyset$. Denote  $$k_0:=\inf\left\{k\in B:f^*(z_k)\neq 0\right\}.$$  There exists a nonzero $c_0\in\mathbb{C}$ such that 
		$$\ll f-c_0\tilde\sigma_{k_0},\tilde\sigma_{k_0}\gg_{S,\lambda}=0,$$ 
		which contradicts to (\ref{E-23-2}). Hence, we finish the first step.
		
		\
		
		\emph{Step 2: we give a product property for the basis of $H^2_{\lambda\gamma}(M\times U,S\times U)$.}
		
		\
		
		Let $\left\{f_l\right\}_{l\in\mathbb{Z}_{>0}}$ be any complete orthonormal basis of $A^2(U,\lambda)$.
		It is clear that $\left\{f_lP_S(\tilde\sigma_k)\right\}_{l\in\mathbb{Z}_{>0},k\in B}$ is orthonormal. It remains to show that the set is complete in the space $H^2_{\lambda\gamma}(M\times U,S\times U)$. For any $f\in H^2_{\lambda\gamma}(M\times U,S\times U)$, we will prove that $f=0$ if $f$ satisfies that
		\begin{equation}\label{E-24-1}
			\ll f,f_lP_S(\tilde\sigma_k)\gg_{S\times U,\lambda\gamma}=0
		\end{equation}
		for any $l\in\mathbb{Z}_{>0},k\in B$. Let
		$$g_k(u)\;:=\int_{S}P_S^{-1}(f(\cdot,u))\overline{\tilde\sigma_k}\lambda d\sigma.$$
		By Fubini's Theorem, it is easy to see that $g_k(u)\in L^2(U,\gamma)$. In the following, we prove $g_k(u)$ is a holomorphic function on $U$.
		Since $\left\{\tilde\sigma_k\right\}_{k\in B}$ is a complete orthonormal basis of $H^2_\lambda(M,S)$, we have
		$$P_S^{-1}(f(\cdot,u))=\sum_{l\in\mathbb{Z}_{>0}}g_k(u)\tilde\sigma_k,$$
		which says that
		\begin{equation}\label{E-24-2}
			f(z,u)=\sum_{l\in\mathbb{Z}_{>0}}g_k(u)\tilde\sigma_k^*(z).
		\end{equation} 
		Substituting $z_1,z_2,\cdots$ into $z$ in equality (\ref{E-24-2}), we get that $g_k(u)\in\mathcal{O}(U)$. Combining with $g_k(u)\in L^2(U,\gamma)$, we have $g_k(u)\in A^2(U,\gamma)$. Then is follows from (\ref{E-24-1}) that
		$$\ll g_k(u),f_l\gg_{U,\gamma}=0,$$
		which implies that $g_k(u)=0$ for $k\in\mathbb{Z}_{>0}$ since $\left\{f_l\right\}_{l\in\mathbb{Z}_{>0}}$ is a complete orthonormal basis  for $A^2(U,\gamma)$.
		Furthermore, $P_S^{-1}(f(\cdot,u))=0$ since $\left\{{e}_k\right\}_{k\in B}$ is a complete orthonormal basis for $H^2_\lambda(M,S)$. Hence $f=0$. We finish the second step.
		
		\
		
		\emph{Step 3: we prove that $\left\{f_lP_S({e_m})\right\}$ is a complete orthonormal basis of $H^2_{\lambda\gamma}(M\times U,S\times U)$.}
		
		\
		
		For any $f\in H^2_{\lambda\gamma}(M\times U,S\times U)$, it follows from \emph{Step 2} that there exists $\left\{a_{l,k}\right\}\subset\mathbb{C}$ such that
		\begin{equation}\label{44-eq-1}
			f=\sum_{l,k}a_{l,k}f_lP_S(\tilde\sigma_k).
		\end{equation}
		Since $\left\{{e_m} \right\}$ is a complete orthonormal basis for $H^2_\lambda(M,S)$, there exists $\left\{{a_{k,m}}\right\}$ such that
		\begin{equation}\label{44-eq-2}
			\tilde\sigma_k=\sum_{m}\tilde{a}_{k,m}{e_m}.
		\end{equation}
		Substituting (\ref{44-eq-2}) into (\ref{44-eq-1}), we get
		\begin{equation*}
			f=\sum_{l,k,m}a_{l,k}\tilde{a}_{k,m}f_lP_S({e_m}).
		\end{equation*}
		We finish the proof.
	\end{proof}

	Let $$I_1=\left\{(g,z_0)\in \mathcal{O}_{z_0}: g=\sum_{\alpha\in \mathbb{N}^{n}}b_\alpha(w-z_0)^\alpha \mbox{near}\;z_0 \;\mbox{s.t.} b_\alpha=0 \;\mbox{for}\;\alpha\in L_{\tilde{\beta}^\prime}\right\},$$ where $L_{\tilde{\beta}^\prime}=\left\{\alpha=(\alpha_{1},\cdots,\alpha_{n})\in \mathbb{N}^{n}:\alpha_j\leq \tilde{\beta}_j\;\mbox{for}\;1\leq j\leq n \right\}.$   
	Let
	$$I_2=\left\{(g,u_0)\in\mathcal{O}_{u_0} : g=\sum_{\alpha\in \mathbb{N}^{m}}b_\alpha(u-u_0)^\alpha\;\mbox{near}\;u_0\;\mbox{s.t.}\;b_\alpha=0\;\mbox{for}\;\alpha\in L_{\tilde{\beta}^{\prime\prime}}\right\},$$
	where $L_{\tilde{\beta}^{\prime\prime}}=\left\{\alpha\in \mathbb{N}^{m} :\alpha\leq \tilde{\beta}^{\prime\prime}\right\}$. It is clear that $(b_0,u_0)\notin I_2$. Denote 
	$${B_{U,\gamma}^{I_2,b_0}}(u_0):
	=\frac{1}{\inf\left\{||f||^2_{U,\gamma} :f\in A^2(U,\gamma)\&(f-b_0,u_0)\in I_2\right\}}$$

	and
	$$K_{S,\lambda}^{I_1,l_0}(z_0):
	=\frac{1}{\inf\left\{||f||^2_{S,\lambda} :f\in H^2_\lambda(M,S)\&(P_{S}(f)-l_0,z_0)\in I_1\right\}}.$$

	Assume that $H^2_{\lambda\gamma}(M\times U, S\times U)\neq\left\{0\right\}$ and denote 
	$f^*:=P_S(f)$ for any $f\in H^2_\lambda(M,S)$.
	Now we prove the following equality.
	\begin{Proposition}\label{p:1211}
		\begin{equation}
			\label{S-decomp}
			K_{S\times U,\lambda\gamma}^{I^\prime,\hat{h}_0}(w_0)=K_{S,\lambda}^{I_1,l_0}(z_0)B_{U,\gamma}^{I_2,b_0}(u_0)
		\end{equation}
		holds. 
	\end{Proposition}
	\begin{proof}
		The proof is similar to the proofs of Lemma \ref{Le-44-1} and Proposition \ref{pf of 1}.
		Let $\left\{e_i\right\}_{i\in \mathbb{N}}$ be  a complete orthonormal basis of $H_\lambda^2(M,S)$. Let $\left\{\tilde{e}_\tau\right\}_{\tau\in S_2}$ be a complete orthonormal basis of $A^2(U,\gamma)$ such that $\tau=ord_{u_0}(\tilde{e}_\tau)$ by Lemma \ref{A_basis}, where $S_2$ is a subset of $\mathbb{N}^m$. Then  we know that $\left\{e_i^*{\tilde e_\tau}\right\}_{i\in\mathbb{N},\tau\in S_2}$ is a complete orthonormal basis of $H^2_{\lambda\gamma}(M\times U,S\times U)$ by Lemma \ref{S-U product}.
		
		Denote 
		\begin{equation*}
			\begin{split}
				&A:=\left\{\tau\in S_2:\beta^{\prime\prime}\leq\tau\leq\tilde{\beta}^{\prime\prime}\right\},\\
				&B:=\left\{f\in H^2_{\lambda\gamma}(M\times U,S\times U):(f-\hat{h}_0,w_0)\in I^\prime\right\},\\
				&C_1:=\left\{f\in H^2_\lambda(M,S):(f^*-l_0,z_0)\in I_1\right\},\\
				&C_2:=\left\{g\in A^2(U,\gamma):(f-b_0,u_0)\in I_2\right\}.
			\end{split}
		\end{equation*}         
		We say that $\inf\left\{||f||_{S\times U,\lambda\gamma}: f\in B\right\}=+\infty$ if and only if $B=\emptyset$, i.e. $K_{S\times U,\lambda\gamma}^{I^\prime,\hat{h}_0}(w_0)=0$.        
		
		If $K_{S\times U,\lambda\gamma}^{I^\prime,\hat{h}_0}(w_0)=0$, it is clear that at least one of $B_{U,\gamma}^{I_2,b_0}(u_0)$ and $K_{S,\lambda}^{I_1,l_0}(z_0)$ equals zero.  Hence, the right member of equality (\ref{S-decomp}) is also zero.
		
		Now  assume that $K_{S\times U,\lambda\gamma}^{I^\prime,\hat{h}_0}(w_0)>0$. 
		We claim that for any $f\in B$, there exists $f_1\in C_1$ and $f_2\in C_2$ such that
		$$||f||^2_{S\times U,\lambda\gamma}\geq ||f_1||^2_{S,\lambda}||f_2||^2_{U,\gamma}.$$
		
		We prove the claim in the following.       
		As $\left\{e_i^*{\tilde e_\tau}\right\}_{i\in\mathbb{N},\tau\in S_2}$ is a complete orthonormal basis of $H^2_{\lambda\gamma}(M\times U,S\times U)$, we have 
		$$f=\sum_{\tau}f_\tau \tilde e_\tau,$$ 
		where $P_{S}^{-1}(f_\tau)\in H^2_\lambda(M,S)$.    Since $(f-\hat{h}_0,w_0)\in I^\prime$, there exists a constant $c_\tau$ such that
		$$(f_\tau-c_\tau l_0,z_0)\in I_1$$
		for any $\tau\in A$. Note that $\sum_{\tau\in A}|c_\tau|>0$. There exsit $f_1\in H^2_\lambda(M,S)$ such that $(P_{S}(f_1)-l_0,z_0)\in I_1$ and
		$$||f_1||_{S,\lambda}^2=\inf\left\{||f||_{S,\lambda}^2:f\in H^2_\lambda(M,S)\&\;(P_{S}(f)-l_0,z_0)\in I_1\right\}.$$ 
		Thus, we know that $(f-\sum_{\tau\in A}c_\tau \tilde e_\tau P_{S}(f_1),w_0)\in I^\prime$ and
		\begin{equation}\label{55-4-ine-1}
			\begin{split}
				||f||_{S\times U,\lambda\gamma}^2&\geq ||\sum_{\tau\in A}P_{S}^{-1}(f_\tau) \tilde e_\tau ||^2_{S\times U,\lambda\gamma}\\
				&=\sum_{\tau\in A}||P_{S}^{-1}(f_\tau)  ||^2_{S,\lambda}\\
				&\geq \sum_{\tau\in A}|c_\tau|^2||f_1  ||^2_{S,\lambda}. 
			\end{split}
		\end{equation}
		Let 
		$$f_2:=\sum_{\tau\in A}c_\tau \tilde e_\tau.$$
		Then we have $f_2\in C_2$. Then it follows from inequality (\ref{55-4-ine-1}) that
		$$||f||_{S\times U,\lambda\gamma}^2\geq ||f_2||^2_{U,\gamma}||f_1||^2_{S,\lambda}.$$
		It follows  that
		\begin{equation*}
			\inf\left\{||f||_{S\times U,\lambda\gamma}^2:f\in B\right\}\geq \inf\left\{||gh||_{S\times U,\lambda\gamma}^2:g\in C_1  \& h
			\in C_2\right\}.
		\end{equation*}
		
		By definitions, we have
		$$\inf\left\{||f||_{S\times U,\lambda\gamma}^2:f\in B\right\}\leq \inf\left\{||gh||_{S\times U,\lambda\gamma}^2:g\in C_1  \& h
		\in C_2\right\}.$$
		Hence, we get the equality
		$$\inf\left\{||f||_{S\times U,\lambda\gamma}^2:f\in B\right\}= \inf\left\{||gh||_{S\times U,\lambda\gamma}^2:g\in C_1  \& h
		\in C_2\right\}.$$
		Thus, equality (\ref{S-decomp}) holds.      
	\end{proof}

	\subsection{Proof of Theorem \ref{k-main2} } 
	It follows from Proposition \ref{key-decomp1p} and Proposition \ref{Pro-28} that
	\begin{equation}
		K_{S\times U,\lambda\gamma}((\zeta,u),(\overline{z},\overline{w}))=K_{S,\lambda}(\zeta,\overline{z})B_{U,\gamma}(u,\overline{w})
		\nonumber
	\end{equation}
	and 
	\begin{equation}
		\begin{split}
			K_{\partial M\times U,\kappa}((\zeta,u),(\overline{z},\overline{w}))=K_{\partial M,\rho}(\zeta,\overline{z})B_{U,\gamma}(u,\overline{w}).
		\end{split}
		\nonumber
	\end{equation}
	Then the inequality  
	\begin{equation}
		K_{S\times U,\lambda\gamma}(z_0,u_0)\geq \bigg(\sum\limits_{1\leq j\leq n}\frac{1}{p_j}\bigg)\pi^{n-1}K_{\partial M\times U,\rho\gamma}(z_0,u_0)
		\label{eq:2418d}
	\end{equation}
	holds if and only if the inequality  
	\begin{equation}
		K_{S,\lambda}(z_0)\geq \bigg(\sum\limits_{1\leq j\leq n}\frac{1}{p_j}\bigg)\pi^{n-1}K_{\partial M,\rho}(z_0)
		\label{ine:2033}
	\end{equation}
	holds with the assumption that $K_{\partial M\times U,\rho\gamma}(z_0,u_0)>0$. Furthermore, the $``="$ in inequality \eqref{eq:2418d}  holds if and only if the $``="$
	in inequality \eqref{ine:2033} holds. Hence, we finish  the proof of Theorem \ref{k-main2} by Theorem \ref{key2-reference}.

	\subsection{Proof of Theorem \ref{Th4.2}}

	We follow the notations in Proposition \ref{p:1211}.
	Using  equality  (see Proposition \ref{pf of 1})
	\begin{equation}
		K_{\partial M\times U,\rho\gamma}^{I^\prime,\hat{h}_0}(w_0)=K_{\partial M,\rho}^{I_1,l_0}(z_0)B^{I_2,b_0}_{U,\gamma}(u_0)
		\nonumber
	\end{equation}
	and  equality (see Proposition \ref{p:1211})
	\begin{equation}
		K_{S\times U,\lambda\gamma}^{I^\prime,\hat{h}_0}(w_0)=K_{S,\lambda}^{I_1,l_0}(z_0)B_{U,\gamma}^{I_2,b_0}(u_0),
		\nonumber
	\end{equation}
	we know the inequality  
	\begin{equation}\label{45-1}
		\left(\prod\limits_{1\leq j\leq n}(\tilde{\beta_j}+1)\right)K_{S\times U,\lambda\gamma}^{I^\prime,\hat{h}_0}(z_0)\geq\left(\sum\limits_{1\leq j\leq n}\frac{\tilde{\beta_j}+1}{p_j}\right)\pi^{n-1}K_{\partial M\times U,\rho\gamma}^{I^\prime,\hat{h}_0}(z_0)
	\end{equation}
	holds if and only if the inequality 
	\begin{equation}\label{45-2}
		\left(\prod\limits_{1\leq j\leq n}(\tilde{\beta_j}+1)\right)K_{S,\lambda}^{I_1,l_0}(z_0)\geq\left(\sum\limits_{1\leq j\leq n}\frac{\tilde{\beta_j}+1}{p_j}\right)\pi^{n-1}K_{\partial M,\rho}^{I_1,l_0}(z_0)
	\end{equation}
	holds  by the assumption that $K_{\partial M\times U,\rho\gamma}^{I^\prime,\hat{h}_0}(w_0)>0$. Furthermore, the $``="$ in inequality (\ref{45-1}) holds if and only if the $``="$ in inequality (\ref{45-2}) holds. So we finish the proof of Theorem \ref{Th4.2} by Theorem \ref{S-M-th1}.

	\vspace{.1in} {\em Acknowledgements}. The authors would like to thank Dr. S.J. Bao and Dr. Z.T. Mi  for checking the manuscript. The first named author was supported by National Key R\&D Program of China 2021YFA1003100 and NSFC-12425101. The third author was supported by China Postdoctoral Science Foundation BX20230402 and 2023M743719.

	\bibliographystyle{references}
	\bibliography{xbib}

\begin{thebibliography}{100}
		
		
		\bibitem{modules-at-boundary}S.J. Bao, $L^2$ extension and effectiveness of strong openness property, PHD thesis, 2022.
		\bibitem{BGY-concavity5}S.J. Bao, Q.A. Guan and Z. Yuan, Concavity property of minimal $L^2$ integrals with Lebesgue measurable gain \uppercase\expandafter{\romannumeral5}--fibrations over open Riemann surfaces, J. Geom. Anal. 33 (2023), no. 6, Paper No. 179, 73 pp.
		
			\bibitem{BGY-concavity6}S.J. Bao, Q.A. Guan and Z. Yuan, Concavity property of minimal $L^2$ integrals with Lebesgue measurable gain \uppercase\expandafter{\romannumeral6}: fibrations over products of open Riemann surfaces, arXiv:2211.05255v2.
			
	\bibitem{BGMY-concavity7}S.J. Bao, Q.A. Guan, Z.T. Mi and Z. Yuan, Concavity property of minimal $L^2$ integrals with Lebesgue measurable gain \uppercase\expandafter{\romannumeral7}--negligible weights. The Bergman kernel and related topics, 1-103,
	Springer Proc. Math. Stat., 447, Springer, Singapore, [2024].
	
	\bibitem{Blocki-12}Z. Blocki, On the Ohsawa-Takegoshi extension theorem, Univ. Iag. Acta Math. 50 (2012), 53-61.
	
		
				\bibitem{Blocki-inv}Z. Blocki, Suita conjecture and the Ohsawa-Takegoshi extension theorem, Invent Math., 193 (2013), 149-158.
		
		\bibitem{demailly2010}J.-P. Demailly, Analytic Methods in Algebraic Geometry, Higher Education Press, Beijing, 2010.
		
		\bibitem{duren}P.L. Duren, Theory of $H^p$ spaces, Academic press, New York and London, 1970.
		\bibitem{chara}O. Forster, Lectures on Riemann surfaces, Grad. Text in Math.,81, Springer-Verlag, New York-Berlin, 1981.
		
			\bibitem{G16}Q.A. Guan,
		A sharp effectiveness result of Demailly's strong openness conjecture,
		Adv. Math. 348 (2019): 51-80.
		
		\bibitem{guan-19saitoh}Q.A. Guan, A proof of Saitoh's conjecture for conjugate Hardy $H^2$ kernels, 
		J. Math. Soc. Japan 71 (2019), no. 4, 1173-1179.
		
		\bibitem{guan}Q.A. Guan, Decreasing equisingular approximations with analytic singularities, J. Geom. Anal. 30 (2020), no.1, 484-492.
		
		
		
		
		
		
		\bibitem{boundary-minimal-concavity}Q.A. Guan, Z.T. Mi and Z. Yuan, Boundary points, minimal $L^2$ integrals and concavity property \uppercase\expandafter{\romannumeral3}---linearity on Riemann surfaces and fibrations over open Riemann surfaces. Acta Math. Sin. (Engl. Ser.) 40 (2024), no. 9, 2091-2152.
		
	    \bibitem{GM}Q.A. Guan and Z.T. Mi, Concavity of minimal $L^2$ integrals related to multiplier ideal
		sheaves, Peking Math. J., 6, 393–457 (2023)
		

		\bibitem{GMY-concavity2}Q.A. Guan, Z.T. Mi and Z. Yuan, Concavity property of minimal $L^2$ integrals with Lebesgue measurable gain \uppercase\expandafter{\romannumeral2}, Adv. Math. 450 (2024), Paper No. 109766, 61 pp.

		
		\bibitem{GY-weight}Q.A. Guan and Z. Yuan, A weighted version of Saitoh's conjecture, accepted by Publ. Res. Inst. Math. Sci., arXiv:2207.10976v2. 
		
		
		\bibitem{GY-Hardy and product}
		Q.A. Guan and Z. Yuan, Hardy space, kernel function and Saitoh's conjecture on products of planar domains, arXiv:2210.14579.
		
		
		
		\bibitem{GY-concavity}Q.A. Guan and Z. Yuan, Concavity property of minimal $L^2$ integrals with Lebesgue measurable gain,  Nagoya Math. J. 252 (2023), 842-905.
		
		\bibitem{GY-conc4}Q.A. Guan and Z. Yuan, Concavity property of minimal $L^2$ integrals with Lebesgue measurable gain \uppercase\expandafter{\romannumeral4}: product of open Riemann surfaces,  Peking Math. J. 7 (2024), no. 1, 91-154.
		
	    \bibitem{guan-zhou CRMATH2012}Q.A. Guan and X.Y. Zhou, Optimal constant problem in the $L^2$ extension theorem. C. R. Math. Acad. Sci. Paris Ser I, 2012,
		350: 753-756.
		
		\bibitem{GZsci}Q.A. Guan and X.Y. Zhou,
		Optimal constant in an $L^2$ extension problem and a proof of a conjecture of Ohsawa, Sci. China Math., 2015, 58(1):35-59.
		
				\bibitem{guan-zhou13ap}Q.A. Guan and X.Y. Zhou, A solution of an $L^{2}$ extension problem with an optimal estimate and applications,
		Ann. of Math. (2) 181 (2015), no. 3, 1139--1208.
		
				\bibitem{GZZCRMATH}Q.A. Guan, X.Y. Zhou and L.F. Zhu, On the Ohsawa-Takegoshi $L^2$ extension theorem and the
		twisted Bochner-Kodaira identity. C. R. Math. Acad. Sci. Paris 349(13–14), 797–800 (2011).
		
		
		
		\bibitem{LXZ}Z. Li, W. Xu and X.Y. Zhou, Xiangyu,
		On Demailly's $L^2$ extension theorem from non-reduced subvarieties,
		Math. Z. 305 (2023), no. 2, Paper No. 23, 22 pp.
		
		
		\bibitem{nehari}Z. Nehari,
		A class of domain functions and some allied extremal problems,
		Trans. Amer. Math. Soc. 69 (1950), 161-178.
		
		   \bibitem{OT87}T. Ohsawa and K. Takegoshi,  On the extension of $L^2$ holomorphic functions, Math. Z. 195 (1987), no. 2, 197-204.
		
		\bibitem{OhsawaObservation}T. Ohsawa,
		Addendum to ``On the Bergman kernel of hyperconvex domain'',
		Nagoya
		Math. J. 137 (1995), 145-148
		
		
		\bibitem{pasternak}Z. Pasternak-Winiarski,
		On weights which admit the reproducing kernel of Bergman type,
		Internat. J. Math. Math. Sci. 15 (1992), no. 1, 1-14.
		
		
		
		\bibitem{rudin}W. Rudin, Real and Complex Analysis (Third Edition). McGraw-Hill Book Co., New York, 1987. xiv+416 pp.
		
		\bibitem{rudin2} W. Rudin, Analytic Functions of Class $H_p$, Transactions of the American Mathematical Society, Jan., 1955, Vol. 78, No. 1 (Jan., 1955), pp. 46-66; published by American Mathematical Society.
		
		
		
		
		
		\bibitem{saitoh}S. Saitoh, Theory of reproducing kernels and its applications, Pitman Research Motes in Mathematics Series, 189, Longman Science $\&$ Technical, Harlow; copublished in the United States with John Wiley $\&$ Sons, Inc., New York, 1988, x+157 pp.
		
				\bibitem{S-O69}L. Sario and K. Oikawa, Capacity functions, Grundl. Math. Wissen. 149, Springer-Verlag, New York, 1969. Mr 0065652. Zbl 0059.06901.
		
		\bibitem{suita72}N. Suita, Capacities and kernels on Riemann surfaces, Arch. Rational Mech. Anal. 46 (1972), 212-217.
		\bibitem{x-z}W. Xu, X.Y. Zhou, Optimal $L^2$ extension of openness type, arxiv:2202.04791.  
		
		\bibitem{yamada}  A. Yamada, Topics related to reproducing kernels, theta functions and the Suita
		conjecture (Japanese), The theory of reproducing kernels and their applications
		(Kyoto, 1998), S\=urikaisekikenky\=usho K\=oky\=uroku, 1067 (1998), 39-47.
		
				\bibitem{ZGZ}L.F. Zhu, Q.A. Guan and X.Y. Zhou, On the Ohsawa–Takegoshi $L^2$
		extension theorem and the Bochner–Kodaira identity with non-smooth twist factor, J. Math.
		Pures Appl. 97 (2012), 579-601
		
	\end{thebibliography}

\end{document}